\documentclass[11pt]{amsart}
\usepackage{amssymb}
\usepackage{xcolor}

\usepackage[
        hypertexnames=false,
        hyperindex,
        pagebackref,
        pdftex,
        breaklinks=true,
        bookmarks=false,
        colorlinks,
        linkcolor=blue,
        citecolor=red,
        urlcolor=red,
]{hyperref}

\numberwithin{equation}{section}
\def\today{\ifcase\month\or
 Jan\or Febr\or  Mar\or  Apr\or May\or Jun\or  Jul\or
 Aug\or  Sep\or  Oct\or Nov\or  Dec\or\fi
 \space\number\day, \number\year}

\newtheorem{theorem}{Theorem}[section]
\newtheorem{lemma}[theorem]{Lemma}
\newtheorem{proposition}[theorem]{Proposition}
\newtheorem{corollary}[theorem]{Corollary}
\newtheorem{conjecture}[theorem]{Conjecture}
\newtheorem{mainconjecture}[theorem]{Main Conjecture}
\newtheorem{definition-lemma}[theorem]{Definition-Lemma}

\theoremstyle{definition}
\newtheorem{definition}[theorem]{\bf Definition}
\newtheorem{example}[theorem]{\bf Example}

\theoremstyle{remark}
\newtheorem{remark}[theorem]{\bf Remark}

\newtheorem{notation}[theorem]{\bf Notation}
\newcommand{\CC}{\mathbb C}

\newcommand{\FF}{\mathbb F}
\newcommand{\GG}{\mathbb G}
\newcommand{\LL}{\mathbb L}
\newcommand{\PP}{\mathbb P}
\newcommand{\QQ}{\mathbb Q}
\newcommand{\QQl}{\overline{\mathbb Q}_{\ell}}
\newcommand{\RR}{\mathbb R}
\newcommand{\VV}{\mathbb V}
\newcommand{\WW}{\mathbb W}
\newcommand{\ZZ}{\mathbb Z}
\newcommand{\bA}{\mathbb A}
\newcommand{\cW}{\mathcal W}
\newcommand{\cO}{\mathcal O}
\newcommand{\cX}{\mathcal X}

\newcommand{\fS}{\mathfrak S}

\newcommand{\GL}{\mathrm{GL}}
\newcommand{\Sym}{\mathrm{Sym}}
\newcommand{\Gal}{\mathrm{Gal}}

\newcommand{\GaSq}{\Gamma[\sqrt{-3}]}
\newcommand{\GaSqOne}{\Gamma_1[\sqrt{-3}]}

\newcommand{\Tr}{\mathrm{Tr}}
\newcommand{\pp}{\mathfrak{p}}

\begin{document}

\title[]{Picard modular forms and \\  the cohomology of local systems \\ on a Picard modular surface} 

\author{Jonas Bergstr\"om}
\address{Matematiska Institutionen, Stockholms Universitet,
SE-106 91 Stockholm, Sweden}
\email{jonasb@math.su.se}

\author{Gerard van der Geer}
\address{Korteweg-de Vries Instituut, Universiteit van Amsterdam, Science \newline Park 904, 1098 XH Amsterdam, The Netherlands.}
\email{g.b.m.vandergeer@uva.nl}

\subjclass[2010]{11F03, 11G18, 11G20, 14H10, 14G35, 14J15 } 

\begin{abstract}
We formulate a detailed conjectural 
Eichler-Shimura type 
formula for the cohomology of local systems on a Picard modular
surface associated to the group of unitary similitudes $\mathrm{GU}(2,1,\QQ(\sqrt{-3}))$. 
The formula is based on counting points over finite fields
on curves of genus three which are cyclic triple covers of the projective line. 
Assuming the conjecture we are able to calculate traces of Hecke
operators on spaces of Picard modular forms. We provide ample evidence
for the conjectural formula. 

Along the way we prove new results on characteristic polynomials of Frobenius acting on the first cohomology group of cyclic triple covers of any genus, dimension formulas for spaces of Picard modular forms and formulas for the numerical Euler characteristics of the local systems.  
\end{abstract}

\maketitle

\setcounter{tocdepth}{1}

\begin{section}{Introduction}
In his 1963 paper \cite{Shimura} Shimura listed a number of arithmetic ball quotients that are rational
and that parametrize 
Jacobians of finite covers of the projective line.
This paper deals with one of these cases and tries to use the link with
the moduli of curves to study the modular forms on one of these ball quotients.
The case at hand is the 2-dimensional ball quotient studied by Picard 
in the 1880s (\cite{P1,P2,P3})
associated to the unitary group in three variables $U(2,1)$ over the field 
$F={\QQ}(\sqrt{-3})$. It parametrizes curves of genus $3$ that are cyclic covers 
of degree $3$ of the projective line. Around 1979 Shintani considered  vector-valued Picard
modular forms on such unitary groups in three variables and gave a criterion
for such modular forms to be Hecke eigenforms, see \cite{Shintani}. 
The volume \cite{L-R} is devoted to showing that the $L$-function of a
Picard modular surface is the product of automorphic $L$-functions. 
But though the literature on automorphic
forms on unitary groups is extensive explicit examples are rare. Holzapfel and Feustel
studied the rings of scalar-valued modular forms on the group in question, \cite{Ho1,Fe}, 
and Finis gave 
in \cite{Fi} a list of Hecke eigenforms of weight $\leq 12$ and gave a few Hecke eigenvalues.

We decided to use the interpretation of this ball quotient as a Hurwitz space 
of cyclic triple covers to investigate the traces of Hecke operators using the cohomology
of local systems on this moduli space. By counting points on the curves in our family
over finite fields we are able to calculate the traces of Frobenius acting on the
local systems associated to the cohomology of these curves. We follow the approach 
initiated in the papers \cite{FvdG,BFvdG} dealing with Siegel modular forms of degree
$2$ and $3$. From the traces of Frobenius on the \'etale cohomology of our local systems
we try to calculate the traces of the Hecke operators on the spaces of vector-valued
Picard modular cusp forms. 

In the case of modular forms on ${\rm SL}(2,{\ZZ})$ the basic formula,
essentially due to Deligne, expresses the compactly supported 
cohomology of a local system $\mathbb V_k$ on the
moduli $\mathcal{A}_1$ of elliptic curves in terms of the motive $S[k+2]$
of modular forms of weight $k+2$ by an Eichler-Shimura type formula
$$
e_c(\mathcal{A}_1, \mathbb V_k)=-S[k+2]-1\, .
$$
The main goal of this paper is to provide a detailed conjectural analogue
of this formula for Picard modular forms for $F$; this formula takes the
form
$$
e_c(\mathcal{X}_{\Gamma_1[\sqrt{-3}]},{\WW}_{\lambda})=
\breve{S}[n(\lambda)]+e_{\rm extr}(\lambda) \, ,
$$
where ${\WW}_{\lambda}$ is a local system, the term $\breve{S}[n(\lambda)]$ is the contribution of genuine Picard
modular forms (that have $3$-dimensional Galois representations) of weight $n(\lambda)$ and $e_{\rm extr}(\lambda)$ 
is a correction term and the analogue of $1$ in the earlier formula, but
rather complicated due to lifts from smaller groups. 

To determine the correction term $e_{\rm extr}(\lambda)$  we need
to substract the contribution of the boundary, the so-called Eisenstein cohomology,
essentially determined by Harder \cite{Ha1}.
Since we are interested in the genuine Picard
modular eigenforms, the forms that are not lifts from smaller groups and come with
$3$-dimensional Galois representations, we also need to subtract the so-called endoscopic
terms. In the case at hand there is a multitude of endoscopic terms and by analyzing
the traces that we computed we were able to make a detailed conjectural 
description of all the endoscopic contributions. Subtracting the Eisenstein contribution
and the conjectured endoscopic terms we find heuristically 
the traces of the Hecke operators on the spaces of genuine Picard modular forms.
These are Picard modular forms on the congruence subgroup of level $\sqrt{-3}$
and thus there is a symmetry group, equal to the symmetric group 
$\mathfrak{S}_4$, acting. More precisely, we get traces of Hecke operators $T_{\nu}$ on spaces of Picard modular forms of a given weight for primes $\nu$ in the ring of integers $\cO_F$ with norm $N(\nu)\equiv_3 1 $ in a equivariant way, that is, taking into account the $\mathfrak{S}_4$-isotypic parts.

The many terms appearing in the correction term
$e_{\rm extr}(\lambda)$ point to the difficulty of getting such
detailed results on Picard modular forms using trace formulas, cf.\ \cite{Keranen}.

In order to do this we need to be able to calculate the characteristic polynomial
of Frobenius on the \'etale cohomology $H^1(C_f,{\QQl})$ of a curve $C_f$ given
by $y^3=f(x)$ over a finite field in an efficient way. 
More precisely,  
we need the characteristic polynomial of Frobenius  on the part of the cohomology 
where the cyclic Galois automorphism $\alpha$ of order $3$ of $C_f$ acts by a given
third root of one. The formula that we give for the characteristic polynomial
for arbitrary genus generalizes a theorem of Gauss dealing with the case of genus~$1$. 

The conjectures in this article are based upon counts of curves together with their  characteristic polynomials of Frobenius for prime powers $q \equiv_3 1$ with $q \leq 67$. Using this data we can compute traces of Frobenius $F_q$ for \emph{any} local systems $\WW_{\lambda}$ (or it is at least computationally very inexpensive). We settled for those of Deligne weight at most $40$. In turn this (assuming our conjectures) gave the traces of $T(\nu)$ for $N(\nu) \leq 67$ on the corresponding spaces of Picard cusp forms.

To make such a computation over a finite field of $q$ elements, we need roughly $q$ operations for each curve (to compute the characteristic polynomial) and there are roughly $q^2$ points, i.e. curves, in our moduli space (since it is a surface). This tells us that it should be possible to make these computations for significantly larger $q$. 

The evidence that we have for the validity of our conjectures is manifold. 
In this paper we calculate the dimensions of the spaces of modular forms
and we calculate the numerical Euler characteristics of the cohomology of our local systems.
To begin with, our procedure for calculating the traces of Hecke operators always yields zero when the dimension of cusp forms is zero.  
Moreover, since we started this project about ten years ago
Fabien Cl\'ery and one of us guided by these heuristic data have constructed explicitly 
vector-valued
modular forms and the results thus obtained in \cite{C-vdG} agree with the conjectures. 
Another striking piece of evidence is provided by congruences of Harder type. Harder 
predicts congruences modulo primes appearing in the critical values 
of $L$-series of Hecke characters and we find quite a number of examples of such
congruences.

\bigskip
We now sketch the contents of this paper. After recalling the modular surfaces
and modular forms and Hecke operators and local systems, 
we treat the BGG complex for our Shimura
variety and use it to describe the Hodge structure on
the cohomology of our local systems. We use it to describe the Eisenstein
cohomology. After that we calculate the dimensions of spaces of cusp forms
on our groups using Riemann-Roch and the holomorphic Lefschetz formula. 
We then discuss the moduli of abelian threefolds with multiplication by $\cO_F$
and the moduli of curves of genus three with a cyclic Galois automorphism
of order three, including degenerations of such curves. We give a theorem describing the characteristic polynomial of Frobenius on the \'etale cohomology on cyclic triple covers
of the projective line. We introduce the Euler characteristics of our local systems
and explain how we carried out the counts on our curves 
over finite fields. We then state the conjectures on the endoscopic terms.
We conclude with many examples of our heuristic results on Picard modular forms
and explain the evidence for the correctness of these results. In particular we
list a number of congruences of Harder type. 

We intend to make our results available on a website in the style of \cite{website}.
\end{section}
\section*{Acknowledgements}
We like to thank Fabien Cl\'ery and G\"unter Harder for many discussions
 and Thomas Peternell for a proof of Prop.\ \ref{peternell}. The second author would like to thank the Stockholm University, YMSC of Tsinghua University and the University of Luxembourg for hospitality during work on this paper.

\tableofcontents
\begin{section}{Picard modular surfaces and Picard modular forms} \label{sec-PMS}
We will first introduce the complex fibres of our spaces 
using their interpretation as quotients of a complex 2-ball by an arithmetic group, 
together with the associated modular forms.  
\begin{subsection}{Picard modular groups and surfaces for the Eisenstein integers} 
\label{PMG}
Let $F$ be the number field ${\QQ}(\sqrt{-3})={\QQ}(\rho)$ with $\rho$ a third root 
of unity, and with ring of integers $\cO_F=\ZZ[\rho]$. 
Consider the vector space $V=F^3$ with non-degenerate hermitian
form $h(z_1,z_2,z_3)=z_1z_2^{\prime}+z_1^{\prime}z_2+z_3z_3^{\prime}$, 
where the prime refers to the Galois automorphism of $F$. 
Let $G$ be the corresponding algebraic group defined over ${\QQ}$  of similitudes of $h$
$$
G(\QQ)=\{ g \in \GL(3,F):  
\forall z \in V, \, h(gz)=\eta(g)h(z) \, \,\mathrm{with} \, \eta(g) \in \QQ \}\, . 
$$ 
We have that $\eta^3(g)=N_{F/{\QQ}}(\det (g))$ (with $N_{F/{\QQ}}$ the norm)
and $\eta$ defines a homomorphism $G \to {\GG}_m$ called the multiplier 
or similitude norm. This group is also denoted by $\mathrm{GU}(2,1,F)$ 
and it is called the group of unitary similitudes of signature $(2,1)$.
The group $ G^0=\ker (\eta)$, also denoted $\mathrm{U}(2,1,F)$, is the ordinary
unitary group and $G^0\cap \ker \det$, also denoted $\mathrm{SU}(2,1,F)$, 
is the special unitary group. 
If we do a base change to $F$ our group $G$ becomes isomorphic 
to $\GL(3,F) \times {\GG}_m$, where the last factor
corresponds to the multiplier~$\eta$. 

We identify the Picard modular group $G^0({\ZZ})$ with 
$\{ g \in {\GL}(3,\cO_F): h(gz)=h(z)\}$ 
and we use the notation
$$ 
\Gamma:=G^0({\ZZ}) \qquad {\rm and} \qquad \Gamma_1:= G^0({\ZZ}) \cap \ker \det \, . 
$$  

The following congruence subgroups play a central role in this paper:
$$ 
\Gamma[\sqrt{-3}]:=\{ g \in \Gamma : g \equiv 1 \, (\bmod \, \sqrt{-3}) \} 
$$
and
$$ 
\Gamma_1[\sqrt{-3}]:=\{ g \in \Gamma_1 : g \equiv 1 \, (\bmod {\sqrt{-3}}) \}\, . 
$$
Note that the center of $\Gamma_1$, $\Gamma_1[\sqrt{-3}]$ and 
$\Gamma[\sqrt{-3}]$ equals $\mu_3$, while the center of $\Gamma$ is $\mu_6$. 
There are isomorphisms (see also below) 
$$ 
\Gamma/\Gamma_1[\sqrt{-3}] \cong \fS_4 \times \mu_6,
\qquad 
\Gamma_1/\Gamma_1[\sqrt{-3}] \cong \fS_4, 
$$ 
where the symmetric group $\fS_4$ occurs as the special orthogonal group of the 
${\FF}_3$-vector space $\cO_F^3 \subset V$ modulo $(\sqrt{-3})$.

Choose an embedding $\sigma: F\to {\CC}$ and identify 
$F \otimes_{\QQ} {\RR}$ with ${\CC}$. With this identification we get a $3$-dimensional 
complex vector space $Z=V_{\RR}=V \otimes_{\QQ} {\RR}$ 
which is a hermitian space of signature $(2,1)$. We let $G(\RR)$ act on $V_{\RR}$ 
as the standard representation.  
The set of complex lines in $Z$ on which $h$ is negative definite 
$$ 
B:=\{ U \subset Z : \dim(U)=1, h_{|U} < 0 \} 
$$
gives us a complex $2$-ball inside ${\PP}(Z)={\PP}^2$. 
If we set $u=z_3/z_2$ and $v=z_1/z_2$ we find an 
explicit description of this ball, 
$$ B=\{ (u,v) \in {\CC}^2: v+\bar{v}+ u \bar{u}  < 0 \}\, . $$
Any element $g=(g_{ij})$ in $G^{+}({\RR}):=\{g \in  G({\RR}): \eta(g)>0 \}$ now acts on $B$ by 
$$ g \cdot (u,v) := \left(
\frac{g_{31}v+g_{32}+g_{33}u}{g_{21}v+g_{22}+g_{23}u} \, , 
\frac{g_{11}v+g_{12}+g_{13}u}{g_{21}v+g_{22}+g_{23}u} 
\right) \, . $$

All finite index subgroups $\Gamma_{*}$ of $\Gamma$ act properly discontinuously on $B$
and the complex quotient surface $\Gamma_{*} \backslash B$, 
denoted by $X_{\Gamma_{*}}$, is called a Picard modular surface. 
Such a quotient is not compact, but can be compactified by adding finitely many 
points, called cusps, which are the orbits of the group action on the set 
$\partial B \cap {\PP}^2(F)$ of rational points. 
This is called the Baily-Borel compactification 
and it will be denoted by $X^*_{\Gamma_{*}}$. 

In the specific cases that we consider, these Picard modular surfaces 
have been studied in 
detail by Holzapfel and Feustel
and most of the statements in this section can be found in \cite{Ho1,Ho2} (see also \cite{Fe}). 

The action of $\Gamma$ on $\partial B \cap {\PP}^2(F)$ has only one orbit
since the class number of $F$ is $1$, see \cite{Zink}. 
The group $\Gamma_1[\sqrt{-3}]$ has four cusps and 
the isomorphism 
$\Gamma/\Gamma_1[\sqrt{-3}]\cong \fS_4 \times \mu_6$ above is given
by $g \mapsto (\sigma(g),\det(g))$, where $\sigma(g)$ is the 
permutation of the four cusps. 

The action of $\Gamma_1[\sqrt{-3}]$ modulo its center on $B$ is not free, but
has three orbits of isolated fixed points.  
These three points become quotient singularities of the form 
$\CC^2/A$ with 
$A=\langle \mathrm{diag}(\rho,\rho^2) \rangle \subset \GL(2,\CC)$ 
on the surface $X_{\Gamma_1[\sqrt{-3}]}$. 
They can be resolved by a configuration of two non-singular rational curves with
self-intersection number $-2$ intersecting transversally in one point. 

The cusps of $X^*_{\Gamma_1[\sqrt{-3}]}$ are singular points and 
each cusp singularity can be resolved by an elliptic curve 
$E= {\CC}/ \sqrt{-3}\, \cO_F$. 
All these elliptic curves have self-intersection number $-3$. 
We number the cusps by $i=1,2,3,4$ and the resolution curves accordingly by $E_i$.

The smooth surface resulting from resolving the quotient and cusp
singularities of $X^*_{\Gamma_1[\sqrt{-3}]}$ is denoted by $Y_{\Gamma_1[\sqrt{-3}]}$. 
Both these spaces admit an action of $\fS_4 \times \mu_6$. 

\smallskip

We can define a modular curve on $Y_{\Gamma_1[\sqrt{-3}]}$ by considering the
embedding of the complex upper half-plane $\mathfrak{H}$ in $B$ by 
$\tau \mapsto (0,\sqrt{-3}\tau)$ with 
corresponding embedding of algebraic groups ${\GL}_2 \to G$ given by 
$$
\left( \begin{matrix} a & b \\ c & d \\ \end{matrix} \right)
\mapsto 
\left( \begin{matrix} a & \sqrt{-3} b & 0 \\ c/\sqrt{-3} & d & 0 \\
0 & 0 & ad-bc  \\ \end{matrix} \right).
$$
This defines an algebraic curve on $X_{\Gamma_1[\sqrt{-3}]}$ isomorphic to 
$\Gamma_0(3)\backslash \mathfrak{H}$. Its closure in $X^*_{\Gamma_1[\sqrt{-3}]}$ 
passes through two cusps. Applying the action of $\fS_4$ we get  
a curve $D_{ij}$ on $X^*_{\Gamma_1[\sqrt{-3}]}$ passing through the cusps 
$i$ and $j$ for each $1 \leq i < j \leq 4$. 

After blow-up, we get the following configuration of curves on 
$Y_{\Gamma_1[\sqrt{-3}]}$.

(i) Four elliptic curves $E_i$ $1 \leq i \leq 4$ with $E_i^2=-3$.

(ii) Six rational curves $D_{ij}$ intersecting $E_i$ and $E_j$ 
transversally. 

(iii) Three pairs of rational curves $R_{ij},R_{kl}$ with $\{i,j,k,l\}=
\{1,2,3,4\}$  resolving the quotient singularities, with 
$D_{ij} R_{ij}=1$ and $D_{ij}R_{k,l}=0$ if $\{k,l\}
\neq \{i,j\}$.

The surface $X^*_{\Gamma_1[\sqrt{-3}]}$ can be identified with the 
$3$-fold cover of the hyperplane $x_1+x_2+x_3+x_4=0$ in ${\PP}^3$ 
given by
\begin{equation}  \label{pmf-zeta}
\zeta^3= \prod_{1\leq i<j\leq 4} (x_i-x_j)\, 
\end{equation} 
with the action of 
$(\sigma,\pm \rho) \in \fS_4 \times \mu_6 \cong \Gamma/\Gamma_1[\sqrt{-3}] $ by 
$x_i \mapsto  x_{\sigma(i)}$ and 
$\zeta \mapsto \rho \zeta$.
The four cusps correspond to the points with $\zeta=0$ and 
$(x_1,x_2,x_3,x_4)=(1:1:1:-3)$, $(1:1:-3:1)$, $(1:-3:1:1)$, $(-3:1:1:1)$, 
and the remaining three singularities to points with $\zeta=0$ and 
$(x_1,x_2,x_3,x_4)=(1:1:-1:-1)$, $(1:-1:1:-1)$, $(1:1:1:-1)$.  
The curves $x_i=x_j$ give the images of the curves $D_{ij}$.

Taking the quotient by $\Gamma[\sqrt{-3}]/\Gamma_1[\sqrt{-3}] \cong \mu_3$ gives 
an identification of $X^*_{\Gamma[\sqrt{-3}]}$ with $\PP^2$ as the hyperplane above in $\PP^3$ 
with action by $\Gamma/\Gamma[\sqrt{-3}]\cong \fS_4 \times \mu_2$.
\end{subsection}
\begin{subsection}{Picard modular forms}\label{PMF}
Fix a point $x_0 \in B$ and let $K$ be the stabilizer of $x_0$ under the action of $G(\RR)$. 
Recall that an automorphy factor is a mapping $J: G(\RR) \times B \to \GL(W)$, 
with $W$ a complex finite-dimensional vector space, 
fulfiling a cocycle condition and where $J$ restricted to $K \times  x_0$ is a representation. 

The action of the group $G^{+}({\RR})$  
on $B$ determines two factors of automorphy 
for $g=(g_{ij}) \in G^{+}({\RR})$ and $(u,v)\in B$
given by 
$$
j_1(g,u,v):= g_{21}v+g_{22}+g_{23}u
$$
and
$$
j_2(g,u,v):=\det(g)^{-1} \left( 
\begin{matrix} 
G_{32}u+G_{33} &  G_{32}v+G_{31} \\  
G_{12}u+G_{13} & G_{12}v+G_{11}\\
\end{matrix} 
\right)
$$
with $G_{ij}$ the minor of $g_{ij}$, see \cite{Shimura-arithmetic}. 
Note that 
$$\det(j_2(g,u,v))=j_1(g,u,v) \cdot (\det(g))^{-1}$$ 
and that the transpose of the 
Jacobian of the action of $G^{+}({\RR})$ on $B$ is given by 
$$
j_1(g,u,v)^{-1} j_2(g,u,v)^{-1}  \,.
$$  
Let $(j,k)$ be a pair of integers with $j\geq 0$. 
Define a slash operator on functions
$f: B \to {\rm Sym}^j({\CC^2})$ for $g\in G^{+}({\RR})$ via
$$
(f_{|j,k}g)(u,v) := j_1(g,u,v)^{-k} {\rm Sym}^j(j_2(g,u,v)^{-1})f(g\cdot (u,v))\, .
$$

For any finite index subgroup 
$\Gamma_{*}$ in $\Gamma$ and character $\chi$ of 
finite order on $\Gamma_{*}$, we define the vector space of modular forms
of weight $(j,k)$ and character $\chi$ on $\Gamma_{*}$ by
$$ M_{j,k}(\Gamma_{*},\chi):= \{ f: B \to {\rm Sym}^j({\CC^2}): 
f \, \text{holomorphic,} \forall g \in \Gamma_{*}, f_{|j,k}g=\chi(g) f \}. $$ 
The space of cusp forms of weight $(j,k)$ and character $\chi$ on $\Gamma_{*}$ 
will be denoted by $S_{j,k}(\Gamma_{*},\chi)$. Only the characters
that are a power of $\det(g)$ will be considered in this paper. 
We will just write $M_{j,k}(\Gamma_{*})$, or $S_{j,k}(\Gamma_{*})$, 
when $\chi$ is trivial. For more details we refer
to \cite{C-vdG} and the references therein. 

Alternatively, we could define another automorphy factor 
$$j_3(g,u,v):=\det(g)$$
and a new slash operator 
$$ 
(f_{|j,k,l}g)(u,v) := 
j_1(g,u,v)^{-k} {\rm Sym}^j(j_2(g,u,v)^{-1})j_3(g,u,v)^{-l} f(g\cdot (u,v))\, 
$$
with corresponding spaces of modular forms 
$M_{j,k,l}(\Gamma_{*}) = M_{j,k}(\Gamma_{*},\det^l)$ 
and cusp forms $S_{j,k,l}(\Gamma_{*}) = S_{j,k}(\Gamma_{*},\det^l)$. 

\smallskip

The group $\Gamma[\sqrt{-3}]/\Gamma_1[\sqrt{-3}] \cong \mu_3$ (generated by 
${\rm diag}(1,1,\rho)$) acts
on $M_{j,k}(\Gamma_1[\sqrt{-3}])$. 
This action is reflected in a decomposition of $M_{j,k}(\Gamma_1[\sqrt{-3}])$
as a sum of spaces of modular forms with character 
$$
M_{j,k}(\Gamma[\sqrt{-3}])\oplus
M_{j,k}(\Gamma[\sqrt{-3]},\det)\oplus M_{j,k}(\Gamma[\sqrt{-3}], {\det}^{2}).
$$
\begin{remark}
Note that $M_{j,k}(\Gamma_1[\sqrt{-3}])=\{0\}$ if $j\not\equiv_3 k$. 
Moreover, we have 
$M_{j,k}(\Gamma[\sqrt{-3}],\det^{\ell})=S_{j,k}(\Gamma[\sqrt{-3}],
\det^{\ell})$ if $\ell \not\equiv_3 j$,
see \cite[Prop. 5.1]{C-vdG}. 
\end{remark}

The rings of scalar-valued modular forms on $\Gamma[\sqrt{-3}]$ and 
$\Gamma_1[\sqrt{-3}]$ were determined by Holzapfel and Feustel 
(\cite{Ho1, Fe}).

\begin{proposition}\label{scalar-rings}
We have 
$$
\oplus_{k=0}^{\infty} M_{0,3k}(\Gamma[\sqrt{-3}])=
{\CC}[x_1,x_2,x_3,x_4]/(x_1+x_2+x_3+x_4),
$$ 
where $x_1,x_2,x_3,x_4$ are modular forms of weight $3$, and the group $\fS_4$ acts on these modular forms by $\sigma: x_i \mapsto \mathrm{sgn}(\sigma) x_{\sigma(i)}$.
The ring $\oplus_{k=0}^{\infty} M_{0,3k}(\Gamma_1[\sqrt{-3}])$ is the ring extension of 
$\oplus_{k=0}^{\infty} M_{0,3k}(\Gamma[\sqrt{-3}])$ 
by the element 
$$\zeta \in M_{0,6}(\Gamma[\sqrt{-3}], \det)$$ 
satisfying the equation~\eqref{pmf-zeta} 
and where $\fS_4$ acts on $\zeta$ by the sign representation. 
\end{proposition}

Since $X^*_{\Gamma_1[\sqrt{-3}]} \cong 
\mathrm{Proj}(\oplus_{k=0}^{\infty} M_{0,3k}(\Gamma_1[\sqrt{-3}]))$ 
we retrieve the identification of $X^*_{\Gamma_1[\sqrt{-3}]}$ 
with ${\PP}^2$ given at the end of Section~\ref{PMG}, see \cite{Ho1}. We point out that the 
$x_i$ have $F$-integral Fourier-Jacobi expansions, see \cite{Fi,C-vdG}.

\begin{proposition}\label{vanishing} If $k<0$ and $j \geq 0$ then 
$\dim M_{j,k}(\Gamma_1[\sqrt{-3}])=0$.
\end{proposition}
\begin{proof} 
If one develops a vector-valued modular form $f \in M_{j,k}(\Gamma_1[\sqrt{-3}])$ along 
the modular curve given in Section~\ref{PMG} then the first component of the vector 
$f(0,\sqrt{-3}\tau)$ is a modular form on $\Gamma_1(3) \subset {\rm SL}_2(\ZZ)$ of weight $k$, 
see \cite[Prop. 8.4]{C-vdG}, and hence it is zero. 
The same thing happens for other modular curves. 
Curves parametrized by $\mathfrak{H} \to B$, $\tau \mapsto (a,\sqrt{-3}\tau)$
with $a \in F$ yield modular curves and restricting modular forms
$M_{j,k}(\Gamma_1[\sqrt{-3}])$
leads to modular forms of negative weight on congruence subgroups of
${\rm SL}(2,{\ZZ})$, and these are zero. Since these curves lie dense, 
we see that 
the first component of $f(u,\sqrt{-3}\tau)$ is zero for all $u,\tau$.
Applying the invariance of modular forms under the unipotent radical of a parabolic subgroup, 
see \cite[Eq. 4]{C-vdG}, we find that all components vanish.  
\end{proof} 

\begin{remark} The proof can be easily adapted to show that the proposition
 holds for any finite index subgroup $\Gamma_{*}$ of $\Gamma$. 
\end{remark}
\end{subsection}
\begin{subsection}{Hecke Operators} \label{sec-heckeoperator}
The Hecke rings for the arithmetic group $\Gamma_1$ and $\Gamma_1[\sqrt{-3}]$
were studied by Shintani \cite{Shintani} and Finis \cite{Fi}. 
Outside the prime $3$ these Hecke
rings are the same and they are generated by elements $T(\nu)$, $T(\nu,\nu)$
for elements $\nu \in \mathcal{O}_F$ with norm $N(\nu)=p$ for primes $p\equiv_3 1$ and elements $T(p)$ and $T(p,p)$ for primes 
$p\equiv_3 2$.
For $\Gamma_1$ we also have elements 
$T(\sqrt{-3})$ and $T(\sqrt{-3},\sqrt{-3})$. 
We refer to Finis' paper and to \cite{C-vdG} for a description of the Hecke
rings and the action on modular forms. Note that for a Hecke eigenform with
eigenvalues $\lambda_{\nu}$ we have $\lambda_{\bar{\nu}}=\bar{\lambda}_{\nu}$.

We define for $\nu\in \mathcal{O}_F$
with norm  a prime $p\equiv_3 1$ and given weight $(j,k)$ 
the polynomial
$$
Q_{\nu}^{j,k}(X,\lambda)= 1-\lambda \, X + \nu^{j+1}\bar{\nu}^{k-2} 
\bar{\lambda} \, X^2 - \nu^{2j+k}\bar{\nu}^{j+2k-3} \, X^3
$$
and for a prime $p\equiv_3 2 $ we  define $Q_{-p}^{j,k}(X,\lambda)$ by
$$
(1-(\lambda -(p-1)(-p)^{j+k-3})\, X + p^{2j+2k-2}X^2)
(1-(-p)^{j+k-1}X)
$$
The local factor of the $L$-function of a Picard modular form $f$ 
of weight $(j,k)$ that is an eigenform for the Hecke algebra with eigenvalue
$\lambda_{\nu}(f)$ for $T(\nu)$ with $N(\nu)=p\equiv_3 1$ 
equals the inverse of 
$Q_{\nu}^{j,k}(N(\nu)^{-s},\lambda_{\nu}(f))$, while for a prime 
$p\equiv_3 2$ the local factor is the inverse of
$Q_{-p}^{j,k}(N(p)^{-s},\lambda_{-p}(f))$. 
\end{subsection}
\begin{subsection}{Modular forms as sections of automorphic vector bundles}\label{vectorbundles}
In this section we will realize our modular forms as sections of some vector bundles.
We will use the interpretation of $B$ as the Grassmann variety of negative lines  
in $Z=V_{\RR}$. 
This interpretation provides $B$ with two vector bundles $T$ and $S$ 
fitting in an exact sequence
$$0 \to T \to B \times Z \to S \to 0, $$
where $T$ is the tautological line bundle that associates to a point of $B$
the negative line it represents, and $S$ is the tautological quotient 
bundle of rank~$2$. 
The tangent bundle to $B$ is given by ${\rm Hom}(T,S)=S\otimes T^{-1}$, 
so the cotangent bundle $\Omega_B^1$ equals $S^{\vee} \otimes T$. 

Let $\Gamma'$ be a freely acting finite index subgroup of $\Gamma$. 
Put $X=X_{\Gamma'}=\Gamma' \backslash B$. Let $X^*=X^*_{\Gamma'}$ denote the Baily-Borel compactification of $X_{\Gamma'}$ and $Y=Y_{\Gamma'}$ its minimal resolution. The 
cotangent bundle $\Omega^1_{X}$ is then equal to  
the quotient $\Gamma' \backslash (S^{\vee} \otimes T)$. 
Let $D$ denote the resolution divisor
on $Y$ of the cusps of $X^*$.   
Mumford's canonical extension of $\Omega^1_{X}$ 
extends to $\Omega^1_{Y}(\log D)$ on $Y$, see \cite[Prop. 3.4]{Mu1}. 
The bundle on $X$ defined  
by $\Gamma' \backslash S^{\vee}$ will be denoted by $U$ and the bundle 
$\Gamma' \backslash T$ by $L$. With abuse of 
notation their canonical extensions to $Y$ will be denoted with the same letters. 

If we choose a base point $x_0=(u_0,v_0) \in B$, then $B$ can be identified with 
$G^0(\RR)/K^0$, 
where the maximal compact subgroup $K^0 \cong U(2) \times U(1)$ is the stabilizer of $x_0$. 
Choosing instead a line in $Z$ for which $h$ is negative definite, then if $K$ is the stabilizer in $G$ of the line spanned 
by $x_0$ we find that $K \cong \CC^* \cdot K^0$ and that $B \cong G(\RR)/K$. 

For each finite-dimensional complex representation
$\lambda: K \rightarrow \GL(W)$ we get an automorphic vector bundle $\mathcal{W}_{\lambda}$
by $G(\RR) \times_K W \rightarrow B$ where $(g,w)$ is identified with 
$(gk,\lambda(k)w)$ for all 
$g \in G(\RR)$, $w \in W$ and $k \in K$. The group $G(\RR)$ acts naturally on the left 
by $g'.(g,w)=(g'g,v)$. After taking the quotient by $\Gamma'$ we get an 
automorphic vector bundle $\cW_{\lambda}$ on $X$.
These vector bundles extend canonically to $Y$, and will be denoted with 
the same letters. 

Recall from Section~\ref{PMF} that an automorphy factor $J$, 
when restricted to $K \times  x_0$, 
is a representation $\lambda$, 
and so $J$ determines a trivialization of $\cW_{\lambda}$ by 
$\Phi_J: G(\RR) \times_K W \rightarrow B \times W$ with $(g,w) \mapsto (gx_0,J(g,x_0)w)$. 
The global sections of the bundle $\cW_{\lambda}$ on $Y$ can be identified with  
the modular forms transforming with the automorphy factor $J$. 
It follows directly from our definitions that the bundle $L$ comes with a trivialization 
given by 
the automorphy factor $j_1$. As mentioned in Section~\ref{PMG}, the 
Jacobian of the action of $G(\RR)$ of $B$ is by $j_1^{-1} j_2^{-1} $. Taking the dual 
and tensoring with $L^{-1}$ gives that $U$ has a trivialization given by $j_2$. Finally, 
we find that the bundle $R:=\det(U)^{-1} \otimes L$ 
corresponds to the automorphy factor $j_3$. In summary, the relations between 
vector bundles and automorphy factors are as follows 
$$L \leftrightarrow j_1, \; \; U \leftrightarrow j_2,  \; \;  R \leftrightarrow j_3,$$ 
and moreover 
$$
\det(U)=L \otimes R^{-1}, \quad U^{\vee}\cong U \otimes L^{-1} \otimes R \, .
$$

\begin{definition}
For integers $j,k,l$ with $j \geq 0$ we put 
$$\cW_{j,k,l}:={\rm Sym}^j(U) \otimes L^{k}\otimes R^{l}.$$
\end{definition}

We then get the following interpretation
of modular forms. 

\begin{proposition} For integers $j,k,l$ with $j \geq 0$, we have that
\begin{equation}\label{mfinterpretation}
M_{j,k,l}(\Gamma')= 
H^0(Y_{\Gamma'}, \cW_{j,k,l}), 
\end{equation}
and
\begin{equation}\label{cfinterpretation}
S_{j,k,l}(\Gamma')= H^0(Y_{\Gamma'}, \cW_{j,k,l}\otimes\cO(-D) ). 
\end{equation}
\end{proposition}

Furthermore, we have that $\Omega_X^1(\log D)\cong U\otimes L$.
The canonical extension of the canonical bundle $\Omega^2_{X}$ 
equals on the one hand $\Omega^2_{Y}(D)$ and on the other hand, 
by a local calculation or 
since the canonical extension commutes with exterior products 
(see \cite[p. 225]{F-C}), it equals $\det(\Omega^1_{X}(\log D))$.
 We thus find
\begin{equation}
\Omega^1_{Y}(\log D) \cong U\otimes L, \quad
\Omega^2_Y(D) \cong L^3 \otimes R^{-1} \, .
\end{equation}

The subgroups of $\Gamma$ that we are mainly considering 
do not act freely on $B$. An automorphic vector bundle $\cW_{\lambda}$ on $B$ will become a 
vector bundle on $X$ precisely if the stabilizer of 
any point $x \in B$ acts trivially on the fibre $(\cW_{\lambda})_x$. 
Let us consider the group $\Gamma_1[\sqrt{-3}]$. The center of $\Gamma_1[\sqrt{-3}]$ 
is generated by $\rho \cdot 1_3$, and since $j_1(\rho \cdot 1_3,u,v)=\rho$, 
$j_2(\rho\cdot 1_3,u,v)=\rho^2\cdot 1_{2}$, $j_3(\rho\cdot 1_3,u,v)=1$, 
a necessary condition for $ \cW_{j,k,l}$ to be vector bundle on 
$X_{\Gamma_1[\sqrt{-3}]}$ is that $j\equiv_3 k$.
The stabilizer of one the three singular points $x$ in $B$ is generated by a 
matrix $g_x$ with eigenvalues $1,\rho,\rho^2$ 
such that $j_1(g_x,x)=\rho$ and
$j_2(g_x,x)=\mathrm{diag}(\rho,1)$ and 
$j_3(g_x,x)=1$ so $\cW_{\lambda}$ is only a vector bundle on $X_{\Gamma_1[\sqrt{-3}]}$  
if $j=0$ and $k \equiv_3 0$. 

To treat the cases of non-freely acting groups we can replace the group by a
freely acting finite index normal subgroup and then take invariants.
By the Koecher principle, these forms extend to holomorphic sections
of ${\rm Sym}^j(U)\otimes L^k\otimes R^l$ over the cusp resolutions. 
Also the quotient singularities pose no problem.
Therefore, the identities of (\ref{mfinterpretation}), (\ref{cfinterpretation}) 
still hold on $\Gamma[\sqrt{-3}]$ and $\Gamma_1[\sqrt{-3}]$.  

\begin{remark} 
Proposition~\ref{vanishing} together with \eqref{cfinterpretation} shows that
$$
H^0(Y_{\Gamma[\sqrt{-3}]},\cW_{j,k,l}\otimes\cO(-D))= 0
$$
for any $j,l\geq 0$ and $k<0$. This argument is easily generalizable 
to other Picard modular surfaces and other arithmetic subgroups. Compare with the 
vanishing results of \cite{Mi,MSt-Y-Z}.  
\end{remark}

\end{subsection}
\end{section}
\begin{section}{Cohomology of complex local systems}
In this section we introduce the local systems of interest to us and we use the BGG-complex 
to find information about the cohomology of these local systems. 

\begin{subsection}{Local systems and roots} \label{rootloc}
A vector bundle $\cW_{\lambda}$, as defined in the previous section, is a local system, i.e. 
locally constant, if the representation $\lambda$ 
is a restriction of a representation of $G(\RR)$. 

We have the local system $\WW$ on $X_{\Gamma_{*}}$ 
coming from the dual of the standard representation of $G(\RR)$ acting on $V_{\RR}$, 
as in the beginning of Section~\ref{vectorbundles}. 
In terms of an automorphy factor $J$, the local systems are the ones 
for which $J(g,x)$ are independent of $x$. 
The bundle $R$ is thus a local system and we find that it is isomorphic 
to $\wedge^3 \WW$. 
It is constant for any $\Gamma_{*}$ that is a subgroup of $\Gamma_1$ 
and since $\Gamma_1/\Gamma \cong \mu_6$ we see that 
$(\wedge^3 \WW)^6 \cong R^6$ is constant for subgroups of $\Gamma$. 

The representation theory of $G$ and $K$ over the complex numbers in terms of roots and 
highest weights will be important for the construction of the BGG-complex which we will 
use in Section~\ref{sec-bgg}.

Note first that the base change of $G$ to $\CC$ is isomorphic to 
${\rm GL}(3,\CC) \times {\GG}_m$, where the last factor corresponds 
to the multiplier $\eta$. 

Let $Q$ be a maximal parabolic subgroup and $Q= M \ltimes {\mathcal U}$ 
a Levi decomposition with ${\mathcal U}$ the unipotent radical of $Q$.  
The complexification of $K$, namely $\GL(2,\CC) \times \GL(1,\CC)$, is conjugate 
to that of $M$. 

Let $T$ be the maximal torus of $G$ of diagonal matrices $g=
{\rm diag}(a_1,a_2,a_3)$ with $a_i \in F^*$ satisfying
$a_1a_2'=a_1'a_2=a_3a_3'=\eta(g) \in {\QQ}^*$.
We have the characters $L_i: g \mapsto a_i$.
The roots 
are $\pm (L_1-L_2), \pm (L_1-L_3), \pm (L_2-L_3)$.
We can view
$\alpha=L_1-L_2$ and  $\beta=L_2-L_3$
as two simple roots and 
a system of fundamental weights is $\gamma_1=L_1$ and  
$\gamma_2 =L_1+L_2$ and $\gamma_3 =L_1+L_2+L_3$.  

Then $\Phi_G^{+}=\{\alpha,\beta, \alpha+\beta\}$ 
is a system of positive roots, 
occurring in the adjoint action on the unipotent radical $\mathcal{U}$ 
and we can take $\Phi_M^{+}=\{\alpha \}$. 

The Weyl group $W_G$ of $G$ is generated by the reflections 
$s_{\alpha}$ and $s_{\beta}$ which act 
on the fundamental weights by  
$$
s_{\alpha}:
\begin{cases}
\gamma_{1} \mapsto \gamma_2-\gamma_1 \cr
\gamma_2 \mapsto \gamma_2\cr 
\gamma_3 \mapsto \gamma_3, \cr 
\end{cases}
\qquad
s_{\beta}:
\begin{cases}
\gamma_1 \mapsto \gamma_1, \cr 
\gamma_2 \mapsto \gamma_1-\gamma_2+\gamma_3 \cr
\gamma_3 \mapsto \gamma_3.\cr
\end{cases}
$$ 
Then put 
$\theta= s_{\alpha}s_{\beta}s_{\alpha}=s_{\beta}s_{\alpha}s_{\beta}$ 
with $\theta(\alpha)=-\beta$ and $\theta(\beta)=-\alpha$ and
$$
\theta(n_1\gamma_1+n_2\gamma_2+n_3\gamma_3)=-n_2\gamma_1-n_1\gamma_2+
(n_1+n_2+n_3)\gamma_3\, .
$$

The Weyl group $W_M$  equals 
$\langle s_{\alpha} \rangle$. Define
$$
W^M:=\{ w \in W_G: \Phi_M^{+} \subset w(\Phi_G^{+}) \}.
$$
We find that $W^M=\{ 1, s_{\beta},s_{\beta} s_{\alpha}\}$.
Put $\delta:=\alpha+\beta=\gamma_1+\gamma_2-\gamma_3$, 
which is half the sum of the positive roots.
We have an involution on $W^M$ given by $w \mapsto s_{\alpha} w \theta$
with $s_{\alpha}$ and $\theta$ the elements of longest length in
$W_M$ and $W_G$.

For an element $w$ in the Weyl group and a weight $\lambda$  
we define an action by
$$ w \ast \lambda := w(\lambda+\delta) -\delta. $$ 

For each weight $\lambda=n_1 \gamma_1+n_2 \gamma_2+n_3 \gamma_3$ we get an 
irreducible finite-dimensional complex representation of $G$ with highest weight $\lambda$. The corresponding local system on $X_{\Gamma_{*}}$ will be denoted by $\WW_{\lambda}=\WW_{n_1,n_2,n_3}$ and it can 
be found inside 
$$\Sym^{n_1} (\WW) \otimes \Sym^{n_2} (\wedge^2 \WW) \otimes \Sym^{n_3} (\wedge^3 \WW).$$
For the same weight $\lambda$ we get an irreducible finite-dimensional complex representation of $K$. To identify this representation we consider the automorphy factors for a diagonal matrix $g$, and we then find that $j_1$ corresponds to $\gamma_1$, $j_3$ corresponds to $\gamma_3$ and $j_2$ to $\gamma_2-\gamma_3$. 
The vector bundle correponding to $\lambda$ is thus $\cW_{n_2,n_1,n_2+n_3}$ in the notation of Section~\ref{vectorbundles}. 
Note that the cotangent bundle $U \otimes L$ has highest weight $\delta=\alpha+\beta$. 
A weight $\lambda$ will be called regular if $n_1>0$ and $n_2>0$. 

We will return to these local systems and vector bundles in terms of the moduli interpretation of our Picard modular surfaces in Section~\ref{sec-euler}. 
\end{subsection}
\begin{subsection}{The BGG-complex}\label{sec-bgg}
We will here apply the methods of Faltings and Chai \cite[Ch.\ VI]{F-C} to our 
situation, especially the theory exposed on pages 228 to 237. For a given local
system one obtains a complex of vector bundles, called the dual BGG complex,
with differentials that are differential operators between vector bundles.
It is obtained as a direct summand of the de Rham complex of the local
system. 

Let $\Gamma_*$ be a finite index subgroup of $\Gamma$, and let $\Gamma'$ be a
normal finite index subgroup of $\Gamma_*$ that acts
freely on $B$. As above, consider the surface $Y_{\Gamma'}$ that is the minimal
resolution of the cusp singularities of the Baily-Borel compactification of 
$X_{\Gamma'}=\Gamma'\backslash B$ with resolution divisor $D$. Let the inclusion
of $X_{\Gamma'}$ in $Y_{\Gamma'}$ be denoted by $j$.

The BGG-complex that the methods of \cite{F-C} give for a local system
${\WW}_{\lambda}$ on $X_{\Gamma'}$ is $K_{\lambda}^{\bullet}$ with
$$
K_{\lambda}^q=\oplus_{w \in W^{M},\ell(w)=q} \mathcal{W}_{w*\lambda}^{\vee}\, .
$$
The vector bundles $\mathcal{W}_{\xi}$ extend canonically 
over the cusp resolutions and the differential operators do as well.
We denote the resulting complex on $Y_{\Gamma'}$ by 
$\overline{K}_{\lambda}^{\bullet}$. 

\begin{proposition} \label{prop-bggres} 
  Let 
  $\lambda=n_1\gamma_1+n_2\gamma_2+n_3\gamma_3$.
The dual BGG-complex 
\begin{multline*}
{\rm Sym}^{n_1}(U) \otimes L^{-n_1-n_2}\otimes R^{n_1+n_2+n_3}
\to {\rm Sym}^{n_1+n_2+1}(U) \otimes L^{-n_1+1}\otimes R^{n_1+n_2+n_3} \\
\to {\rm Sym}^{n_2}(U) \otimes L^{n_1+3} \otimes  R^{n_2+n_3-1} \to 0 
\end{multline*}
is quasi-isomorphic to $Rj_{*}{\WW}_{\lambda}$. Similarly, $Rj_!{\WW}_{\lambda}$
is quasi-isomorphic to the dual BGG-complex tensored with $\mathcal{O}(-D)$.
\end{proposition}
\begin{proof}
This follows as in \cite[Proposition 5.4]{F-C}. 
Taking the dual corresponds to applying $-\theta$ to $\lambda$
and then $W_M$ changes to $W_{M'}=\langle s_{\beta}\rangle$.
We thus consider the triples representing $w*(-\theta(\lambda))$
for $w\in W^{M'}$. These are
$(n_2,n_1,-n_1-n_2-n_3)$, $(-n_2-2,n_1+n_2+1,-n_1-n_2-n_3)$
and $(n_1-n_2-3,n_2,-n_2-n_3+1)$.
Taking the duals of the resulting $\mathcal{W}_{\mu}$ gives the result.
\end{proof}

For the case $n_1=n_2=n_3=0$ we get the usual logarithmic 
de Rham complex $\Omega^{\bullet}(\log D)$.

\begin{remark} \label{rmk-duality}
The dual of ${\WW}_{\lambda}$ with
$\lambda=n_1\gamma_1+n_2\gamma_2+n_3\gamma_3$ corresponds to
$-\theta(\lambda)=n_2\gamma_1+n_1\gamma_2+(-n_1-n_2-n_3)\gamma_3$. 
The Serre dual
of $\mathcal{W}_{\mu}={\rm Sym}^a(U) \otimes L^b \otimes R^c$ is
$$
{\rm Sym}^a(U)\otimes L^{-a-b+3}\otimes R^{a-c-1}\otimes \mathcal{O}(-D)\, .
$$
The Serre duals of the terms occurring in the BGG complex for 
$j_{*}{\WW}_{\lambda}$ occur in reverse order in the BGG complex
of $j_!{\WW}_{-\theta(\lambda)}$.
\end{remark}

Put $|\lambda|:=n_1+2n_2+3n_3$.
From \cite[Theorem 5.5]{F-C} it follows that  
$H^i(X_{\Gamma'},{\WW}_{\lambda})$ has a Hodge
structure of weight $\geq i +|\lambda|$ and the compactly
supported cohomology $H^i_{c}(X_{\Gamma'},{\WW}_{\lambda})$ 
has a Hodge structure of weight $\leq i +|\lambda|$. We also get the following. 

\begin{proposition} \label{prop-bggfilt} 
For $\lambda=n_1\gamma_1+n_2\gamma_2+n_3\gamma_3$, 
we have a Hodge filtration on  
$H^i_c(X_{\Gamma'},{\WW}_{\lambda})$ 
equal to  
$$
F^{n_1+n_2+n_3+2} \subset F^{n_2+n_3+1} \subset F^{n_3}$$
and the graded pieces can be identified with
$$
\begin{aligned}
&H^{i-2}(Y_{\Gamma'}, {\rm Sym}^{n_{2}}(U)\otimes L^{n_{1}}\otimes  R^{n_2+n_3} \otimes \Omega^2_Y), 
\\
&H^{i-1}(Y_{\Gamma'}, {\rm Sym}^{n_{1}+n_{2}}(U)\otimes L^{-n_1} \otimes R^{n_1+n_2+n_3}
\otimes \Omega^1_Y), \\
&H^{i}(Y_{\Gamma'}, {\rm Sym}^{n_{1}}(U)\otimes L^{-n_{1}-n_{2}} \otimes R^{n_1+n_2+n_3} \otimes 
\mathcal{O}(-D))\, .\\
\end{aligned}
$$
\end{proposition}

For $i=2$ we see that the first step of the Hodge filtration is isomorphic to the space of cusp forms 
$$S_{n_2,n_1+3,n_2+n_3-1}(\Gamma').$$
But note that we want this space of cusp forms to have Hodge weight 
$n_1+n_2+2$ and the discrepancy is due to a ``twisting" that will be described further in Section~\ref{sec-normalized}. 

We define the inner cohomology $H^i_{!}(X_{\Gamma'},{\WW}_{\lambda})$
as the image under the natural map 
of $H^i_{c}(X_{\Gamma'},{\WW}_{\lambda})$ in
$H^i(X_{\Gamma'},{\WW}_{\lambda})$. It follows that the inner cohomology will have a pure Hodge structure of weight $|\lambda|+i$. 
By results of Ragunathan, Li-Schwermer and Saper \cite{Rag,L-S,Sap} for 
{\sl regular} $\lambda$ (that is $n_1>0$ and $n_2>0$) we 
know that 
$H^i_{!}(X_{\Gamma'},{\WW}_{\lambda})\neq 0$ implies $i=2$.

Note that by taking invariants under $\Gamma_{*}/\Gamma^{\prime}$, all results in this section hold also for $\Gamma^{\prime}$ replaced by $\Gamma_{*}$. In particular they hold for $\Gamma_{*}$ equal to $\Gamma[\sqrt{-3}]$ or $\Gamma_1[\sqrt{-3}]$.
\end{subsection}
\begin{subsection}{The neighborhood of a cusp}\label{cuspnhd}
  For a freely acting finite index normal subgroup
  $\Gamma'$ of $\Gamma_1[\sqrt{-3}]$, 
each cusp of $X^{*}_{\Gamma'}$ 
is resolved by an elliptic curve and the quotient group
$\Gamma/\Gamma'$ acts transitively on the set of cusps.
Therefore it suffices to look at the cusp at infinity
represented by $(1:0:0) \in B$.
The isotropy group $\Gamma_{\infty}$ 
of this cusp in $\Gamma$ consists of elements in $\Gamma$ of the form
\begin{equation} \label{eq-cusp}
\left( \begin{matrix}
t_1 && \\ &t_2& \\ &&t_3\\ \end{matrix} \right)
\left( \begin{matrix} 1 & x & -y' \\ 0 & 1 & 0 \\ 0 & y & 1 \\ \end{matrix} 
\right) 
\end{equation}
satisfying the conditions $t_1t_2'=t_3t_3'$ and $x+x'=-yy'$.
For the groups $\Gamma$ and $\Gamma_1$ we have $x\in \sqrt{-3}
\mathcal{O}_F$ and for $\Gamma$ (respectively $\Gamma_1$) we have $y \in 
\mathcal{O}_F$ (respectively  $y \in \sqrt{-3}\, \mathcal{O}_F$.  
Hence $t_i \in \mu_6$ and it follows that $t_1=t_2=t_3$ 
and the diagonal matrix lies in the center. 

Calculating the automorphy factors for an element
of $\Gamma_{\infty}$ as in \eqref{eq-cusp} gives
$$
j_1=t_2, \quad j_2=\left(\begin{matrix} 1/t_3 & y'/t_3 \\ 0 & 1/t_1\\
\end{matrix} \right), \quad j_3=t_1t_2t_3 \, .
$$
In particular, for an element of $\Gamma'_{\infty}$ as in \eqref{eq-cusp}
we find 
$j_1=1$, $j_2=(1, y'; 0, 1)$
and $j_3=1$. In this case the restriction of $L$ and $R$ to an
elliptic curve $E$ is trivial and $U$ restricted to $E$ is a
unipotent bundle of rank $2$, cf.\ \cite{Atiyah}.
This result also follows from the next proposition for which Thomas Peternell
kindly provided us with a proof. We identify a neighborhood of $E$ with the total space of the normal bundle of $E$ in $Y_{\Gamma'}$.
 
\begin{proposition}\label{peternell}
If $E$ is an elliptic curve and $N$ is a line bundle on $E$  with total
space $X$ then the rank $2$ bundle $V=\Omega^1_X(\log E)$ restricted to $E$
is an extension of ${\mathcal O}_E$ by ${\mathcal O}_E$ and it is non-trivial if and only if
$\deg N$ is non-trivial.
\end{proposition}

\begin{corollary}\label{indecvb}
For $\Gamma'$ as above the restriction of ${\rm Sym}^n(U)$ 
to the elliptic curve $E$ is isomorphic to $F_{n+1}$, the unique indecomposible
vector bundle of rank $n+1$ with a $1$-dimensional space of sections.
\end{corollary}
\end{subsection}
\begin{subsection}{Eisenstein cohomology} \label{sec-eis}
The Eisenstein cohomology is the contribution to the cohomology 
coming from the boundary. 
The study of Eisenstein cohomology was initiated by Harder 
(see \cite{Ha1}, \cite{Ha2}), who  used
the Borel-Serre compactification and topological methods to determine
the Eisenstein cohomology in a closely related case.
Here we use, as in \cite{vdG1}, 
coherent cohomology to determine the Eisenstein cohomology.

Let $\Gamma' \subset \Gamma_1[\sqrt{-3}]$ be a normal finite index
subgroup that acts freely on $B$. Let $\mathbb W_{\lambda}$ be a local system on
$X_{\Gamma'}$. 
The full Eisenstein cohomology $e_{\rm Eis,f}(X_{\Gamma'}, \mathbb W_{\lambda})$
is defined as 
$$
e_c(X_{\Gamma'}, \mathbb W_{\lambda})-e(X_{\Gamma'}, \mathbb W_{\lambda})\, ,
$$
where 
$e_c(X_{\Gamma'},\mathbb W_{\lambda})$, respectively  $e(X_{\Gamma'},\mathbb W_{\lambda})$, stands for the 
Euler characteristics 
$$
\sum_i (-1)^i [H^i_c(X_{\Gamma'},\mathbb W_{\lambda})], \;\;
\text{respectively $\sum_i (-1)^i [H^i(X_{\Gamma'},\mathbb W_{\lambda})]$},
$$
where the square brackets refer to classes in the Grothendieck group of mixed Hodge
structures.
The compactly supported Eisenstein cohomology 
$e_{\rm Eis}(X_{\Gamma'},\mathbb W_{\lambda})$ is defined as
$$
\sum (-1)^i [\ker \bigl(H^i_c(X_{\Gamma'}, \mathbb W_{\lambda}) \to H^i(X_{\Gamma'}, \mathbb W_{\lambda}) \bigr)]\, .
$$
Note that there is a map 
$e_{\rm Eis}(X_{\Gamma'},\mathbb W_{\lambda}) \to e_{\rm Eis,f}(X_{\Gamma'},\mathbb W_{\lambda})$.

\begin{lemma}\label{compEiscoh}
Put $\mathcal{F}_i$, for $i=0,1,2$, equal to
$$
\mathcal{W}_{n_1,-n_1-n_2,n_1+n_2+n_3}, \quad
\mathcal{W}_{n_1+n_2+1,-n_1+1,n_1+n_2+n_3}, \quad
\mathcal{W}_{n_2,n_1+3,n_2+n_3-1} \, .
$$
\begin{itemize}
\item[(i)] For any  $\lambda=n_1\gamma_1+n_2\gamma_2+n_3\gamma_3$  we have
$$
e_{\rm Eis,f}(X_{\Gamma'}, {\WW}_{n_1,n_2,n_3})=
\sum_{i=0}^2 (-1)^i \bigl([H^0(D,{\mathcal{F}_i}_{|D})]-[H^1(D,{\mathcal{F}_i}_{|D})]\bigr).
$$
\item[(ii)] For any $\lambda=n_1\gamma_1+n_2\gamma_2+n_3\gamma_3$ such that $n_1+n_2>1$ we have
$$e_{\rm Eis}(X_{\Gamma'},{\WW}_{n_1,n_2,n_3})=
-[H^0(D,{\mathcal{F}_0}_{|D})]+[H^1(D,{\mathcal{F}_0}_{|D})]
+[H^0(D,{\mathcal{F}_1}_{|D})].
$$
\end{itemize}
\end{lemma}
\begin{proof} Put $Y=Y_{\Gamma'}$.  From \cite[Theorem 5.5]{F-C} it follows  
(compare the description of the BGG-complexes for $j_*{\WW}_{\lambda}$ and 
$j_{!}{\WW}_{\lambda}$ in Section~\ref{sec-bgg}) that 
$$
e_{\rm Eis,f}(X_{\Gamma'}, {\WW}_{n_1,n_2,n_3})=
\sum_{i=0}^2 \sum_{j=0}^2(-1)^{i+j} \bigl([H^i(Y,{\mathcal{F}_j}(-D))]-[H^i(Y,{\mathcal{F}_j})]\bigr).
$$

For any locally free sheaf $\mathcal{F}$ on $Y_{\Gamma'}$ we have an exact sequence 
\begin{equation*} 
0 \to \mathcal{F}(-D) \to 
\mathcal{F} \to \mathcal{F}_{|D}
\to 0.
\end{equation*}
Taking the long exact sequence deduced from this sequence for 
the three cases $\mathcal{F}_r$, with $r=0,1,2$, statement (i) follows. 

For (ii) we again use the BGG-complexes of Section~\ref{sec-bgg} to conclude that 
$$
e_{\rm Eis}(X_{\Gamma'}, {\WW}_{n_1,n_2,n_3})=
\sum_{i=0}^2 \sum_{j=0}^2(-1)^{i+j} [\mathrm{ker}\bigl(H^i(Y,{\mathcal{F}_j}(-D)) \to H^i(Y,{\mathcal{F}_j}) \bigr)].
$$
We now consider the three cases $r=0,1,2$ individually. 
For the case $r=0$ the vanishing of $H^0(Y,\mathcal{W}_{j,k,l})$
for negative $k$ implies that the sequence reduces to
$$
\begin{aligned}
0\to H^0(D,{\mathcal{F}_0}_{|D}) \to H^1(Y,\mathcal{F}_0(-D))\to 
H^1(Y,\mathcal{F}_0) \to &\\
H^1(D,{\mathcal{F}_0}_{|D}) \to H^2(Y,\mathcal{F}_0(-D)) \to
H^2(Y,\mathcal{F}_0) \to  & \, 0  \\
\end{aligned}
$$
We can identify the last three terms
in the sequence with the Serre duals of
$$
M_{n_1,n_2+3,-n_2-n_3-1}(\Gamma') \leftarrow 
S_{n_1,n_2+3,-n_2-n_3-1}(\Gamma') \leftarrow 0 \, .
$$
It follows from Corollary~\ref{indecvb} that $h^0(D,\mathcal{F}_0)$ 
equals the number $c$ of cusps of $X^*_{\Gamma'}$. Possibly 
replacing $\Gamma'$ by a finite index subgroup (and in the end 
taking invariants) there is, since $n_1+n_2>1$, an Eisenstein series 
for each cusp, see \cite[Prop. 12.1]{C-vdG}. 
Since both $U$ and $L$ have 
degree $0$ when restricted to an elliptic curve in $D$, 
Riemann-Roch tells us that $h^1(D,\mathcal{F}_0)=h^0(D,\mathcal{F}_0)$. 
Finally, since
$$\dim M_{n_1,n_2+3,-n_2-n_3-1}(\Gamma') =c+\dim S_{n_1,n_2+3,-n_2-n_3-1}(\Gamma') $$
it follows that 
$$H^1(D,{\mathcal{F}_0}_{|D}) \to H^2(Y,\mathcal{F}_0(-D)) $$ 
is injective.

For $\mathcal{F}_1$ we find that the sequence reduces to
$$
0 \to H^0(D,{\mathcal{F}_1}_{|D}) \to H^1(Y,\mathcal{F}_1(-D)) \to
H^1(Y,\mathcal{F}_1) \to H^1(D,{\mathcal{F}_1}_{|D}) \to 0 \, ,
$$
again by the vanishing of $M_{j,k,l}(\Gamma')$ for $k<0$.
Similarly to the case $r=0$, we can identify $H^0(D,\mathcal{F}_2)$ with a space
of Eisenstein series, and hence with the cokernel of the second arrow in
the exact sequence
$$
\begin{aligned}
0 \to H^0(Y,\mathcal{F}_2(-D))\to  H^0(Y,\mathcal{F}_2) \to H^0(D,\mathcal{F}_2)& \\
H^1(Y,\mathcal{F}_2(-D))\to H^1(Y,\mathcal{F}_2)\to H^1(D,\mathcal{F}_2) 
\to & \, 0.
\end{aligned}
$$
This finishes the proof of (ii). 
\end{proof}

\begin{corollary}\label{vanishingh1}
For any regular $\lambda=n_1\gamma_1+n_2\gamma_2+n_3\gamma_3$ we have
$$H^1(Y_{\Gamma'},\mathcal{F}_{0})=0.$$ 
\end{corollary}
\begin{proof}
As noted at the end of Section~\ref{sec-bgg}, since $\lambda$ is regular, $H^1_{!}(X,{\WW}_{\lambda})=0$. 
The result now follows directly from Remark~\ref{rmk-duality} together with Proposition~\ref{prop-bggfilt}. 
\end{proof}

\begin{definition} We write $\mathbf s_{\mu}$ for the irreducible representation of $\fS_4$ indexed in the usual way by $\mu$ a partition of $4$. We then define the $\mathfrak{S}_4$-representations,
$$
\gamma_{i}= \begin{cases} \mathbf s_4+\mathbf s_{3,1} & i\equiv_6 0 \\
\mathbf s_{1^4}+\mathbf s_{2,1^2} & i \equiv_6 3  \\ 
0 & \text{else} \\
 \end{cases}
$$
\end{definition}

In the following proposition $\LL^{i,j}$ will denote the motive corresponding to a Hecke character, see further in Section~\ref{sec-hecke}; it is $1$-dimensional and has Hodge degree $(i,j)$.

\begin{proposition} \label{prop-eis} For regular $\lambda=n_1\gamma_1+n_2\gamma_2+n_3\gamma_3$ with $n_1 \equiv_3 n_2$ and $n_1 \equiv_2 n_3$, put $i=n_2+n_3$. Then $e_{\rm Eis}(X_{\GaSq},{\WW}_{\lambda})$ 
consists of the following three contributions, 
$$
(-\gamma_i \, \LL^{0,0}+ \gamma_{i-n_2-1}\, \LL^{n_2+1,0}+
\gamma_{i+n_1+1}\, \LL^{0,n_1+1}) \, \LL^{n_3,n_3+i} \, .
$$
\end{proposition}

As another approach to this Eisenstein cohomology is already in the
literature \cite{Ha1,Ha2} we will only provide a sketch of proof. We will normalize the weights, which corresponds to removing the factor $\LL^{n_3,n_3+i}$ in Proposition~\ref{prop-eis}, see also further in Section~\ref{sec-normalized}. 

For each elliptic curve $E$ occurring in the divisor
$D$ each of the 
the contributions in Lemma \ref{compEiscoh} are either zero or $1$-dimensional,
see Corollary~\ref{indecvb}.
We observe that the Hodge weights of the three contributions are $0$, $0$ and $n_2+1$. Interchanging $n_1$ and $n_2$ amounts to changing the complex
structure by its complex conjugate. Hence this limits the Hodge types
to $(0,0)$, $(n_2+1,0)$, $(0,n_1+1)$ and $(n_2+1,n_1+1)$. But the last one
can occur only in the dual of $e_{\rm Eis}$.
So we find the Hodge types 
$$
(0,0),\,  (n_2+1,0), \, (0,n_1+1)
$$
for the three contributions.
 
We now specialize to the situation $\Gamma'=\Gamma_1[\sqrt{-3}]$ where
$\Gamma/\Gamma_1[\sqrt{-3}] = \mathfrak{S}_4 \times \mu_6$ acts 
on $Y_{\Gamma_1[\sqrt{-3}]}$ by permuting 
the cusps and the group $\mathfrak{S}_4\times \mu_3$ acts effectively. 
The group
$\langle \rho 1_3 \rangle$ sits in the center of $\Gamma[\sqrt{-3}]$
making the ${\mathcal F}_i$ orbifold bundles. We thus get contributions
to the Eisenstein cohomology only if $n_1\equiv_3 n_2$.

Since 
for each elliptic curve occuring in the divisor $D$ on 
$Y_{\Gamma_1[\sqrt{-3}]}$ the
cohomology group $H^i(E,\mathcal{F}_j)$ is at most $1$-dimensional
and the stabilizer of such an elliptic curve is $\mathfrak{S}_3 \times \mu_3$
it contributes either zero or a $\mathfrak{S}_4$-representation  
of the form
$$
{\rm Ind}_{\mathfrak{S}_3}^{\mathfrak{S}_4} \mathbf s_3= \mathbf s_4+\mathbf s_{3,1}
\quad \text{or}\quad 
{\rm Ind}_{\mathfrak{S}_3}^{\mathfrak{S}_4} \mathbf s_{1^3}= \mathbf s_{2,1^2}+\mathbf s_{1^4}\, .
$$
Furthermore, the element 
$\gamma={\rm diag}(1,1,\rho)$ acts by multiplication
by $\rho$ on each resolution elliptic curve $E$. 
The Eisenstein cohomology for $\Gamma[\sqrt{-3}]$ corresponds to the
invariant part under the action of $\gamma$.
\end{subsection}
\end{section}
\begin{section}{The dimension of spaces of Picard modular forms}
In this section we are interested in the dimensions of spaces of modular
forms and cusp forms on our groups $\Gamma_1[\sqrt{-3}]$ and
$\Gamma[\sqrt{-3}]$. We dwell upon this since it provides important checks 
on the conjectures on the cohomology of local systems. 

Recall that the space $M_{j,k}(\Gamma_1[\sqrt{-3}])$
splits as
$$
M_{j,k}(\Gamma_1[\sqrt{-3}])=
M_{j,k}(\Gamma[\sqrt{-3}])\oplus M_{j,k}(\Gamma[\sqrt{-3}],{\det})
\oplus M_{j,k}(\Gamma[\sqrt{-3}], {\det}^2)
$$
and similarly for the spaces of cusp forms $S_{j,k}(\Gamma_1[\sqrt{-3}])$.
Recall furthermore that the graded ring 
$\oplus_{k=0}^{\infty} M_{0,3k}(\Gamma_1[\sqrt{-3}])$ 
is a degree $3$ extension of  
$$\oplus_{k=0}^{\infty} 
M_{0,3k}(\Gamma[\sqrt{-3}])={\CC}[x_1,x_2,x_3,x_4]/(x_1+x_2+x_3+x_4)
$$
generated by
$\zeta \in S_6(\Gamma[\sqrt{-3}],{\det})$ 
satisfying the relation~\eqref{pmf-zeta}. Moreover, $\mathfrak{S}_4$
acts by $\sigma: x_i \mapsto \mathrm{sgn}(\sigma) x_{\sigma(i)}$
and by the sign on $\zeta$.
We find (see Section~7 and Section~12 of \cite{C-vdG} for more details) that
$
\dim M_{0,3k}(\Gamma_1[\sqrt{-3}])= (3/2) k (k-1) +  4$ for $k\geq 2$ and $\dim M_{0,3}=3$;
moreover, for $k\geq 2$
$$
\dim S_{0,3k}(\Gamma[\sqrt{-3}],{\det}^{\ell}) =
\begin{cases}
(k^2+3k-6)/2 & \ell=0 \\
(k^2-k)/2 & \ell=1 \\
(k^2-5k+6)/2 & \ell=2\, . \\
\end{cases}
$$
We now deduce a formula for the space of cusp forms on a freely acting finite index subgroup
$\Gamma'$ of $\Gamma_1[\sqrt{-3}]$. The surface $\Gamma'\backslash B$ 
is smoothly compactified by adding a divisor $D$
consisting of disjoint elliptic curves.

\begin{theorem}\label{dimSjk}
For $j \geq 0$ and $k>0$ we have 
$$
\dim S_{j,3+k}(\Gamma')=
\frac{1}{6} (j+1)(k+1)(j+k+2)\, {\rm vol}(\Gamma'\backslash B)+
\frac{1}{12}(j+1) \,  D^2
$$
with ${\rm vol}(\Gamma'\backslash B)=c_2(\Gamma'\backslash B) =3 \, c_1(L)^2$.
\end{theorem}
\begin{remark}
Note that the formula is symmetric in $j$ and $k$ 
up to the factor $j\, D^2/12$. For a group $\Gamma'$ acting on the ball 
with a compact quotient $\Gamma' \backslash B$ one would get a formula symmetric in
$j$ and $k$.
\end{remark}
Before we give the proof we state a simple lemma.
\begin{lemma}
With $\gamma=c_1(L)$, the Chern character satisfies 
$$
{\rm ch}({\mathcal W}_{a,b})=
(a+1)\left(1+(a/2+b)\gamma+(b^2/2+ab/2-a/4)\gamma^2 \right) \, .
$$
\end{lemma}
\begin{proof} Since $\Gamma'$ acts freely
we have that ${\mathcal W}_{a,b}={\rm Sym}^a(U) \otimes L^b$ and we know
that $\det(U)=L$. Let $\alpha_1$ and $\alpha_2$ be the Chern roots of $U$.
Then $\alpha_1+\alpha_2=\gamma$ and by Hirzebruch-Mumford
proportionality we have
$c_1(\Omega^1(\log D))^2=3c_2(\Omega^1(\log D))$,
hence $(3\gamma)^2=3(\alpha_1\alpha_2+2 \gamma^2)$,
so $\alpha_1+\alpha_2=\gamma$ and $c_2(U)=\alpha_1\alpha_2=\gamma^2$.
This implies $c_1({\rm Sym}^a(U))= (1/2) a(a+1) \gamma$ and
$$
{\rm ch}_2({\rm Sym}^a(U))= \sum_{i=0}^a ((a-i)\alpha_1+i(\alpha_2))^2/2=
-a(a+1) \gamma^2/4.
$$
The result easily follows from this.
\end{proof}

\begin{proof}
We work on $Y=Y_{\Gamma'}$. We denote the resolution divisor of the cusps by
$D$. The canonical line bundle $K_Y$ is $L^{\otimes 3}\otimes {\mathcal O}(-D)$.
In view of the vanishing of $h^1$ and $h^2$ (by Corollary~\ref{vanishingh1} 
and Proposition \ref{vanishing}) 
of ${\mathcal W}_{j,k}\otimes \Omega^2$  
we have that
$$
\dim S_{j,k+3}(\Gamma')= h^0(Y,{\mathcal W}_{j,k}\otimes \Omega_Y^2)=
\chi(Y, {\mathcal W}_{j,k}\otimes \Omega_Y^2)
$$
with $\chi$ denoting the Euler characteristic.
By Serre duality we have
$$
\chi(Y, {\mathcal W}_{j,k}\otimes \Omega_Y^2)= \chi(Y,{\mathcal W}_{j,-j-k})
$$
We apply the Hirzebruch-Riemann-Roch formula
$$
\chi(Y,{\mathcal W}_{a,b})=
({\rm ch}({\mathcal W}_{a,b}) \cdot {\rm td})[Y]
$$
with ${\rm td}$ the Todd class $1-3\gamma/2+D/2+(c_1^2+c_2)[Y]/12$ and
$c_1(Y)=-3\gamma+D$, where again we write $\gamma=c_1(L)$ for
the first Chern class of $L$. 

Using  $\chi(Y,{\mathcal O}_Y)= (c_1^2+c_2)[Y]/12= \gamma^2+D^2/12$
we now find
$$
\begin{aligned}
\chi(Y,{\mathcal W}_{a,b})=& 
(a+1)\left((b^2+ab-2a-3b)\gamma^2/2+\chi(Y,{\mathcal O}_Y)\right) \\
=& (a+1) \left( (b^2+ab-2a-3b+2)\gamma^2/2+D^2/12) \right) . \\
\end{aligned}
$$
Substituting $a=j$ and $b=-j-k$ gives
$$
(j^2k+k^2j+j^2+4jk+k^2+3j+3k+2)\gamma^2/2+(j+1)D^2/12.
$$
\end{proof}
\begin{remark}
For our group $\Gamma_1[\sqrt{-3}]$ we have 
${\rm vol}(\Gamma_1[\sqrt{-3}]\backslash B)=1$ and for 
$\Gamma[\sqrt{-3}]$ we have ${\rm vol}(\Gamma[\sqrt{-3}]
\backslash B)=1/3$.
\end{remark}

For our group $\Gamma_1[\sqrt{-3}]$ we have the following dimension formula.

\begin{proposition}
For $j\geq 0$ and $k> 0$ with $j\equiv_3 k$ 
the dimension of $S_{j,k+3}(\Gamma_1[\sqrt{-3}])$ is given by 
$$
\frac{1}{6}(j+1)(k+1)(j+k+2) -(j+1) +
\begin{cases} 2/3 & j \equiv_3 0  \\
-2/3 & j \equiv_3 1  \\
0 & j \equiv_3 2  \\
\end{cases}
$$
\end{proposition}
\begin{remark} \label{rem-k0} We conjecture that the formulas above also hold for $k=0$ in all cases except when $j \equiv_6 0$ for which the dimension is $(j^2-3j+6)/6$. This is based upon computations of the numerical Euler characteristic in Section~\ref{sec-numeric} and Conjecture~\ref{conj-dim}.
\end{remark}
\begin{proof}
The subgroup $\Gamma_1[3]$ of $\Gamma_1[\sqrt{-3}]$ acts freely and the
formula of Theorem~\ref{dimSjk} holds with ${\rm vol}(\Gamma_1[3])=81$
and $D$ consisting of
$4\cdot 27$ elliptic curves with self-intersection number $-9$. The
quotient group $G'=\Gamma_1[\sqrt{-3}]/\Gamma_1[3]$ is of order $3^5$,
but its center $\langle \rho \, {\rm id}\rangle=\mu_3$ 
acts trivially. We will
apply the holomorphic Lefschetz formula to the action of the group $G=G'/\mu_3$
on $Y_{\Gamma_1[3]}$ and the vector bundle 
$\mathcal{W}_{j,k}\otimes \Omega^2_{Y_{\Gamma_1[3]}}$. 
The group $G$ is 
isomorphic to $({\ZZ}/3{\ZZ})^4$. The action of $G$ on $Y_{\Gamma[3]}$ has
$3\cdot 27$ fixed points lying over the quotient singularities of
$X_{\Gamma_1[\sqrt{-3}]}$ and no other fixed points. 
We know that the quotient singularieties are of 
order $3$ and type $(1,2)$. The holomorphic Lefschetz formula now says
that for 
the action of $G$ on the manifold $Y=Y_{\Gamma_1[3]}$ and vector bundle
$V$ we have
$$
\sum_{i} (-1)^i \dim H^i(Y,V)^G=(1/\# G) \sum_{g\in G} \chi(Y,V,g)
$$
with $\chi(Y,V,g)={\rm ch}(Y,V,g){\rm td}(Y,g)[Y^g]$ with ${\rm td}(Y,g)$
the Todd class associated to $Y^g$. At a fixed point of a non-trivial element
of $G$ the Todd class is $1/(1-\rho^2)(1-\rho)=1/3$. The action on 
the fibre of $L^3$ over a fixed point is trivial, while on the fibre of
$\Omega^1$ it acts by $(\rho,\rho^2)$. Hence on the fibre of $\mathcal{W}_{j,k}$
it acts by eigenvalues $\rho^r \rho^{2(j-r)}$ for $r=0,\ldots,j$.
So one non-trivial element of $G$ yields for one fixed point  a contribution
$$
c(j)= (1/3)\sum_{r=0}^j \rho^{2j-r}=\begin{cases} 1/3 & j\equiv_3 0  \\
-1/3 & j\equiv_3 1 \\ 0 & j \equiv_3 2 \\ \end{cases}
$$
In total the $81$ fixed points are each fixed by $2$ non-trivial elements,
hence we get a contribution
$(1/81)\cdot 81 \cdot 2 \cdot c(j)$.
The identity element contributes $(j+1)(k+1)(j+k+2)/6 -(j+1)$. Together this
proves the formula.
\end{proof}
Next we want a formula for the dimension 
$s_{j,k+3,l}$ 
of the space
of modular forms $S_{j,k+3}(\Gamma[\sqrt{-3}], \det^l)$. 
Since we have a formula for
$$
\sum_{l=0}^2s_{j,k+3,l}=\dim S_{j,k+3}(\Gamma_1[\sqrt{-3}])
$$
it suffices to calculate the trace of a generator
$\gamma={\rm diag}(1,1,\rho)$ of
the group $\Gamma[\sqrt{-3}]/\Gamma_1[\sqrt{-3}]\cong {\ZZ}/3{\ZZ}$
on $S_{j,k+3}(\Gamma_1[\sqrt{-3}])$.
Indeed, for an operator $\gamma$ of order
$3$ acting on a complex vector space of dimension $n$ with
with eigenspaces of dimension $m_1$, $m_{\rho}$ and $m_{\rho^2}$ for
the eigenvalues $1, \rho$ and $\rho^2$ (with
$n=m_1+m_{\rho}+m_{\rho^2}$) and with
${\rm trace}(\gamma)= a+b \, \rho$ with $a,b \in {\ZZ}$
we have
$m_1= (n+2\, a -b)/3,$ $m_{\rho}=(n-a+2\, b)/3$
and $m_{\rho^2}= (n-a -b)/3$.

\begin{theorem} \label{dim-Ajkl}
If $j\not\equiv_3 k$ then $\dim S_{j,k+3,l}(\Gamma[\sqrt{-3}])=0$. If
$j \equiv_3 k$ and $k>0$ we have 
$$
\dim S_{j,k+3,l}(\Gamma[\sqrt{-3}])= 
\frac{1}{18}(j+1)(k+1)(j+k+2)+A_{j,k,l},
$$
with $A_{j,k,l}$ depending on the residue class of $(j,l)$ in
$({\ZZ}/3{\ZZ})^2$
as follows
$$
\begin{matrix}  && l\equiv_{3}0 & l\equiv_31 & 
l\equiv_{3}2 \\
\\
j\equiv_{3}0 && 2k/3-10/9 & -j/3-1/9 & -2j/3-2k/3+8/9 \\
j\equiv_{3}1 && -j/3-5/9 & 2k/3-14/9 & -2j/3-2k/3+4/9 \\
j\equiv_{3}2 && 0 & 0 & -j-1\\
\end{matrix}
$$
\end{theorem}
\begin{remark} We conjecture that the formulas above also hold for $k=0$ in all cases except when $l=0$ and $j \equiv_6 0$, and then $A_{j,0,0}=-1/9$. This is based upon the same evidence as in Remark~\ref{rem-k0}.
\end{remark}

\begin{proof}
We apply the holomorphic Lefschetz formula to the action
of a representative $\gamma={\rm diag}(1,1,\rho)$ of
$\Gamma[\sqrt{-3}]/\Gamma_1[\sqrt{-3}]\cong {\ZZ}/3{\ZZ}$ on the surface
$Y=Y_{\Gamma_1[\sqrt{-3}]}$ and the orbifold vector bundle
$\Omega^2_{Y} \otimes {\mathcal W}_{j,k}$.
We assume that $j\equiv_3 k$ otherwise
$M_{j,k}(\Gamma_1[\sqrt{-3}])$ is zero. 
We can write 
$$
\Omega^2_Y \otimes \mathcal{W}_{j,k}=
\Omega^2_Y \otimes {\rm Sym}^j(U\otimes L) \otimes L^{k-j} \, ,
$$
where the center $\mu_3$ acts on all three factors trivially.
The fixed point locus of $\gamma$ on the surface $Y_{\Gamma_1[\sqrt{-3}]}$
consists of the six curves $D_{ij}$ (see Section \ref{sec-PMS}) 
and the three intersection points
of the two resolution curves of the three quotient singularities
on $X_{\Gamma_1[\sqrt{-3}]}$. Each of the $D_{ij}$ is a smooth rational
curve which is an exceptional curve, 
the restriction of $\Omega^1_Y$ to $D_{ij}$
is ${\mathcal O}(2)\oplus {\mathcal O}(-1)$, with
the first factor the cotangent bundle to $D_{ij}$ and the second the
conormal bundle.  And since the resolution divisor $D$ of the cusps
intersects $D_{ij}$ transversally at two points, we find 
$$
\Omega^1(\log D)_{|D_{ij}}= \mathcal{O}\oplus {\mathcal O}(1)\, .
$$
The action of $\gamma$ preserves the two factors. It acts trivially on the
cotangent bundle of $D_{ij}$ and by $\rho^2$ on the conormal bundle
since one $D_{ij}$ is given by $u=0$ and 
$\gamma$ acts by $(u,v)\mapsto (\rho u,v)$. Since $\Omega^1(\log D)
\cong U\otimes L $ we get $U\otimes L_{|D_{ij}}=
\mathcal{O}\oplus \mathcal{O}(1)$
with $\gamma$ acting by $1$ on the first factor
and by $\rho^2$ on the second one. 

We need the Todd class along $D_{ij}$; if we write the first Chern class of
${\mathcal O}_{D_{ij}}(1)$ as $P$, the class of a point, then we find
$$
{\rm td}(D_{ij},\gamma)= \frac{2P}{1-e^{-2P}} \frac{1}{1-\rho^2 e^P} \, ,
$$
where the first factor comes from the tangent bundle and the second
from the normal bundle
resulting in
$$
{\rm td}(D_{ij},\gamma)=\frac{2+\rho}{3}-\frac{1+\rho}{3} P\, .
$$
We have $L^3_{|D_{ij}}=\mathcal{O}(1)$ with $\gamma$ acting trivially.
We need ${\rm ch}(\Omega^2_Y \otimes \mathcal{W}_{j,k}{|D_{ij}})$
in the cohomology of $D_{ij}$ tensored with the representation ring
of $\mu_3=\langle \gamma \rangle$. 
We thus write ${\rm ch}({\rm Sym}^j(U\otimes L)_{|D_{ij}})=r(j)+d(j)P$, 
where $P$ denotes the cohomology class of a point on $D_{ij}$ 
and we interpret $r(j)$ and $d(j)$ as elements of ${\ZZ}[\rho]$. 
Furthermore, ${\rm ch}(\Omega^2_Y |D_{ij})=\rho^2(1-P)$.
From the description just given we obtain:
$$
{\rm ch}(\Omega^2\otimes \mathcal{W}_{j,k}|D_{ij})
=(r_j+d_j P) (\rho^2-\rho^2 \, P) \left( (1+(\frac{k-j}{3})P\right) \, .
$$
We have
$ r_j=\sum_{a=0}^j \rho^{2a}$,
and
$d_j=\sum_{a=0}^j \rho^{2a} \, a $
and we thus find
$$
r(j)+d(j) \, P = 
\begin{cases}
1+ \frac{j}{3}(2+\rho)\, P & j\equiv_3 0\\
-\rho + \left(\frac{j+2}{3}(-1-2\rho) +\rho \right)\, P & j\equiv_3 1 
\\
0 + \frac{j+1}{3}(\rho-1) \, P & j\equiv_3 2  \\
\end{cases}
$$
The contribution of the six $D_{ij}$ is
the coefficient of $P$ in
$$
6 \,  \rho^2 (r(j)+d(j) P)(1+(\frac{k-j-3}{3})P)
(\frac{2+\rho}{3}-\frac{1+\rho}{3} P) 
$$
and this is
$$
\begin{cases} 
(-2k+2j+6)/3+ (-4k-2j+6)\rho/3 & j\equiv_3 0 \\
(-4k-2j+6)/3+(-2k+2j+6)\rho/3 & j \equiv_3 1 \\
2(j+1)(1+\rho) & j \equiv_3 2 \\
\end{cases}
$$
By adding to these three cases $-(j+1)$ (respectively $-(j+1)\rho$ and
$(j+1)\rho^2$) as the contribution of the three isolated fixed points
of $\gamma$ one finds for the trace
$$
(2k+j-3)(-1-2\rho)/3, \quad (2k+j-3)(-2-\rho)/3, \quad -(j+1)\rho^2 \, .
$$
This gives the desired traces.
\end{proof}

We end this section with a definition.
\begin{definition}
The space $S_{j,k,l}(\GaSq)$ is a representation of $\fS_4$ and we denote by $S_{j,k,l}(\GaSq)^{\mu}$ the isotypic component corresponding to the irreducible representation indexed by ${\mu}$, a partition of $4$. We then put
$$
\dim_{\fS_4} S_{j,k,l}(\GaSq):=\sum_{\mu \vdash 4} \frac{\dim S_{j,k,l}(\GaSq)^\mu}{\dim \mathbf s_{\mu}} \, \mathbf s_{\mu}.$$
In Section~\ref{sec-conj}, we define a (conjectural) subset 
$$ S^{\rm gen}_{j,k,l}(\GaSq) \subset S_{j,k,l}(\GaSq)$$ 
of so-called genuine forms  and we define $\dim_{\fS_4} S^{\rm gen}_{j,k,l}(\GaSq)$ analogously to the above. 
\end{definition} 
\end{section}

\begin{section}{Moduli spaces of abelian threefolds with $\rho$-action}
In this part of the article we will study our spaces 
using their interpretation as moduli spaces of abelian threefolds 
with an action by $\rho$. This will give us a Deligne-Mumford stack defined over $\ZZ[\rho,1/3]$,  
which enables us to find cohomological information through its finite fibres. 
Our goal is the $\ell$-adic Euler characteristics of the
local systems on our moduli spaces as motives, or more specifically as
representations of the absolute Galois group $\Gal(\overline{F}/F)$. 
\begin{subsection}{Picard modular stacks} \label{sec-stacks} 
For any scheme $S$ defined over $\cO_F[1/3]$ consider the groupoid whose objects are tuples $(A,\lambda,\iota)$, where $A$ is an abelian scheme of relative dimension $3$ over $S$, $\lambda: A \to A^{\vee}$ is a principle polarization of $A$, and $\iota:\cO_F \to \mathrm{End}_S(A)$ is a homomorphism such that the Rosati involution associated to $\lambda$ acts by complex conjugation on $\iota(\cO_F)$ and that gives $\Omega^1_{A/S}$ a structure of $\cO_S \otimes_{\ZZ} \cO_F$-module of signature $(2,1)$. Isomorphisms between $(A,\lambda,\iota)$ and $(A',\lambda',\iota')$ are given by isomorphisms $f:A \to A'$ such that $\lambda=f^{\vee} \circ \lambda' \circ f$ and $f \circ \iota(a) = \iota'(a) \circ f$ for all $a \in \cO_F$. This moduli problem is represented by a Deligne-Mumford stack $\cX'_{\Gamma}$ of relative dimension $2$ that is separated, smooth, connected and of finite type over $\cO_F[1/3]$, see \cite[Cor. 1.4.1.12]{Lan} and \cite{Larsen}. The Picard modular surface $X_{\Gamma}$ is equal to the complex fibre $\cX'_{\Gamma}(\CC)$.

\begin{notation} Put $\alpha:=\iota(\rho)$.
\end{notation}

To our moduli problem we now add the principal level-structure with respect to the endomorphism $1-\alpha$. Consider tuples $(A,\lambda,\iota,\sigma)$ where 
$$\sigma:(\cO_F^3/(1-\alpha)\cO_F^3)_S \to A[1-\alpha]$$  
is an $\cO_F$-equivariant isomorphism, see \cite[Definition 1.3.6.1]{Lan}. This is represented by a Deligne-Mumford stack $\cX'_{\GaSq}$ with the same properties as above and with $X_{\GaSq}$ as complex fibre. Note that $\cX'_{\GaSq}$ comes with an action of the finite group $\GaSq/\Gamma$.

More generally, for any open compact subgroup of $G(\hat \ZZ)$ we get a level-structure that we can add to our moduli problem and get a stack with the same properties as above, see \cite{Lan}. The open compact subgroup 
$$K_{\Gamma}:= \{g \in G(\hat \ZZ): g \cdot \cO_F^3\otimes_{\ZZ} \hat \ZZ =\cO_F^3\otimes_{\ZZ} \hat \ZZ\},$$
corresponds to $\cX'_{\Gamma}$ and if we replace $G$ by $G^0$ in the definition we get $\cX'_{\Gamma_1}$. 
From the subgroup 
$$K_{\Gamma[\sqrt{-3}]}:=\{g \in G(\hat \ZZ): (g-1) \cdot \cO_F^3\otimes_{\ZZ} \hat \ZZ \subset \sqrt{-3} \cdot \cO_F^3\otimes_{\ZZ} \hat \ZZ\},$$ 
we get $\cX'_{\GaSq}$, and  replacing $G$ by $G^0$ we get $\cX'_{\GaSqOne}$, with the corresponding Picard modular surfaces as complex fibres. 
\end{subsection}
\begin{subsection}{Shimura varieties and complex tori} 
We briefly revisit our Picard modular surfaces as Shimura varieties and as moduli space of complex tori. For any compact open subgroup $K$ in $G(\bA_f)$ put, 
$$S_{K}(G,B):=G(\QQ) \backslash B \times G(\bA_f)/K.$$
This is a Shimura variety defined over $\CC$. Taking $K$ equal to any of the compact open subgroups of the previous section we get a connected Shimura variety isomorphic to the corresponding Picard modular surface, see \cite{Go}. 

An abelian variety $A$ in $\cX'_{\Gamma}(\CC)$ is a complex torus $V/\Gamma_{*}$ with 
$\Gamma_{*}$ an $\cO_F$-module of rank~$3$. Since the class number is $1$, the lattice $\Gamma_{*}$ is isomorphic to $\cO_F^3$. 
The polarization of $A$ gives rise to an alternating form $E$ on the underlying
real vector space of $V$ which satisfies $E(Jx,Jy)=E(x,y)$ with
$J$ the complex structure on $V$. This gives that $E(\alpha x, y)=
E(x, \alpha'y)$ for all $\alpha \in \cO_F$ and where $z\mapsto z^{\prime}$ 
denotes the Galois automorphism of $F/{\bf Q}$. The corresponding
hermitian form $h$ may be normalized (see \cite{Shimura}) so that 
$$ h(z_1,z_2,z_3) = z_1z_2^{\prime} +z_1^{\prime}z_2+z_3z_3^{\prime}. $$
\end{subsection}
\begin{subsection}{Moduli spaces of curves} \label{sec-moduliofcurves}
The Torelli map that sends a smooth curve of genus $g$ to its Jacobian, 
induces an embedding of coarse moduli spaces.  
For the corresponding stacks this does not hold in general due to the fact that if $G$ is the automorphism group  of a non-hyperelliptic curve $C$, 
then the automorphism group of its Jacobian equals $G \times \{-1\}$. 

An abelian threefold is either geometrically indecomposable, or a product of the Jacobian of a smooth genus two 
curve and an elliptic curve, or an unordered product of three elliptic curves. We cut up our spaces 
$\cX'_{\Gamma_{*}}$, for all $\Gamma_{*}$ among the four groups $\Gamma, \Gamma_1, \Gamma[\sqrt{-3}]$ 
and $\Gamma_1[\sqrt{-3}]$, into three pieces $\cX'_{1,\Gamma_{*}}$, $\cX'_{2,\Gamma_{*}}$ and $\cX'_{3,\Gamma_{*}}$  
according to this distinction.  

We will now shift focus to the corresponding moduli spaces of curves $\cX_{\Gamma_{*}}$, defined over $\ZZ[\rho,1/3]$. The stratification on the moduli of abelian varieties induces via Torelli a stratification on the moduli of curves, denoted by $\mathcal{X}_{1,\Gamma_*}$, $\mathcal{X}_{2,\Gamma_*}$ and
$\mathcal{X}_{3,\Gamma_*}$. 

The space $\cX_{\Gamma}$ will then be the moduli space of curves of compact type together with an automorphism of order three, with action of type $(2,1)$, that induces an admissible tricyclic cover of the projective line (compare the definitions in \cite{Achter-Pries}). The level structure for $\GaSq$ will be described in terms of markings of points on the curves. We will not attempt an analogous description for  $\GaSqOne$.

Define also, in a completely analogous way, the moduli space $\cX^{(2)}_{\Gamma_{*}}$ of abelian surfaces with signature $(1,1)$ and the moduli space $\cX^{(1)}_{\Gamma_{*}}$ of elliptic curves with signature $(1,0)$. 

The two spaces $\cX_{\Gamma_{*}}$ and $\cX'_{\Gamma_{*}}$ only differ for the open strata of geometrically indecomposable abelian threefolds, due to the difference in automorphism group mentioned above. This is intimately connected to the fact that a geometrically indecomposable abelian threefold is either the Jacobian of a smooth curve of genus three, or the $(-1)$-twist of the Jacobian of a (non-hyperelliptic) smooth curve of genus three. 

Our main interest is the cohomology of local systems of these spaces, and in Remark~\ref{rmk-curveab} we will relate the cohomology of these two types of spaces. This relation shows that there are no new motives appearing in the cohomology of local systems on $\cX_{\Gamma_{*}}$ other than the ones appearing for $\cX'_{\Gamma_{*}}$. This is in sharp contrast to the situation when comparing the cohomology of local systems on the moduli space of curves $\mathcal M_3$ and the moduli space of principally polarized abelian varieties $\mathcal A_3$, see \cite{CFG}.

\begin{notation}
Let $K$ denote a field, which is not of characteristic~$3$, containing a primitive third root of unity $\tilde \rho$ that we fix. 
\end{notation}
\subsection{Smooth curves of genus $3$} \label{sec-smooth3}
Let $C/K$ be a smooth curve of genus $3$ and let $\alpha$ be an automorphism of $C$ of order $3$ of type $(2,1)$, which means that it will have eigenvalues $(\tilde \rho,\tilde \rho,\tilde \rho^2)$ when acting on the $3$-dimensional vector space $H^0(C,\Omega^1_C)$.

Since there are no invariant differentials of $\alpha$ we see that $\alpha$ induces a cyclic triple cover of $\PP^1 \cong C/\alpha$. The Riemann-Hurwitz formula tells us that there are five ramification points, which is the same as fixed points of $\alpha$. Let $c_{\tilde \rho^i}$ be the number of ramification points $x$ such that the action of $\alpha$ on $\Omega^1_{C,x}$ is by multiplication by $ \tilde \rho$. We then have that $c_{\tilde \rho^2}=5-c_{\tilde \rho}$. The Woods-Hole formula, together with the isomorphism between $H^1(C,\mathcal{O}_C)$ and the dual of $H^0(C,\Omega^1_C)$, tells us that 
$$\sum_{i=0}^1 (-1)^i \mathrm{Tr} \bigl(\alpha, H^i(C,\mathcal{O}_C)\bigr)=1-(\tilde \rho + 2\tilde \rho^2) 
=\frac{c_{\tilde \rho}}{1-\tilde \rho}+\frac{c_{\tilde \rho^2}}{1-\tilde \rho^2} 
$$
giving $c_{\tilde \rho}=4$. 

Elementary Galois theory tells us that a cyclic triple cover of $\PP^1$, with coordinate $x$, can be given by an equation $y^3=f(x)$ with cover given by  $(x,y) \mapsto x$, together with the automorphism 
$\alpha: (x,y) \mapsto (x,\tilde \rho y)$, and where $f(x)$ does not contain any irreducible factor to the power larger than two. The space $H^0(C,\Omega^1_C)$ of regular differentials of $C$ is generated by $\{dx/y, dx/y^2, xdx/y^2\}$ and the eigenvalues of $\alpha$ are thus of the right form, $(\tilde \rho^2,\tilde \rho,\tilde \rho)$ on the given basis. The ramification point for which the action of $\alpha$ is by $\tilde \rho^2$ is necessarily defined over $K$ and using a projective transformation we put it in infinity. The polynomial $f(x)$ should then have four distinct roots over $\bar K$ and since the action on the ramification is by $\tilde \rho$ the polynomial should be square-free. 

\begin{remark} \label{rmk-type} Doing the above in more generality, we begin with a covering $y^3=f_1f_2^2$, with square-free polynomials $f_1,f_2$. If the field is not too small we use a projective transformation to make sure that the point over infinity is not ramified, i.e that $3$ divides $\deg(f)=2(g+2)-\deg(f_1)$. As above, we then find that the action of $\alpha$ on $H^0(C,\Omega^1_C)$ is of type $\bigl((g-1+d)/3,(2g+1-d)/3 \bigr)$.
\end{remark}

Let $P_{d}(K) \subset K[x]$ be the subset consisting of polynomials of degree $d$ with non-zero discriminant. To each $f \in P_4(K)$ we associate the cyclic triple cover $(C_f,\alpha)$ given by the equation $y^3=f(x)$. The isomorphisms between pairs of the form $(C_f,\alpha)$ are given by $(x,y) \mapsto (ax+b,cy)$ with $a,c \in K^*$ and $b \in K$. 
The groupoid of pairs  $(C_f,\alpha)$ with $f \in P_4(K)$ is equivalent to the groupoid $\cX_{1,\Gamma}(K)$.
\begin{subsubsection}{Ramification and $(1-\alpha)$-torsion} \label{sec-ramalpha}
The $3$-torsion group of an abelian threefold $A$ in $\cX'_{\GaSq}(K)$ is isomorphic to $(\ZZ/3\ZZ)^6$ over $\bar K$. We have that $3=(1-\alpha)(1-\alpha^2)$ and the $(1-\alpha)$-torsion group (which equals the $(1-\alpha^2)$-torsion group) is isomorphic to $(\ZZ/3\ZZ)^3$ over $\bar K$. This is a totally isotropic subspace of the $3$-torsion group with respect to the Weil pairing and an isomorphism $A[1-\alpha] \cong (\ZZ/3\ZZ)^3$ is acted upon by $\mathrm{O}(h,\ZZ/3\ZZ)$, the group of orthogonal matrices with coefficients in $\ZZ/3\ZZ$ that respect the hermitian form $h$ from Section~\ref{PMG}. This group is isomorphic to $\Gamma/\Gamma[\sqrt{-3}] \cong \fS_4 \times \mu_2$, c.f. \cite[p. 153]{Fi}. 

Take a pair $(C_f,\alpha)$ as in the previous section and let $p_1,p_2,p_3,p_4$ be the ramification points (not necessarily defined over $K$) where $\alpha$ acts by $\tilde \rho$ and let $p_0$ be the point where it acts by $\tilde \rho^2$. The degree $0$ divisors $\beta_1=p_1-p_0$, $\beta_2=p_2-p_0$ and $\beta_3=p_3-p_0$ are points on the Jacobian of our curve which are fixed by $\alpha$, that is, they belong to $J(C_f)[1-\alpha]$. 

The morphism 
$$ \phi: ({\ZZ}/3{\ZZ})^3 \rightarrow J(C_f)[1-\alpha]$$
defined by $(c_1,c_2,c_3) \mapsto c_1 \beta_1+c_2\beta_2+c_3\beta_3$ is an  
isomorphism. A non-trivial element in the kernel of $\phi$ can easily be rearranged using 
$$\mathrm{div}(y)=\sum_{i=1}^4 (p_i-p_0)$$
and the relations $3p_i\sim 3p_0$ for all $i$, to a relation of the form $p_i+p_j \sim p_k+p_l$ for some $i,j,k,l$. This relation implies that the curve is hyperelliptic which is not possible, see Remark~\ref{hyperell} below. If we order the points $p_1$, $p_2$, $p_3$, $p_4$ then we have an action of $\fS_4$ on $J(C_f)[1-\alpha]$. This action corresponds to the group $\mathrm{SO}(h,\ZZ/3\ZZ)$. And if we add the action $\beta_i \mapsto -\beta_i$ we get the whole $\mathrm{O}(h,\ZZ/3\ZZ)$. 

The groupoid $\cX_{1,\Gamma[\sqrt{-3}]}(K)$ is equivalent to the groupoid of pairs $(C_f,\alpha)$ together with an ordering of the four ramification points where $\alpha$ acts by $\tilde \rho$.

\begin{remark}\label{hyperell}
A smooth curve $C$ of genus $3$ which is a cyclic triple cover of ${\PP}^1$ of degree $3$ is not hyperelliptic. Indeed, if $\tau$ denotes the hyperelliptic involution, then $\tau$ and $\alpha$ commute, and $\alpha$ permutes the eight fixed points of $\tau$, and vice versa. In particular, since $\alpha$ has an unique fixed point $p_0$, where it acts by $\tilde \rho^2$, we see that $\tau$ also has to fix $p_0$. 
Since $8 \equiv_3 2 $, the action of $\alpha$ must have at least one more fixed point among the fixed points of $\tau$, say $p_1$. Then $p_0-p_1$ defines a point in the kernel of
both the endomorphism $2$ and $(1-\alpha)(1-\alpha^2)=3$ of the Jacobian of $C$, hence $p_0 \sim p_1$,
a contradiction.
\end{remark}
\end{subsubsection}
\end{subsection}
\subsection{Smooth curves of genus $2$} \label{sec-smooth2}
Arguing as in Section~\ref{sec-smooth3}, a curve $C$ of genus $2$ together with an automorphism $\alpha$ of order $3$ inducing a cyclic cover of $\PP^1$ can be given on the form $y^3=f_1(x)f_2(x)^2$ with $f_1$ and $f_2$ being relatively prime square-free polynomials for which $(\deg f_1,\deg f_2)=(2,2)$, $(2,1)$ or $(1,2)$ and where $\alpha: (x,y) \mapsto (x,\tilde \rho y)$ has eigenvalues $(\tilde \rho,\tilde \rho^2)$ on $H^0(C,\Omega^1_C)$. Denote the curve corresponding to $f_1$ and $f_2$ by $C_{f_1,f_2}$. The isomorphisms between pairs of the form $(C_{f_1,f_2},\alpha)$ are generated by $\mathrm{PGL}_2(K)$ acting on $x$ together with $y \mapsto ay$ for any $a \in K^*$.

Let $p_1$, $p_2$ be the points above the roots of $f_1$ and $q_1$, $q_2$ be the points above the roots of $f_2$. They are related by $p_1+p_2 \sim q_1+q_2$. The divisors $\beta_1=p_1-p_2$ and $\beta_2=q_1-q_2$ give a basis of the $(1-\alpha)$-torsion of the Jacobian of $C_{f_1,f_2}$. 

The groupoid $\cX^{(2)}_{\Gamma[\sqrt{-3}]}(K)$ 
is equivalent to the groupoid of pairs $(C_{f_1,f_2},\alpha)$ together with an ordering of the pair $p_1$ and $p_2$ and the pair $q_1$ and $q_2$ of ramification points of $\alpha$. The group  $(\ZZ/2\ZZ)^2$ acts on this moduli space by switching $p_1$ and $p_2$ respectively $q_1$ and $q_2$. 
\subsection{Elliptic curves} \label{sec-smooth1}
Let $E$ be an elliptic curve with an automorphism $\alpha$ of order $3$ inducing a cyclic cover of $\PP^1$ and let the origin be a fixed point of $\alpha$. Arguing again as in Section~\ref{sec-smooth3}, such a curve can be given on the form $y^3=f(x)^t$ with $t=1$ or $2$, $f$ a square-free polynomial of degree~$2$, $\alpha: (x,y) \mapsto (x,\tilde \rho y)$, and where the origin is placed over infinity. The action of $\alpha$  on $H^0(C,\Omega^1_C)$ has eigenvalue $\tilde \rho^t$. For a fixed $t$, the isomorphisms between pairs of the form $(C_{f^t},\alpha)$ are generated by $x \mapsto ax+b$ and $y \mapsto cy$ for any $a,c \in K^*$ and $b \in K$.

If $K$ is algebraically closed then the (coarse) moduli space $\cX^{(1)}_{\Gamma}(K)$, which corresponds to $t=1$, 
consists of one point. This point can be represented by $f(x)=x^2+x$.

Let $r_1$, $r_2$ be the points above the roots of $f$. The divisor $r_1-r_2$ gives a basis for the $(1-\alpha)$-torsion points. 
The groupoid $\cX^{(1)}_{\Gamma[\sqrt{-3}]}(K)$ is equivalent to the groupoid of pairs $(C_{f},\alpha)$ together with an ordering of the pair $r_1$ and $r_2$. The group $\ZZ/2\ZZ$ acts on this moduli space by switching $r_1$ and $r_2$. This is also the effect of the involution $-1$ on $E$, and so if $K$ is a algebraically closed field then the (coarse) moduli space $\cX^{(1)}_{\Gamma[\sqrt{-3}]}(K)$ also consists of one point.
\subsection{A smooth genus $2$ curve joined with an elliptic curve} \label{sec-strata2}
A curve $C$ in $\cX_{2,\Gamma}(K)$ consists of a curve $C_{f_1,f_2}$ in $\cX^{(2)}_{\Gamma}(K)$ and $C_f$ in $\cX^{(1)}_{\Gamma}(K)$ joined at a ramification point of each curves. The ramification point of $C_{f_1,f_2}$ should be above a root of $f_2$ and the ramification point of $C_{f}$ should be above infinity (isomorphisms are induced by the ones of the individual curves that fixes these points). This leaves us, as we want, with four fixed points of $\alpha$ with eigenvalue $\tilde \rho$ acting on the tangent space, and one with $\tilde \rho^2$. 
It is then straightforward to see that the groupoid $\cX_{2,\GaSq}(K)$ is given by adding an ordering of the four fixed points. 
\subsection{Triples of elliptic curves} \label{sec-strata3}
A curve in $\cX_{3,\Gamma}(K)$ has three components, two curves $C_{f_1}$, $C_{f_2}$ in $\cX^{(1)}_{\Gamma}(K)$ together with a curve of the form $C_{f_3^2}$ corresponding to a curve $C_{f_3}$ in $\cX^{(1)}_{\Gamma}(K)$. The two curves $C_{f_1}$, $C_{f_2}$ are joined 
to the curve $C_{f_3^2}$ at a ramification point over infinity and at a ramification point over one of the roots of $f_3$. This leaves us again with the wanted four fixed points of $\alpha$ with eigenvalue $\tilde \rho$ acting on the tangent space, and one with $\tilde \rho^2$, 
and the groupoid $\cX_{3,\GaSq}(K)$ is given by adding an ordering of the four fixed points. 
\begin{subsection}{Stable admissible covers}\label{sec-stable} 
 Let us briefly discuss a compactification, of the moduli space $\cX_{\GaSq}$
 using degenerations of cyclic covers. 

Let $\tilde \cX_{\GaSq}$ be the moduli space defined over $\ZZ[\rho,1/3]$ of
stable marked admissible $\ZZ/3$-covers of (stable) curves of genus $0$
with action of type $(2,1)$. For the definition of admissible covers, see
\cite{H-M} (and compare with \cite{Achter-Pries}).
An element of $\tilde \cX_{\GaSq}(K)$ is
a nodal curve $C$ over $K$ of genus $3$ with an action $\alpha$ of an
automorphism of order $3$ with $C/\alpha$ isomorphic to a curve $P$ of
genus $0$, stably marked by the ramification points
$p_0, p_1,\ldots,p_4$ of the cover $C \to P$ and such that 
$\alpha$ acts by $\tilde \rho$ on the tangent space of the points above
$p_1,\ldots,p_4$ and by $\tilde \rho^2$ on the point above $p_0$.

There is a morphism 
$$
\tilde \cX_{\GaSq} \to \overline{\mathcal M}_{0,1+4} \, ,
$$
with $\overline{\mathcal M}_{0,1+4}$ 
the moduli space of $(1+4)$-pointed 
genus $0$ curves. 

Note that $\overline{\mathcal M}_{0,1+4}$ has a stratification with five strata
according to the topological type of the genus $0$ curve: a $2$-dimensional
open stratum ${\mathcal M}_{0,1+4}$, two $1$-dimensional strata corresponding
to a join $P_1,P_2$ of two ${\PP}^1$'s intersecting in a point with
marked points $p_0,p_1,p_2$ on $P_1$ and $p_3,p_4$ on $P_2$, or
$p_0,p_1$ on $P_1$ and $p_2, p_3,p_4$ on $P_2$ and finally two
strata each consisting of one point corresponding to a
linear chain of three ${\PP}^1$'s, $P_1$, $P_2$, $P_3$ with marked points
$p_1,p_2$ on $P_1$, $p_0$ on $P_2$ and $p_3,p_4$ on $P_3$, or 
$p_0,p_1$ on $P_1$, $p_2$ on $P_2$ and $p_3,p_4$ on $P_3$
(all described up to the action of $\fS_4$). 

This will induce a stratification of $\tilde \cX_{\GaSq}$. The first three cases
above correspond to $\cX_{1,\GaSq}$, $\cX_{2,\GaSq}$ and $\cX_{3,\GaSq}$
respectively.

The fourth strata is $1$-dimensional and the curves it parametrizes consist of an 
elliptic curve $C_1$ with an order $3$ automorphism $\alpha_1$ with action of type
$(1,0)$ and a rational curve $C_0={\PP}^1$ with an automorphism $\alpha_0$ that acts
by $x \mapsto \rho x$, joined by identifying the three points of an
$\alpha_1$-orbit of length $3$ to the points
$1, \tilde \rho, \tilde \rho^2$ on $C_0$. There are four components depending upon 
the choice of marking of the ramification point on the rational curve. 

The fifth strata is $0$-dimensional and the curves it parametrizes consist of
 an union of an elliptic curve $C_1$ with an
order $3$ automorphism $\alpha_1$ with action of type $(1,0)$
and two ${\PP}^1$'s with automorphism $x \mapsto 1/(1-x)$, say $C_0$ and
$C_0'$, that intersect each other in $0$, $1$ and
$\infty$ such that $C_1$ and $C_0'$ are disjoint,
while $C_1$ and $C_0$ intersect in a fixed point of $\alpha_1$. 
This strata consists of twelve points depending upon 
the choice of marking of the ramification points on the two rational curves. 

Let $\cX^*_{\GaSq}$ denote the Satake-Baily-Borel compactification of 
$\cX'_{\GaSq}$ defined over $\ZZ[\rho,1/3]$, and so, 
$$\cX^*_{\GaSq}(\CC) \cong X^*_{\GaSq}.$$
Sending a curve to its (generalized) Jacobian gives a morphism, 
\begin{equation} \label{eq-tildestar}
  \tilde \cX_{\GaSq}  \to \cX^*_{\GaSq}.  
\end{equation}
The Jacobians over the fourth and fifth strata become extensions of an 
elliptic curve with the multiplicative group $\mathbb G_{m,F}$, and these 
strata will furthermore constitute a $\PP^1$-bundle over the four cusps of 
$\cX^*_{\GaSq}$.
\end{subsection}
\end{section}

\begin{section}{Characteristic polynomials of Frobenius}
In this section we find properties of the characteristic polynomial of Frobenius acting on $\rho$-eigenspaces of the first \'etale cohomology group of a cyclic triple cover of the projective line. 

\begin{subsection}{Notation for primes, generators and finite fields} \label{sec-notation}
Let $k=\FF_q$ always denote a finite field with $q=p^r$ elements with $q \equiv_3 1$. 
For any $n \geq 1$ let $k_n=\FF_{q^n}$, a degree $n$ extension of $k$.

If $p \equiv_3 1$, choose a third root of unity 
$\tilde \rho$ in $\FF_p$. This gives us a third root $\tilde{\rho}$
in any extension field $k={\FF}_q$ for $q=p^r$. The choice of $\tilde{\rho}$
determines a generator $a_{\pp_p}+b_{\pp_p}\rho$ of an ideal $\pp_p$ of 
norm $p$, namely let $a_{\pp_p},b_{\pp_p}$ be the unique pair of 
integers such $a_{\pp_p}^2-a_{\pp_p}b_{\pp_p}+b_{\pp_p}^2=p$, 
$a_{\pp_p} \equiv_3 1$, $b_{\pp_p} \equiv_3 0$  and such that an (hence any)
 isomorphism between $\ZZ[\rho]/\pp_p$ and $k=\FF_p$, sends $\rho$ to $\tilde \rho$. Define the integers $a_{\pp_q}$, $b_{\pp_q}$ by the equation $a_{\pp_q}+b_{\pp_q}\rho=(a_{\pp_p}+b_{\pp_p}\rho)^r$ and define the ideal
$\pp_q=(a_{\pp_q}+b_{\pp_q}\rho)$.

For any $p \equiv_3 2$ choose an arbitrary third root of unity $\tilde \rho$ in $\FF_{p^{2}}$. For any even $r \geq 1$ choose an embedding of $\FF_{p^{2}}$ in $k=\FF_{p^{r}}$ and let the chosen third root of unity of $\FF_{p^{r}}$ be the one coming from $\FF_{p^{2}}$. Note that it is presence of the automorphism $x \mapsto x^p$ of $\FF_{p^{2}}$ that makes that these choices do not matter for the later results for the moduli spaces, see Proposition~\ref{prop-p2}. The ideal $\pp_p=(p)$ is prime in $\ZZ[\rho]$ and also here we 
also choose a generator $a_{\pp_p}+b_{\pp_p}\rho$ such that  $a_{\pp_p} \equiv_3 1$, $b_{\pp_p} \equiv_3 0$, 
namely $a_{\pp_p}=-p$ and $b_{\pp_p}=0$. Define also $a_{\pp_q}=(-p)^r$, $b_{\pp_q}=0$ and the ideal $\pp_q=(a_{\pp_q}+b_{\pp_q}\rho)$. 
\end{subsection}
\begin{subsection}{The characteristic polynomial} \label{sec-charpolfrob}
Let $\chi$ denote the third power residue symbol, that is, 
if $a\in k^*$ then $\chi(a)=\rho^i$ where $\tilde \rho^i =a^{\frac{q-1}{3}}$, and $\chi(0)=0$.
Let $C'_f$ be a cyclic triple cover of the projective line given by an 
equation of the form $y^3=f(x)$, where $f$ is a cube-free polynomial with coefficients in $k$. With $f(\infty)$ we mean the leading coefficient of $f$ if $\deg f \equiv_3 0$, and $0$ otherwise. 
Let $g$ be the genus of $C_f$.
Put $C_f:=C'_f \otimes_k \bar k$ and let $F_q$ denote the geometric Frobenius morphism acting on $C_f$. The number of points over $k$ of $C_f$ equals 
$$|C_f(k)|=|C_f^{F_q}|=\sum_{a \in \PP^1(k)} \bigl(1+\chi(f(a))+\overline{\chi(f(a))} \bigr).$$
Let $H^i_c$ denote compactly supported $\ell$-adic \'etale cohomology. The Lefschetz trace formula (see \cite[Th.~3.2]{SGA}) tells us that  
$$|C_f^F|=\sum_{i=1}^2 (-1)^i \Tr(F_q,H_c^i(C_f,\QQl))$$
and so 
$$a_1(C_f):=\Tr(F_q,H_c^1(C_f,\QQl))=-\sum_{a \in \PP^1(k)} \bigl(\chi(f(a))+\overline{\chi(f(a))} \bigr). $$

Let $\alpha$ be the automorphism of $C_f$ given by $(y,x) \mapsto (\tilde \rho y,x)$, which 
commutes with Frobenius. We find that 
$$|C_f^{F_q \circ \alpha^i}|=\sum_{a \in \PP^1(k)} \bigl(1+\rho^i \,\chi(f(a))+\overline{\rho^i \, \chi(f(a))} \bigr).$$ 

The automorphism $\alpha$ splits $H^j_{c}(C_f,\QQl)$ into $\rho^i$-eigenspaces $H^j_{c}(C_f,\QQl)^{\rho^i}$. The projection formula gives,  
$$\sum_{j=0}^2 (-1)^j \, \mathrm{Tr}(F_q,H^j_{c}(C_f,\QQl)^{\rho^i})=\frac{1}{3} \, \sum_{k=0}^2 \rho^{-ik}\,\, |C_f^{F_q \circ \alpha^k}|.$$ 

The $1$-dimensional cohomology groups $H^0_c$ and $H^2_c$ are $\alpha$-invariant.  
Since the quotient by $\alpha$ has genus $0$, $H^1_c$ has no $\alpha$-invariant part. 
These two things can also be deduced using the Lefschetz trace formula for $\alpha$ together 
with the fact that $\alpha$ has $g+2$ fixed points (using the Hurwitz formula).
It follows that  
$$a_{1,\rho^i}(C_f):=\mathrm{Tr}(F_q,H^1_{c}(C_f,\QQl)^{\rho^i})=-\sum_{a \in \PP^1(k)} \chi(f(a))^i.$$ 
for $i=1,2$.

Let $\alpha_1(C_f),\ldots,\alpha_g(C_f)$ be the eigenvalues of 
Frobenius acting on the $g$-dimensional vector space 
$H^1_{c}(C_f,\QQl)^{\rho}$ and denote the characteristic 
polynomial of Frobenius by $\mathfrak{ch}_{\rho}(C_f)$. Let  
$e_i$ denote the $i$th elementary symmetric polynomial in $g$ variables. Note then that  
 $$e_1 \bigl(\alpha_1(C_f),\ldots,\alpha_g(C_f)\bigr)=a_{1,\rho}(C_f).$$
and that 
$$\mathfrak{ch}_{\rho}(C_f)=\sum_{i=0}^g x^{g-i}(-1)^i e_i \bigl(\alpha_1(C_f),\ldots,\alpha_g(C_f)\bigr).$$

Since $\alpha_i(C_f)\bar \alpha_i(C_f)=q$ for $1 \leq i \leq g$, we immediately get that 
\begin{equation} \label{eq-duality} e_{ g-i}q^i=e_g \bar e_{i}
\qquad \text{for $1 \leq i \leq g$.} \end{equation}
\begin{subsubsection}{Characteristic polynomials of elliptic curves} \label{sec-charpol-ell}
We will first consider the elliptic curves in some detail. 
Assume that $p \neq 2$. If $\gamma$ is a generator of $k^*$, then 
$y^3=x^2+\gamma^i$ for $i=0, \ldots, 5$ are representatives of the six $k$-isomorphism classes of curves in 
$\cX^{(1)}_{\Gamma}(k)$. 

Let $\nu$ denote the second power residue symbol. 
Using Jacobi sums (see for instance \cite[Chapter 8]{IR}) we have 
\begin{equation}  \label{eq-charell} a_{1,\rho}(C_{x^2+D})=\nu(-4D)\chi(-4D) (a+b\rho)
\end{equation}
for any $D \in k^*$. Note that $-4D$ is the discriminant of the polynomial $x^2+D$.

Say that $\sigma$ switches the two marked ramification points of a curve in $\cX^{(1)}_{\GaSq}(\bar k)$. 
The fixed points of Frobenius composed with $\sigma$ acting on $\cX^{(1)}_{\Gamma}(\bar k)$ 
are the curves $y^3=x^2+D$ such that $\nu(-4D)=1$.

Assume now that $p=2$. If $\gamma$ is a generator of $k^*$, then the following are representatives of the 
six $k$-isomorphism classes of curves in $\cX^{(1)}_{\Gamma}(k)$, $x^2+\gamma^ix$ for $i=0,1,2$ and 
$x^2+\gamma^ix+\delta_i$ for $i=0,1,2$ where $\delta_i$ is any element of $k$ such that the polynomial is irreducible. 

For any $a \in k^*$ and $b \in k$, define $\nu_a(b)$ to be $1$ if the equation $x^2+ax+b=0$ has 
two solutions in $k$, and $-1$ otherwise. Note first that $t^2+t+1$ has two roots, say $\alpha$ and $\beta$, 
in $k$ and that $t^2+at+a^2$ has roots $a \alpha$ and $a \beta$, so in particular $\nu_a(a^2)=1$.
Using that the Jacobi sum $J(\chi,\chi)$ equals $-a-b\rho$ we get 
 \begin{equation} \label{eq-charell2} a_{1,\rho}(C_{x^2+ax+b})=-\sum_{w} \chi(w^2+aw+b) 
= \nu_a(b) \chi(a^2)(a+b\rho).
 \end{equation}
Note that $a^2$ is the discriminant of the polynomial $x^2+ax+b$. 
The fixed points of Frobenius composed with $\sigma$ acting on $\cX^{(1)}_{\Gamma}(\bar k)$ as above  
are the curves $y^3=x^2+ax+b$ such that $\nu_a(b)=1$.

Theorem~\ref{thm-det} below is a generalization of the formulas \eqref{eq-charell} and \eqref{eq-charell2}, which go back to Gauss. 
\end{subsubsection}
\begin{subsubsection}{The characteristic polynomial modulo $1-\rho$} 
Let $p_1,\ldots,p_{g+2}$ be the roots of $f=f_1f_2^2$. 
The elements $(v_1,\ldots,v_{g})$ with $v_i=p_i-p_{i+1}$ form a basis of the $g$-dimensional $\ZZ/3\ZZ$-vector space  $J(C_f)[1-\alpha]$. Using the Tate module of $J(C_f)$ we see that the action of Frobenius on $H^1_{c}(C_f,\QQl)^{\rho}$ modulo $(1-\rho)$ is equal to the action of Frobenius on $J(C_f)[1-\alpha]$. 
Let $\mathfrak{ch}_{\rho}(C_f)_{\rho=1} \in (\ZZ/3\ZZ)[x]$ denote the polynomial $\mathfrak{ch}_{\rho}(C_f)$ modulo $(1-\rho)$. This polynomial is then equal to the characteristic polynomial of Frobenius acting on $J(C_f)[1-\alpha]$. 

Say that the Frobenius $F_q$ induces a permutation $\sigma \in S_{g+2}$, that has $c_i(\sigma)$ cycles of length $i$, on the set of points $\{p_1,\ldots,p_{g+2}\}$. We will now describe the $g \times g$-matrix $A_{\sigma}$ induced by $F_q$ acting on the basis $(v_1,\ldots,v_{g})$ with $v_i=p_i-p_{i+1}$. Put $d_1=\deg f_1$, $d_2=\deg f_2$ and $\mathfrak{ch}(A_{\sigma})=\det(xI-A)$. 

Let us first handle the cases for which $c_1(\sigma) \geq 1$. In this case we can look at any $\sigma$ without it having to correspond to an actual curve. 
By reordering, we can assume that $g+2$ is fixed by $\sigma$. If $\sigma=(1,\ldots,g+1)(g+2)$ then $F_q(v_i)=v_{i+1}$ for $i \leq g-1$ 
and $F_q(v_g)=-v_{1}-v_{2}-\ldots -v_{g}$. We find that $\mathfrak{ch}(A_{\sigma})=(x^{g+1}-1)/(x-1)$. Say now that we have computed $A_{\tau}$ for some $\tau$ with $c_1(\tau) \geq 1$. If $\sigma$ consists of a cycle $(1,\ldots,h)$ followed by $\tau$ reordered such that $j$ is replaced by $j+h$, then $F_q(v_i)=v_{i+1}$ for $i \leq h-2$, $F_q(v_{h-1})=-v_{1}-v_{2}-\ldots -v_{h}$ and $F_q(v_{h})=v_{1}+v_{2}+\ldots +v_{h+1}$. If we define $B$ to be the $h \times h$-matrix given by $w_i \mapsto w_{i+1}$ for $i \leq h-2$, $w_{h-1} \mapsto -w_{1}-w_{2}-\ldots -w_{h}$ and $v_{h} \mapsto w_{1}+w_{2}+\ldots +w_{h}$ on a basis $(w_1,\ldots,w_h)$, then 
$\mathfrak{ch}(A_{\sigma})=\mathfrak{ch}(A_{\tau})\mathfrak{ch}(B)=\mathfrak{ch}(A_{\tau})(x^h-1)$.  
This describes, by induction, the structure of $A_{\sigma}$ for any $\sigma$ with $c_1(\sigma) \geq 1$. 

\begin{example} For $\sigma=(1)(23)(4567)(8)$ we have 
  $$ A_{\sigma}=
  \left(  \begin{matrix}
    1 & 1 & 0 & 0  & 0  & 0 \\
    0 &-1 & 0 & 0  & 0  & 0 \\
    0 & 1 & 1 & 1  & 0  & 0 \\
    0 & 0 & 0 & 0  & 1  & 0 \\
    0 & 0 & 0 & 0  & 0  & 1 \\
    0 & 0 & 0 & -1 & -1 & -1 \\
    \end{matrix} \right)
  $$
  \end{example}

Let us now handle the cases for which $c_1(\sigma)=0$. Reorder the points $p_1,\ldots,p_{g+2}$ such that the roots of $f_1$ come before the roots of $f_2$ and put $v_{g+1}=p_{g+1}-p_{g+2}$. 

Note first that $3p_i \sim 3p_j$ for any $i,j$, and then that on the one hand 
$$\mathrm{div}_0(y)=\sum_{i=1}^{d_1}p_i+\sum_{i=1+d_1}^{d_1+d_2} 2p_i $$
and on the other
$$\mathrm{div}_{\infty}(y) \sim (d_1+2d_2)p_i,$$
for any $i$. If $d_1 \equiv_3 0$ then this can be used to give the relation
$$\sum_{i=1}^{d_1}p_i+\sum_{i=1+d_1}^{d_1+d_2} 2p_i \sim \sum_{i=1}^{d_1/3} 3p_{3i-1}+\sum_{i=1+d_1/3}^{d_2/3} 6p_{3i-1}$$
from which it follows that 
$$\sum^{d_1/3}_{i=1} (v_{3i-2} - v_{3i-1})+ \sum_{i=1+d_1/3}^{(d_1+d_2)/3} (-v_{3i-2} + v_{3i-1})=0.$$

Similarly, if $d_1 \equiv_3 1$ then
$$\sum^{(d_1-1)/3}_{i=1} (v_{3i-2} - v_{3i-1})+v_{d_1}+\sum_{i=1+(d_1-1)/3}^{(d_1+d_2-2)/3} (-v_{3i} + v_{3i+1})=0,$$
and if $d_1 \equiv_3 2$ then
$$
\sum^{(d_1-2)/3}_{i=1} (v_{3i-2} - v_{3i-1})+v_{d_1-1}-v_{d_1}-v_{d_1+1}
+  \sum_{i=2+(d_1-2)/3}^{(d_1+d_2-1)/3} (-v_{3i-1} + v_{3i})=0.
$$

We will now use the same reasoning as above. The difference is that if $\sigma$ contains the 
cycle $(s,\ldots,s+t-1)(s+t,s+t+1,\ldots,g+2)$ then $F_q(v_g)=v_{g+1}$ if $s+t \leq g$ and 
$F_q(v_g)=v_{s}+\ldots +v_{g+1}$ if $s+t=g+1$. We can express $v_{g+1}$ in terms of 
$v_1,\ldots,v_g$ using the formulas above, but we find that only the coefficients of 
$v_{t-1},\ldots,v_g$ will affect $\mathfrak{ch}(A_{\sigma})$. Using that $\sigma$ necessarily 
permutes the roots of $f_1$ and $f_2$ respectively, we find that the contribution of the 
cycles $(s,\ldots,s+t-1)(t,t+1,\ldots,g+2)$ to $\mathfrak{ch}(A_{\sigma})$ equals 
$(x^{t}-1)(x^{g+2-s-t}-1)/(x-1)^2$. 

 \begin{example}
  For $\sigma_1=(123)(456)$ 
and  $\sigma_2=(12)(34)(56)$ we have 
  $$ A_{\sigma_1}=
  \left(  \begin{matrix}
    0 & 1 & 0& 0\\
      -1 &-1 & 0&0\\
      1 & 1 & 1 & 1 \\
      1 & -1 & 0 & 1 \\
    \end{matrix} \right)\, ,
\qquad
  A_{\sigma_2}=
  \left(  \begin{matrix}
    -1 & 0 & 0 & 0\\
      1 &1 & 1 & 0\\
      0 & 0 & -1 & 0 \\
      1 & -1 & 1 & -1 \\
    \end{matrix} \right)
  $$
   \end{example}

Summing up, we find the following. 

\begin{theorem} \label{thm-charmodrho} Let $C_f$ be a tricyclic cover of genus $g$. If Frobenius induces a permutation $\sigma$ on the set of $g+2$ ramification points with $c_i(\sigma)$ cycles of length $i$ then we have
$$\mathfrak{ch}_{\rho}(C_f)_{\rho=1}=\mathfrak{ch}(A_{\sigma})=\frac{1}{(x-1)^2}\prod_{i=1}^{g+2} (x^i-1)^{c_i(\sigma)}
$$
as polynomials in $(\ZZ/3\ZZ)[x]$.
\end{theorem}

\begin{remark} Computing 
  $$a_{n,\rho}(C_f):=-\sum_{a \in \PP^1(k_n)} \chi(f(a))=\sum_{i=1}^g \alpha_i(C_f)^n \in \ZZ[\rho]$$
  modulo $1-\rho$ is more straightforward since
  $$-\sum_{a \in \PP^1(k_n)} \chi(f(a))=-(q^n+1-r_i)=1+r_n \mod (1-\rho),$$
  where $r_n$ is the number of roots of $f$ defined over $k_n$.
  \end{remark}

\end{subsubsection}

\begin{subsubsection}{The determinant of Frobenius}
Say that $\alpha$ has $s$ eigenvalues equal to $\tilde \rho$ when acting on $H^0(C,\Omega^1_C)$. It then 
follows from \cite[Th\`eoreme 1]{Gi} 
that $e_g$ generates the ideal $\mathfrak p_q^s \cdot (\overline{\mathfrak p_q})^{g-s}$ and hence 
\begin{equation}\label{giraud}
e_g\bigl(\alpha_1(C_f),\ldots,\alpha_g(C_f)\bigr)=(-1)^{j_1}\rho^{j_2} (a_{\mathfrak p_p}+b_{\mathfrak p_p}\rho)^{rs}(a_{\mathfrak p_p}+b_{\mathfrak p_p}\rho^2)^{r(g-s)}
\end{equation}
for some integers $j_1$ and $j_2$.

\begin{theorem} \label{thm-det} For any polynomial $h$, let $D(h)$ denote the discriminant of $h$, 
and $\epsilon(h)$ the number of irreducible factors (over $k$) of $h$.
If we assume that $f=f_1 f_2^2$ and $3|\deg f$ then  
\begin{equation}\label{eq-det} \frac{e_g\bigl(\alpha_1(C_f),\ldots,\alpha_g(C_f)\bigr)}{(a_{\mathfrak p_p}+b_{\mathfrak p_p}\rho)^{rs}(a_{\mathfrak p_p}+b_{\mathfrak p_p}\rho^2)^{r(g-s)}}=
(-1)^{g+\epsilon(f_1)+\epsilon(f_2)} \chi \bigl(D(f_1)\bigr) \overline{\chi \bigl(D(f_2) \bigr)}. \end{equation}
\end{theorem}

\begin{remark} Note that if $p \neq 2$, then by  Stickelberger's theorem 
(see \cite[Thm 1.3]{Cox} or \cite{Dalen})
$$(-1)^{g+\epsilon(f_1)+\epsilon(f_2)}= \nu \bigl(D(f_1)D(f_2) \bigr)= \nu \bigl(D(f_1) \bigr)\overline{ \nu \bigl(D(f_2) \bigr)},
$$
where $\nu$ denotes the second power residue symbol.
\end{remark}
\begin{proof} Since, $e_g\bigl(\alpha_1(C_f),\ldots,\alpha_g(C_f)\bigr)=(-1)^{j_1}$ modulo $(1-\rho)$, Theorem~\ref{thm-charmodrho}
tells us immediately that $j_1=g+\epsilon(f_1)+\epsilon(f_2)$.

The action of Frobenius on $H^1_c(C_f, \overline{\QQ}_{\ell})^{\rho}$ modulo
$1-\rho$ is equal to the action on $J(C_f)[1-\alpha]$. To determine $j_2$
it suffices to calculate the expression \eqref{giraud} modulo $3$ since it equals
$(-1)^{j_1}(1-\theta)^{j_2}$ in $\cO_F/(3)\cong {\FF}_3[\theta]$ with 
$\theta=1-\rho \, (\bmod \, 3)$.
That means that it suffices to calculate the determinant of Frobenius on the
${\FF}_3[\theta]$-module $J(C_f)[3]$. In view of the exact sequence
$$
0 \to J(C_f)[1-\alpha] \to J(C_f)[3] {\buildrel 1-\alpha \over
\longrightarrow} J(C_f)[1-\alpha] \to 0
$$
and the fact that $J(C_f)[1-\alpha]$ is isotropic for the Weil pairing,
as kernel of an endomorphism, 
we see that the action of $\alpha$
on $C_f$ induces a cyclic $\mu_3$-action on the three
possibilities for this determinant. (If we lift our
abelian variety together with $\alpha$ to ${\CC}$ then this action
corresponds to the action of ${\rm diag}(1,1,\rho) \in \Gamma[\sqrt{-3}]$.)
The  determinant of Frobenius on $J(C_f)[3]$
is determined up to a third root of $1$ by the level structure
$J(C_f)[1-\alpha] \sim (\cO_F/(1-\alpha))^3$. 

The moduli stack of triples $(J,l,d)$ with $J$ a Jacobian of a 
cyclic triple cover $C_f \to {\PP}^1$ of signature $(s,g-s)$
with a level $(1-\alpha)$-structure $l$ on $J(C_f)[1-\alpha]$ and the determinant $d$
of the cohomology modulo $3$, is a threefold \'etale cover of the moduli
stack of tuples $(J,l)$. It is \'etale since
the ramification points of the cover $C_f \to {\PP}^1$ determine the
level $(1-\alpha)$-structure and $\alpha$ then induces the $\mu_3$-action.
This degree three cover extends to the appropriate 
moduli stacks (Picard modular stacks, see
\ref{sec-stacks}) of principally
polarized abelian varieties with level structure.

In the case of $g=3$ and the covers considered in subsection
\ref{sec-smooth3} this \'etale cover is given by the $\mu_3$-cover
$\mathcal{X}_{\Gamma_1[\sqrt{-3}]}
\to \mathcal{X}_{\Gamma[\sqrt{-3}]}$ which is \'etale
outside the locus where the discriminant of $f$ vanishes. 
Indeed, we know that this cover
$\mathcal{X}_{\Gamma_1[\sqrt{-3}]}
\to \mathcal{X}_{\Gamma[\sqrt{-3}]}$ is the cover defined over $\mathcal{O}_{F}$ by
equation \eqref{pmf-zeta}, see also Prop.\ \ref{scalar-rings}.
Therefore the action of Frobenius on the fibres of
$\mathcal{X}_{\Gamma_1[\sqrt{-3}]}
\to \mathcal{X}_{\Gamma[\sqrt{-3}]}$ is determined by 
the cubic character of the discriminant of $f$, hence $j_2$ is. 
The normalization of the cubic character then follows 
by checking that it satisfies the
formula of the theorem in examples for the case of genus $3$
or by checking it for abelian threefolds that are a product of 
elliptic curves. 

In the general case the threefold cover of stacks is ramified 
along the codimension $1$
locus where
the discriminant of $f$ vanishes. Therefore $j_2$ is determined
by the cubic character of the discriminant of $f$. Then we can specialize to the
case where the Jacobian $J(C_f)$ splits as a direct sum of Jacobians of
curves of lower genus to check the formula inductively starting from the
cases of $g\leq 3$.
\end{proof}
\end{subsubsection}
\end{subsection}
\end{section}

\begin{section}{Euler characteristics of $\ell$-adic local systems} \label{sec-euler}
In this section we will introduce the motivic Euler characteristic of local system on our moduli spaces, stating 
basic results, showing how the Lefschetz trace formula can be used to find cohomological information and 
presenting a formula for the integer valued Euler characteristic of any local system. 

\begin{subsection}{Hecke characters} \label{sec-hecke}
Recall the notation of Section~\ref{sec-notation}. 
Define the Hecke character $\psi$ of conductor $(3)$ for any $\pp_p$ by putting $\psi(\pp_p)=a_{\pp}+b_{\pp}\rho$. This gives a 
$1$-dimensional 'motive'  that we will denote $\LL^{1,0}$, pure of weight $1$ and Hodge type $(1,0)$ and as an $\ell$-adic 
${\rm Gal}(\bar{F}/F)$-representation, the trace of a Frobenius element $F_{\pp_p}$, corresponding to a prime ideal $\pp_p$, is 
given by $\psi(\pp_p)$. 

Recall that 
$$f_{\psi}(z)=\sum_{\alpha} \psi(\alpha)q^{N(\alpha)} \in \ZZ[q],$$
where $q=e^{2\pi i z}$ and the sum is over all integral ideals $\alpha$ prime to $(3)$ and $N(\alpha)$ is the norm of $\alpha$, is a cusp form of weight $2$ and level $\Gamma_0(27)$. Moreover we have that, 
$$f_{\psi}(z)=\eta(3z)^2 \, \eta(9z)^2
                  = q-2q^2-q^7+5q^{13}+4q^{16}-7q^{19}-5q^{25}+\dots.
$$
Similarly we define $\LL^{0,1}$ of Hodge type $(0,1)$ by using the Hecke character $\bar \psi(\pp_p):=a_{\pp_p}+b_{\pp_p}\rho^2$.
Finally, for any pair of integers $n,m$ we define 
$\LL^{n,m}$ by taking tensor products of the `motives' above. Note that $\LL^{1,1}$ becomes the usual Lefschetz motive, also 
denoted $\LL^1$. 
\end{subsection}

\begin{subsection}{Euler characteristics of $\ell$-adic local systems} \label{modloc} 
For any of our moduli spaces $\cX$ (which are stacks) introduced in Section~\ref{sec-stacks} and Section~\ref{sec-moduliofcurves} we have a universal family  $\pi:{\mathcal C} \to \cX$ and we consider the $\ell$-adic local
system ${\VV}:= R^1\pi_* {\QQl}$. This local system has rank $6$, 
where the fiber of a geometric point represented by an abelian variety $A$ 
equals the $\ell$-adic \'etale cohomology group $H^1(A,\QQl)$. 
It is provided with a non-degenerate alternating pairing 
${\VV} \times {\VV} \to {\QQl}(-1)$. 

The action of $\alpha$ gives rise to a decomposition of the base change to 
$F$ of ${\VV}$,
as a direct sum of two local systems of rank $3$ over $F \otimes \QQl$:
$$
{\VV}\otimes F = {\WW} \oplus \WW'
$$
with ${\WW}$ (respectively $\WW'$) the $\rho$-eigenspace 
(respectively the $\rho^2$-eigenspace) of $\alpha$. 

Note that we can also take $\cX(\CC)$ and define the (Betti) local system $\VV := R^1\pi_*{\QQ}$, and then the $\rho$-eigenspace $\WW$ is the same local system as the one defined in Section~\ref{rootloc}.

We define local systems $\WW_{\lambda}$ using the 
representations of ${\rm GL}(3,\CC)\times {\GG}_m$ as in Section~\ref{rootloc}. 
The multiplier defines the constant local system $F(-1)$. 

Let ${\WW}^{\vee}$ denote the $F$-linear dual. Then note that
$$(\WW_{n_1,n_2,n_3})^{\vee}\cong \WW_{n_2,n_1,-n_1-n_2-n_3}.$$
The non-degenerate pairing implies that the conjugate takes the form 
$$\WW' \cong {\WW}^{\vee} \otimes {F}(-1),$$
and so 
\begin{equation} \label{eq-dual}
\WW'_{n_1,n_2,n_3} \cong \WW_{n_2,n_1,-n_1-n_2-n_3} \otimes {F(-n_1-2n_2-3n_3)}.
\end{equation}

Let $H^{*}_c$ denote compactly supported $\ell$-adic \'etale cohomology. 
The action of the symmetric group $\fS_4 \cong \GaSq/\Gamma$ on 
$\cX_{\GaSq}$ induces an action on its cohomology groups. 
We define the Euler characteristic of the local system 
${\WW}_{\lambda}$ on $\cX_{\GaSq}\otimes \bar{F}$ in $K_0({\rm Gal}^{\fS_4}_F)$, 
the Grothendieck group of $\ell$-adic ${\rm Gal}(\bar{F}/F)$-representations equipped with an action of
$\fS_4$ by  
$$e_c(\cX_{\GaSq},{\WW}_{\lambda}):= \sum^4_{i=0} (-1)^i [H^i_c(\cX_{\GaSq} \otimes \bar{F}, {\WW}_{\lambda})].$$
Similarly, consider compactly supported Betti cohomology and define by (abuse of) the same notation
$$e_c(\cX_{\GaSq},{\WW}_{\lambda}):= \sum^4_{i=0} (-1)^i [H^i_c(\cX_{\GaSq}(\CC), {\WW}_{\lambda})]$$
in the Grothendieck group of Hodge modules equipped with an action of
$\fS_4$. 
Let $e_{c,\mu}(\cX_{\GaSq},{\WW}_{\lambda})$ correspond to a $\mu$-isotypic component of the Euler characteristics 
in the sense that 
$$e_c(\cX_{\GaSq},{\WW}_{\lambda})= \sum_{\mu \vdash 4} e_{c,\mu}(\cX_{\GaSq},{\WW}_{\lambda}) \, \mathbf s_{\mu}.$$
The statements in Section~\ref{sec-conj-main} will be called motivic, and by this we will mean that these are statements about the Euler characteristic in both these Galois groups. 

\begin{proposition} \label{prop-eultriv} The Euler characteristics fulfils the following for all $\lambda$: 
\begin{itemize}
\item[(1)] $e_c(\cX_{\GaSq},{\WW}_{\lambda})=0$ if $n_1 \not \equiv_3 n_2 $; 
\item[(2)] $e_c(\cX_{\GaSq},{\WW}_{\lambda} \otimes F(-k))=e_c(\cX_{\GaSq},{\WW}_{\lambda}) \, \LL^{k}.$ 
\end{itemize}
\end{proposition}
\begin{proof} We prove $(1)$, the proof of $(2)$ is standard. 
The automorphism $\alpha$ acts on the fibre $(\WW_{\lambda})_A$
by $\rho^{n_1+2n_2}$ for any closed point $A$ of $\cX_{\GaSq}$.
So if $n_1 \not\equiv_3 n_2$, then this action has no invariants
and hence the cohomology has to vanish. 
\end{proof}
\end{subsection}
\begin{subsection}{Traces of Frobenius} \label{sec-traces} 
Recall the notation of Section~\ref{sec-notation} and Section~\ref{sec-charpolfrob}. 
Compare the following section to the article~\cite{BFvdG} and the references therein. 
We define the (geometric) Frobenius  
$F_q \in \mathrm{Gal}(\bar k/k)$ to be the inverse of $x \mapsto x^q$.  
We choose a Frobenius element $F_{\pp_{q}} \in \mathrm{Gal}(\overline F/F)$, using an element of 
the Galois group of the $p$-adic completion of $F$ that is mapped to the Frobenius element $F_q  \in \mathrm{Gal}(\bar k/k)$. 
These Frobenii satisfy  
\begin{equation} \label{eq-frobenii}
\Tr \bigl(F_{\pp_q},e_{c,\mu}(\cX_{\GaSq} \otimes \overline{F},\WW_{\lambda})\bigr) = \Tr\bigl(F_q,e_{c,\mu}(\cX_{\GaSq}\otimes \bar k,\WW_{\lambda})\bigr) \end{equation}
and these traces are element of $\ZZ[\rho]$.  
The traces of $F_{\pp_p}$ for (almost) all unramified primes $\pp_p$ will (using a Chebotarov density argument) determine $e_{c,\mu}(\cX_{\GaSq} \otimes \overline{F},\WW_{\lambda})$ as an element of $K_0({\rm Gal}_F)$, c.f. \cite[Prop. 2.6]{Fermat}. For any element of $V \in K_0({\rm Gal}_F)$ we can define a virtual representation $\overline V$ by the property $\Tr(F_{\pp_p},\overline V)=\overline{\Tr(F_{\pp_p},V)}$. 

\begin{proposition}  \label{prop-dual} For any $\mu$ and $\lambda=n_1\gamma_1+n_2\gamma_2+n_3\gamma_3$ 
$$\overline{e_{c,\mu}(\cX_{\GaSq},{\WW}_{n_1,n_2,n_3})}=e_{c,\mu}(\cX_{\GaSq},{\WW}_{n_2,n_1,-n_1-n_2-n_3})\LL^{-n_1-2n_2-3n_3},$$
as elements of $K_0({\rm Gal}_F)$.
\end{proposition}
\begin{proof} This follows directly from equation~\eqref{eq-dual}. \end{proof}

\begin{proposition} \label{prop-p2}  For any $p \equiv_3 2$, even $r \geq 1$, $\mu$ and $\lambda$   
$$\Tr\bigl(F_q ,e_{c,\mu}(\cX_{\GaSq}\otimes \bar k,{\WW}_{\lambda})) \in \ZZ.$$ 
\end{proposition}
\begin{proof} If $p \equiv_3 2$, then the automorphism $x \mapsto x^p$ sends $\tilde \rho$ to $\tilde \rho^2$ in $k=\FF_{p^r}$ 
for any even $r \geq 1$. This immediately shows that 
$$\Tr\bigl(F_q,e_{c,\mu}(\cX_{\GaSq}\otimes \bar k,{\WW}_{\lambda}))=\overline{\Tr\bigl(F_q,e_{c,\mu}(\cX_{\GaSq}\otimes \bar k,{\WW}_{\lambda}))}.$$
\end{proof}

Let $e_i$ denote the $i$th elementary symmetric polynomial and $p_i$ the $i$th power sum polynomial. The number of variables should in the future be clear from context. For any partition $\nu \vdash n$, let $s_{\nu}$ denote the Schur polynomial associated to $\nu$ and put 
$$e_{\nu}:=\prod_{i=1}^n e_i^{\nu_i} \;\; \text{and} \;\; p_{\nu}:=\prod_{i=1}^n p_i^{\nu_i}.$$
Recall that $s_{\nu}$ is a polynomial with integer coefficients in the elementary symmetric polynomials $e_1, e_2,\ldots,e_n$, and with rational coefficients in the power sum polynomials and we define $c_{\nu,\xi}$ by, 
$$s_{\nu}=\sum_{\xi \vdash n} c_{\nu,\xi} \frac{p_{\xi}}{z_{\xi}}, \;\; \text{where} \;\; z_{\xi}:=\prod_{i=1}^n \xi_i!i^{\xi_i}.$$
In the representation ring of $\fS_n$ (tensored with $\QQ$) we have the corresponding equality 
$$\mathbf s_{\nu}=\sum_{\xi \vdash n} c_{\nu,\xi} \frac{\mathbf p_{\xi}}{z_{\xi}}.$$

Recall the notation in Section~\ref{sec-charpolfrob}. For any $C \in \cX_{\GaSq}(k)$ we have, using Poincar\'e duality in \'etale cohomology between $H^1(C,\QQl)$ and $H_c^1(C,\QQl)$, that  
$$\Tr(F_q,(\WW)_{C \otimes \bar k})=\overline{\alpha_1(C)}+\overline{\alpha_2(C)}+\overline{\alpha_{3}(C)},$$
and so for any partition $\lambda$ 
$$\Tr(F_q,(\WW_{\lambda})_{C \otimes \bar k})=s_{\lambda}\bigl(\overline{\alpha_1(C)},\overline{\alpha_2(C)},\overline{\alpha_{3}(C)}\bigr).$$

For any partition $\nu$ of $4$, let $\cX_{\nu}(k)$ inside $\cX_{\GaSq}(k)$ consist of the curves whose ramification 
points, where $\alpha$ acts by $\rho$, are defined over $k_{\nu_i}$ but not over a subfield of $k_{\nu_i}$ 
for $1 \leq i \leq 4$. From the Lefschetz trace formula (see \cite[Th.~3.2]{SGA}), it follows that 
if $\sigma_{\nu} \in \fS_4$ has cycle type $\nu$ then  

\begin{equation} \label{eq-traceFq}
\Tr\bigl(F_q \circ \sigma_{\nu},e_c(\cX_{\GaSq} \otimes \bar k,\WW_{\lambda})\bigr) = \sum_{C \in \cX_{\nu}(k)/\cong_k} \frac{s_{\lambda}\bigl(\overline{\alpha_1(C)},\overline{\alpha_2(C)},\overline{\alpha_{3}(C)}\bigr)}{|\mathrm{Aut}_{k}(C)|}.
\end{equation}

By the projection formula we then have  
\begin{equation*}
\Tr\bigl(F_q ,e_{c,\mu}(\cX_{\GaSq} \otimes \bar k,\WW_{\lambda})\bigr) = \sum_{\nu \vdash 4}c_{\mu,\nu} \Tr\bigl(F_q \circ \sigma_{\nu},e_c(\cX_{\GaSq} \otimes \bar k,\WW_{\lambda})\bigr),  
\end{equation*}
giving the equality 
\begin{equation*}
\Tr\bigl(F_q ,e_{c,\mu}(\cX_{\GaSq} \otimes \bar k,\WW_{\lambda})\bigr) \mathbf s_{\mu} = \sum_{\nu \vdash 4}c_{\mu,\nu} \Tr\bigl(F_q \circ \sigma_{\nu},e_c(\cX_{\GaSq} \otimes \bar k,\WW_{\lambda})\bigr) \frac{\mathbf p_{\nu}}{z_{\nu}}.
\end{equation*}

\begin{proposition} \label{prop-det3} For any $\lambda$, we have
$$e_c(\cX_{\GaSq},{\WW}_{\lambda} \otimes \det(\WW)^3)=e_c(\cX_{\GaSq},{\WW}_{\lambda}) \,\LL^{6,3} \, \mathbf s_{1^4},$$
as elements of $K^{\fS_4}_0({\rm Gal}_F)$.
\end{proposition}
\begin{proof}
Let $C$ be an element of $\cX_{\mu}(k)$. We see from Theorem~\ref{thm-det} that 
\begin{multline*}
\Tr(F_q,(\WW_{\lambda} \otimes (\det \WW)^3)_{C \otimes \bar k})=\Tr(F_q,(\WW_{\lambda})_{C \otimes \bar k}) e_{3}\bigl(\alpha_1(C),\alpha_2(C),\alpha_{3}(C)\bigr)^3 \\ 
=\Tr \bigl(F_q,(\WW_{\lambda})_{C \otimes \bar k} \bigr) (-1)^{\mathrm{sign}(\nu)} (a_{\pp_q}+b_{\pp_q}\rho)^6 (a_{\pp_q}+b_{\pp_q}\rho^2)^3.
\end{multline*}
From this the result follows. \end{proof}
\begin{remark} \label{rmk-curveab} If the weight $n_1+2n_2+3n_3$ of the local system $\lambda=n_1\gamma_1+n_2\gamma_2+n_3\gamma_3$ is even then 
$$e_c(\cX'_{\GaSq},{\WW}_{\lambda})=e_c(\cX_{\GaSq},{\WW}_{\lambda}),$$
but if the weight is odd then 
$$e_c(\cX'_{\GaSq},{\WW}_{\lambda})=0$$ 
due to the presence of the automorphism $-1$ of the abelian varieties that $\cX'_{\GaSq}$ parametrizes. But we see from 
Proposition~\ref{prop-det3} that there are no new motives appearing for a local system ${\WW}_{\lambda}$ 
on $\cX_{\GaSq}$ of odd weight, since these motives will, after being tensored with the ``trivial factor'' $\LL^{6,3} \otimes \mathbf s_{1^4}$, appear in 
$e_c(\cX'_{\GaSq},{\WW}_{\lambda} \otimes \det(\WW)^3)$.
\end{remark}

\subsubsection{Normalization of the Euler characteristic} \label{sec-normalized}
In the proof of Proposition~\ref{prop-det3} we also see that for any $\mu$ and $\lambda$ 
$$e_{c,\mu}(\cX_{\GaSq},{\WW}_{\lambda} \otimes \det(\WW)^i)=\LL^{2i,i} \, V_{\lambda,\mu}$$ 
for some element $V_{\lambda,\mu}$ of $K_0({\rm Gal}^{\fS_4}_F)$. 

\begin{definition} \label{def-norm} If $\lambda=n_1\gamma_1+n_2\gamma_2+n_3\gamma_3$ and $\lambda'=n_1\gamma_1+n_2\gamma_2$ then ${\WW}_{\lambda}={\WW}_{\lambda' \otimes \det(\WW)^{n_3}}$ and by taking away the factor ${\LL}^{2n_3,n_3}$ we define the \emph{normalized} Euler characteristic to be 
$$e^{\mathrm{norm}}_{c,\mu}(\cX_{\GaSq},{\WW}_{\lambda}) =V_{\lambda',\mu}.$$ 
\end{definition}

\subsubsection{The appearance of Picard modular cusp forms} \label{sec-pic}
Using Proposition~\ref{prop-eultriv}, \ref{prop-dual} and \ref{prop-det3} we can restrict ourselves to determining the 
Euler characteristics of local system $\lambda=n_1\gamma_1+n_2\gamma_2+n_3\gamma_3$ for which $n_2 \leq n_1$, $n_1 \equiv_3  n_2$, $n_1 \equiv_2 n_3$ and $0 \leq n_3 \leq 5$.

Proposition~\ref{prop-bggfilt} suggests that one finds a motive
in the Euler characteristic $e^{\mathrm{norm}}_c(\cX_{\GaSq},{\WW}_{\lambda})$ that will correspond to the space of Picard modular cusp forms $S_{n_2,n_1+3,n_2+n_3-1}$,
see further in Section~\ref{sec-conj}.
\end{subsection}
\end{section}

\begin{section}{Counts over finite fields} \label{sec-counts}
\begin{subsection}{Information needed} 
In this section we will see what information is needed to compute 
$$\Tr\bigl(F_q ,e_{c,\mu}(\cX_{\GaSq} \otimes \bar k,\WW_{\lambda})\bigr).$$ 
The results will be used 
in Section~\ref{sec-computer} and are the basis of our conjectures in Sections~\ref{sec-conj-lifts} and \ref{sec-conj-main}.

First, define the contributions of the strata to the trace
$$\Tr_{i,\nu,\lambda,q}:=\Tr\bigl(F_q \circ \sigma_{\nu},e_c(\cX_{i,\GaSq} \otimes \bar k,\WW_{\lambda})\bigr).$$

\subsubsection{Counts of smooth curves of genus $3$} \label{sec-counts-g3} 
Let $P_{1,\mu}(k)$ denote the set of square-free polynomials $f$ with coefficients in $k$ of degree four such that $f$ has $\mu_i$ roots defined over $k_i$ but not over any proper subfield of $k_i$. From Equation~\eqref{eq-traceFq}, together with the results of Section~\ref{sec-smooth3} we find that 
$$
\Tr_{1,\nu,q}=\frac{1}{q(q-1)^2} \sum_{f \in P_{1,\mu}(k)} s_{\lambda}\bigl(\overline{\alpha_1(C)},\overline{\alpha_2(C)},\overline{\alpha_{3}(C)}\bigr). 
$$

If we have computed 
$$e_1\bigl(\alpha_1(C_f),\alpha_2(C_f),\alpha_{3}(C_f)\bigr)=a_{1,\rho}(C_f)=-\sum_{a \in \PP^1(k)} \chi(f(a)),$$
then we can use equation~\eqref{eq-duality} together with Theorem~\ref{thm-det} to easily compute 
$e_i \bigl(\alpha_1(C_f),\alpha_2(C_f),\alpha_{3}(C_f)\bigr)$, for $i=2,3$. With this information we can compute $s_{\lambda}\bigl(\alpha_1(C_f),\alpha_2(C_f),\alpha_{3}(C_f)\bigr)$ for any $\lambda$.

One can then simplify the computation of $\Tr_{1,\nu,q}$ by using the group of isomorphisms to find normal forms. Fix a generator $\gamma$ of $k^*$. If $p$ is odd and $p\neq 3$, then a curve of the form
\begin{equation} \label{eq-D1}
y^3=a_4 x^4+a_3x^3+a_2x^2+a_1x+a_0
\end{equation}
with $a_2, a_3$ and $a_4$ non-zero, is isomorphic (over $k$) to a curve of the form 
\begin{equation} \label{eq-D2}
y^3=\gamma^i(a_4'x^4+x^2+x+a'_0),
\end{equation}
for some $a_4'$, $a_0'$ and $ 0 \leq i \leq 2$.
The curves of the latter form are all non-isomorphic and have an automorphism group of order $3$ generated by
$y \mapsto \tilde{\rho} y$. If $D_1 \subset P_{1,\mu}(k)$ is the subset of polynomials of the form~\eqref{eq-D1} and $D_2 \subset P_{1,\mu}(k)$ of the form~\eqref{eq-D2} then 
\begin{multline*}
  \frac{1}{q(q-1)^2} \sum_{f\in D_1} s_{\lambda}\bigl(\overline{\alpha_1(C)},\overline{\alpha_2(C)},\overline{\alpha_{3}(C)}\bigr) \\ =\frac{1}{3}  \sum_{i=1}^3 \sum_{f=a_4'x^4+x^2+x+a'_0 \in D_2}  s_{\lambda}\bigl(\tilde \rho^i\overline{\alpha_1(C_f)},\tilde \rho^i\overline{\alpha_2(C_f)},\tilde \rho^i\overline{\alpha_{3}(C_f)}\bigr).
\end{multline*}
In a similar way, one can construct other normal forms if $a_2$ or $a_3$ is zero.

If $p=2$ then a curve of the form
$$y^3=a_4 x^4+a_3x^3+a_2x^2+a_1x+a_0$$
with $a_1, a_0$ and $a_4$ non-zero, will be isomorphic to a curve of the form 
$$y^3= \gamma^i(x^4+x^3+a_1'x+a'_0)$$
for some $a_1'$, $a_0'$ and $ 0 \leq i \leq 2$.

In this manner we can reduce the number of free parameters in the polynomials in the sum $\Tr_{1,\nu,q}$ from $5$ to $2$ (which is optimal since we are considering a surface).
\subsubsection{Counts of smooth genus $2$ curves joined with elliptic curves} \label{sec-counts-strata2}
We will denote by $P_{2,\mu}(k)$ the set of triples of polynomials $(f_1,f_2,f)$ with coefficients in $k$, as in Section~\ref{sec-strata2}, but where we assume that $f_2$ is of degree $1$ by putting the point $q_2$ in infinity using a linear transformation in $x$, and such that $f$ and $f_1$ together have $\mu_i$ roots defined over $k_i$ but not over any subfield of $k_i$.  We find that 
$$
\Tr_{2,\nu,q}=\frac{1}{(q^2(q-1)^4} \sum_{(f_1,f_2,f) \in P_{2,\mu}(k)} s_{\lambda}\bigl(\overline{\alpha_1(C_{f_1,f_2})},\overline{\alpha_2(C_{f_1,f_2})},\overline{\alpha_{1}(C_f)}\bigr). 
$$
If we have computed 
$$e_1\bigl(\alpha_1(C_{f_1,f_2}),\alpha_2(C_{f_1,f_2})\bigr)=-\sum_{a \in \PP^1(k)} \chi(f_1(a))\chi(f_2(a))^2,$$
we can use Theorem~\ref{thm-det} to compute $e_2 \bigl(\alpha_1(C_{f_1,f_2}),\alpha_2(C_{f_1,f_2}) \bigr)$.  Using the equations in Section~\ref{sec-charpol-ell} we can compute $e_1 \bigl(\alpha_{1}(C_f)\bigr)$. From this we can determine  
$e_i \bigl(\alpha_1(C_{f_1,f_2}),\alpha_2(C_{f_1,f_2}),\alpha_{1}(C_f)\bigr)$ for $i=1,2,3$.

As in Section~\ref{sec-counts-g3}, we can use the group of isomorphisms to find normal forms which will simplify the computation of $\Tr_{2,\nu,q}$. Fix $\gamma$, a generator of $k^*$. If $p$ is odd, and $(b_2x^2+b_1x+b_0,c_1x+c_0,d_2x^2+d_1x+d_0)$ is in $P_{2,\mu}(k)$ with $c_0 \neq 0$, then we can find an isomorphism to a curve given by $(f'_1,f'_2,f'_3)$ in $P_{2,\mu}(k)$, where $f'_3=x^2+\gamma^i$, $f'_2=x+1$ and $f'_1=\gamma^j(x^2+b_1)$ for some $b_1 \in k$, $0 \leq i \leq 5$ and $0 \leq j \leq 2$. The curves of the latter form are all non-isomorphic and have an automorphism group of order $3$ generated by
$y \mapsto \tilde{\rho} y$. Similar normal forms can be found if $c_0=0$ and if $p$ is even. 

In this manner we can reduce the free parameters in the polynomials in the sum $\Tr_{2,\nu,q}$ from $8$ to $1$ (which is optimal since we are here considering a curve).
\subsubsection{Counts of triples of elliptic curves} \label{sec-counts-strata3}
Let $P_{3,\mu}(k)$ denote the set of triples of polynomials $(f_1,f_2,f_3^2)$ with coefficients in $k$, as in Section~\ref{sec-strata3}, such that $f_1$ and $f_2$ together have $\mu_i$ roots defined over $k_i$ but not over any subfield of $k_i$. Define $P'_{3,\mu}(k)$ in the same way, but where $f_1,f_2$ are defined over $k_2$ and where Frobenius sends $f_1$ to $f_2$. Note that we have more isomorphisms between the curves corresponding to elements of these sets, namely by switching the two "wings" of these curves, that is between $C_{f_1,f_2,f_3^2}$ and $C_{f_2,f_1,f_3^2}$. We have that 
\begin{multline*}
\Tr_{3,\nu,q}=\frac{1}{2q^3(q-1)^3} \sum_{(f_1,f_2,f_3^2) \in P_{3,\mu}(k)} s_{\lambda}\bigl(\overline{\alpha_1(C_{f_1})},\overline{\alpha_1(C_{f_2})},\overline{\alpha_{1}(C_{f_3^2})}\bigr)\\  +\frac{1}{2q^3(q-1)(q^2-1)} \sum_{(f_1,f_2,f_3^2) \in P'_{3,\mu}(k)} s_{\lambda}\bigl(\overline{\alpha_1(C_{f_1})},\overline{\alpha_2(C_{f_2})},\overline{\alpha_{1}(C_{f_3^2})}\bigr).
\end{multline*}
If $(f_1,f_2,f_3^2)$ are in $P_{3,\mu}(k)$  then using the equations in Section~\ref{sec-charpol-ell} we can compute $e_1 \bigl(\alpha_{1}(C_{f_i})\bigr)$ for $i=1,2,3$. For $(f_1,f_2,f_3^2)$ in $P'_{3,\mu}(k)$ we can determine $e_1 \bigl(\alpha_{1}(C_{f_3^2})\bigr)$ in the same way. Moreover, $e_1\bigl(\alpha_1(C_{f_1}),\alpha_2(C_{f_2}) \bigr) = 0$  
and 
$$p_2 \bigl(\alpha_1(C_{f_1}),\alpha_2(C_{f_2}) \bigr)=-\sum_{a \in \PP^1(k_2)} \chi(f_1(a))-\sum_{a \in \PP^1(k_2)} \chi(f_2(a)),
$$
where $\chi$ is the third power residue symbol for $k_2$. In both cases, this gives enough information to determine $e_i(\alpha_1(C_{f_1}),\alpha_1(C_{f_2}),\alpha_1(C_{f_3^2})\bigr)$ for $i=1,2,3$.

In Section~\ref{sec-charpol-ell} a representative of each $k$-isomorphism class of $\cX^{(1)}_{\Gamma}(k)$ is given. With this information $\Tr_{3,\nu,q}$ is easily computed for any $\nu$ and $q$. 
\end{subsection}
\begin{subsection}{Counts with constant coefficients} \label{sec-constant} Let us consider the Euler characteristic when 
${\WW}_{\lambda}=\QQl$. 

We will repeatedly use the trick below that summing over all elements defined over $k$, of one of 
the groupoids at hand, and then dividing by the number of $k$-isomorphisms between these 
elements is the same as summing elements weighted by the reciprocal of their number of 
$k$-automorphisms, compare for instance
\cite[Section 5]{Ber}. 

The elements of $\cX_{1,\Gamma}(k)$ together with their isomorphisms are described in 
Section~\ref{sec-smooth3}. For $\GaSq$ we divide into the different choices of four branch 
points on $\PP^1$ giving 
\begin{multline*}
\Tr\bigl(F_q ,e_c(\cX_{1,\GaSq} \otimes \bar k,\QQl)\bigr)=
\Bigl(q(q-1)(q-2)(q-3) \frac{\mathbf p_1^4}{24}\\ 
+(q^2-q)q(q-1)\frac{\mathbf p_1^2 \mathbf p_2}{4}+
(q^3-q)q\frac{\mathbf p_1 \mathbf p_3}{3}+
(q^2-q)(q^2-q-2)\frac{\mathbf p_2^2}{8}\\
+(q^4-q^2)\frac{\mathbf p_4}{4} \Bigr)/\bigl(q(q-1)\bigr)=q^2\mathbf s_4+(1-q) \mathbf s_{3,1}-q\mathbf s_{2,2}+\mathbf s_{2,1,1}.
\end{multline*}

Similarly for $\cX_{2,\GaSq}$ we use Sections~\ref{sec-smooth2}, \ref{sec-smooth1} and \ref{sec-strata2} and 
we find that 
\begin{multline*}
\Tr\bigl(F_q ,e_c(\cX_{2,\GaSq} \otimes \bar k,\QQl)\bigr)=
\biggl(\Bigl(\frac{\mathbf p_1^2}{2}+\frac{\mathbf p_2}{2} \Bigr) \Bigl( q(q-1)(q-2) \frac{\mathbf p_1^2}{2}\\ 
+q(q^2-q) \frac{\mathbf p_2 \mathbf p_2}{2}\Bigr) \biggr)/\bigl(q(q-1)\bigr)
=(q-1)\mathbf s_4+(q-2) \mathbf s_{3,1}+(q-1)\mathbf s_{2,2}-\mathbf s_{2,1,1}.
\end{multline*}

For $\cX_{3,\GaSq}$ we use Sections~\ref{sec-smooth1} and \ref{sec-strata3}, and we recall the plethysm $\circ$ 
to deal with the symmetry of the two elliptic curves that form the ``wings". We find that 
\begin{multline*}
\Tr\bigl(F_q ,e_c(\cX_{3,\GaSq} \otimes \bar k,\QQl)\bigr)=
\Bigl(\frac{\mathbf p_1^2}{2}+\frac{\mathbf p_2}{2} \Bigr) \circ \Bigl(\frac{\mathbf p_1^2}{2}+\frac{\mathbf p_2}{2} \Bigr)  \\
=\frac{1}{2}\Bigl(\frac{\mathbf p_1^2}{2}+\frac{\mathbf p_2}{2} \Bigr)^2+\frac{1}{2}\Bigl(\frac{\mathbf p_2^2}{2}+\frac{\mathbf p_4}{2} \Bigr)
=\mathbf s_4+\mathbf s_{2,2}.
\end{multline*}

The trace of Frobenius on $e_c(\cX_{3,\GaSq} \otimes \bar k)$ for all $q \equiv_3 1$ determines $e_c(\cX_{\GaSq} \otimes \bar k,\QQl)$ as an element in $K_0({\rm Gal}^{\fS_4}_F)$, see Section~\ref{sec-traces}. Summing the three cases above we then get the following. 
\begin{proposition} \label{prop-const} We have an equality 
of elements in $K_0({\rm Gal}^{\fS_4}_F)$:
$$e_c(\cX_{\GaSq} \otimes \bar k,\QQl)\bigr)=(\LL^2+\LL)\mathbf s_4-\mathbf s_{3,1}.$$ 
\end{proposition}

We continue with case iv) of Section~\ref{sec-stable}. On the elliptic curve $C_1$, there is a choice of a point 
not equal to any of the ramification points. This gives a contribution $q+1-a_1(C_1)-r_1(C_1)$, where 
$r_1(C_1)$ is the number of ramification points of $C_1$ defined over $k$. 
A computation similar to the one in Section~\ref{sec-charpol-ell} shows, due to the symmetry, that the 
contribution from $a_1(C_1)$ vanishes. The genus $0$ curve contributes a $p_1$, and so the 
trace of Frobenius on the Euler characteristic of the strata for case iv) equals,  
\begin{multline*}
p_1 \Bigl((q+1-3) \frac{\mathbf p_1^3}{6}+(q+1-1)\frac{\mathbf p_1\mathbf p_2}{2}  
+  (q+1)\frac{\mathbf p_3}{3}  \Bigr)  \\  =q\mathbf s_4+(q-1)\mathbf s_{3,1}-\mathbf s_{2,2}-\mathbf s_{2,1,1}.
\end{multline*}

Case v) is straightforward and the trace of Frobenius on its Euler characteristic equals,  
$$
\mathbf p_1^2 \Bigl(\frac{\mathbf p_1^2}{2}+\frac{\mathbf p_2}{2} \Bigr) = \mathbf s_4+2\mathbf s_{3,1}+\mathbf s_{2,2}+\mathbf s_{2,1,1}.
$$

Summing cases i) to v) and using the purity of a smooth and proper (Deligne-Mumford) stack we find the following.
\begin{proposition} We have an equality 
of elements in ${\rm Gal}^{\fS_4}_F$:
$$H_c^i(\tilde \cX_{\GaSq} \otimes \bar k,\QQl)\bigr)=\begin{cases} \LL^2 \mathbf s_4 & \text{if } i = 4\\ \LL(2\mathbf s_4+ \mathbf s_{3,1}) & \text{if } i = 2 \\ \LL^0 \mathbf s_4 & \text{if } i = 0 \\ 0 & \text{if } i  \text{ odd} \end{cases}$$ 
\end{proposition}
 
Together, case iv) and v) contribute $(q+1)(\mathbf s_4+\mathbf s_{3,1}) $ and they form a $\PP^1$-bundle under the morphism \eqref{eq-tildestar}. Hence, their image (the four cusps) contribute $\mathbf s_4+\mathbf s_{3,1}$ and we get a trace of Frobenius equal to 
$$\Tr\bigl(F_q ,e_c(\cX^*_{\GaSq} \otimes \bar k,\QQl)\bigr)=(q^2+q+1)\mathbf s_4,$$ 
which echoes the fact that $X^*_{\GaSq} \cong \PP^2$.
\begin{subsubsection}{The genus $2$ case} \label{sec-constant-g2}
There are two strata in $\cX^{(2)}_{\GaSq}$, one consisting of smooth genus $2$ curves 
and one consisting of pairs of genus one curves, one with action of type $(1,0)$ and one of type $(0,1)$, joined at a ramification point on each curve. There is an action of $\fS_2 \times \fS_2$ on the two pairs of ramification points, one pair where the action of $\alpha$ 
is by $\rho$ and one by $\rho^2$. Let us use the notation $\mathbf p_i$ and $\tilde{\mathbf p}_i$, and $\mathbf s_{\mu}$ and 
$\tilde{\mathbf s}_{\mu}$, for the basis of representations of the two components of  $\fS_2 \times \fS_2$. A consideration analogous 
to the ones above shows that  
\begin{multline*}
\Tr\bigl(F_q ,e_c(\cX^{(2)}_{\GaSq} \otimes \bar k,\QQl)\bigr)=\bigl( (q+1)q(q-1)(q-2) \frac{\mathbf p_1^2 \tilde{\mathbf p}_1^2}{4} \\
+(q^2-q)(q+1)q \frac{\mathbf p_2\tilde{\mathbf p}_1^2+\mathbf p_1^2\tilde{\mathbf p}_2}{4}+(q^2-q)(q^2-q-2)\frac{\mathbf p_2\tilde{\mathbf p}_2}{4}\bigr)/\bigl((q+1)q(q-1)\bigr) \\
+\frac{\mathbf p_1^2\tilde{\mathbf p}_1^2}{4}+\frac{\mathbf p_1^2\tilde{\mathbf p}_2}{4}+\frac{\mathbf p_2\tilde{\mathbf p}_1^2}{4}+\frac{\mathbf p_2\tilde{\mathbf p}_2}{4}=q\mathbf s_2 \tilde{\mathbf s}_2-\mathbf s_{1,1}\tilde{\mathbf s}_{1,1}, 
\end{multline*}
and we can conclude the following. 
\begin{proposition} \label{prop-genus2} We have an equality 
of elements in $K_0({\rm Gal}^{\fS_4}_F)$:
$$e_c(\cX^{(2)}_{\GaSq} \otimes \bar k,\QQl)\bigr)=\LL \mathbf s_2 \tilde{\mathbf s}_2-\mathbf s_{1,1}\tilde{\mathbf s}_{1,1}.$$ 
\end{proposition}
\end{subsubsection}
\end{subsection}
\begin{subsection}{Euler characteristics of local systems for elliptic curves} 
Using the results of Section~\ref{sec-charpol-ell} and Section~\ref{sec-traces}
we find that 
$$\Tr\bigl(F_q ,e_c(\cX^{(1)}_{\GaSq} \otimes \bar k,\WW_k)\bigr)=\sum_{i=0}^{2} \rho^{ki}(a_{\pp_q}+b_{\pp_q}\rho^2)^k \bigl( \mathbf p_{1}^2+(-1)^{k} \mathbf p_{2}\bigr),
$$
for all $q$ such that $q \equiv_3 1$. This equality, together with the fact that an element in $K_0({\rm Gal}_F)$ is determined by 
all traces of Frobenius, shows the following.  
\begin{proposition} For any $k \geq 0$, 
we have the equality in $K^{\fS_4}_0({\rm Gal}_F)$:
$$e_c(\cX^{(1)}_{\GaSq},{\WW}_{k})=\begin{cases} \LL^{0,k} \, \mathbf s_{2} & \text{if } k \equiv_6 0  \\ \LL^{k,0} \, \mathbf s_{1^2} & \text{if } k \equiv_6 3  \\ 0 & \text{if } k \not \equiv_3 0 \end{cases}$$ 
\end{proposition}
\end{subsection}
\end{section}

\begin{section}{Numeric Euler characteristics of local systems} \label{sec-numeric}
In this section the ground field will be $\CC$, we will consider the compactly supported Betti cohomology, and we will find a formula for the integer-valued Euler characteristic, 
$$E_{c,\mu}(X_{\GaSq},{\WW}_{\lambda}):= \sum^4_{i=0} (-1)^i \dim_{\CC} H^i_{c,\mu}(X_{\GaSq}, {\WW}_{\lambda}) \in \ZZ$$
for any $\lambda$ and $\mu$, where $H^i_{c,\mu}$ is the $\mu$-isotypic component of $H^i_{c}$. Examples of computations using this formula will be found in Section~\ref{sec-ex-individ}. 

Similarly to the above we will write,
$$E_c(X_{\GaSq},{\WW}_{\lambda}):= \sum_{\mu \vdash 4} \frac{E_{c,\mu}(X_{\GaSq},{\WW}_{\lambda})}{\dim \mathbf s_{\mu}}\mathbf s_{\mu}  \in \ZZ[\fS_4].$$
The reader should compare this section to 
the article~\cite{BvdG} and the references therein. 
Note that by comparison theorems this numerical Euler characteristic will be the same if Betti cohomology is replaced by $\ell$-adic \'etale cohomology as described in Section~\ref{sec-euler}. So, $E_c(X_{\GaSq},{\WW}_{\lambda})$ equals $\dim e_c(\mathcal X_{\GaSq},{\WW}_{\lambda})$ for any $\lambda$ of even weight.

We stratify our moduli space $X_{\Gamma}$, first into $X_{1,\Gamma}$, $X_{2,\Gamma}$ and $X_{3,\Gamma}$ as in Section~\ref{sec-moduliofcurves}. We then stratify further into strata $\Sigma_i(G)$, for $i=1,2,3$ and $G$ a finite group, 
consisting of the curves corresponding to points of $X_{i,\Gamma}$ whose automorphism group equals $G$. As usual, let $H^1(C,\CC)^{\rho}$ denote the $\rho$-eigenspace of $H^1(C,\CC)$ when acting by $\alpha$. Say that $g \in G$ has eigenvalues $\xi_1(g)$, $\xi_2(g)$ and $\xi_3(g)$ when acting on $H^1(C,\CC)^{\rho}$ of a curve $C \in \Sigma(G)$. Say furthermore that the induced action of $g \in G$ on the four ramification points of 
$C \in  \Sigma_i(G)$ where $\alpha$ acts by $\rho$ has $\nu_i$ cycles of length $i$. Note that this data will be constant on the strata, i.e. independent of the choice of $C \in  \Sigma_i(G)$. If $e_c \bigl(\Sigma_i(G) \bigr)$ denotes the usual compactly supported 
Euler characteristic of $\Sigma_i(G)$ then 
\begin{equation} \label{eq-numeric}
  E_c(X_{\GaSq},{\WW}_{\lambda})=\sum_{i=1}^3 \sum_{G} \frac{E_c \bigl(\Sigma_i(G) \bigr)}{|G|}  \sum_{g \in G}
  s_{\lambda} \bigl(\xi_1(g), \xi_2(g),\xi_3(g) \bigr) \mathbf p_{\nu}.
  \end{equation}
In the sections below, we will find the necessary information to compute this formula for any given $\lambda$.

In all cases below, the automorphism groups that appear are cyclic, so it is enough to give the three eigenvalues of a generator, which we will denote by $\phi$, together with its cycle type as a permutation of the four ramification points.
\begin{subsection}{Numerical Euler characteristics for smooth curves of genus $3$} One easily finds that there are four different cyclic automorphism groups in this case, namely the generic case $\mathbf C_3$ and then $\mathbf C_6$, $\mathbf C_9$ and $\mathbf C_{12}$. 
  The strata for $\mathbf C_9$ and $\mathbf C_{12}$ consist of a single point. For $\mathbf C_6$, the stratum is $1$-dimensional. Each isomorphism class can be represented by a curve of the form $f=x^4+ax^2+1$ with $a \in \mathbb C$. To make this a smooth curve we need $a^2 \neq 0,1$. Moreover, two curves of this form are isomorphic precisely if their coefficients $a$ differ by a sign. This shows that the Euler characteristic of this stratum equals $-1$. The whole moduli space $X_{1,\Gamma}$ is described in the end of Section~\ref{sec-PMS} and we find that it has Euler characteristic $1$ (compare with the point count in Section~\ref{sec-constant}). From this it follows that the Euler characteristic of the generic (open dense) stratum must be $0$.

We have an isomorphism $H^1(C,\CC) \cong  H^0(C,\Omega) \oplus \overline {H^0(C,\Omega)}$ and the subspace $H^1(C,\CC)^{\rho}$ has a basis consisting of $dx/y^2$, $xdx/y^2$ and the dual of $dx/y$. The eigenvalues of the action of $\phi$ on $H^1(C,\CC)^{\rho}$ can thus be found through its action on this basis.

We exemplify such a computation in the case of $\mathbf C_9$. The other cases are completely analogous. We have that $\phi$ applied to $dx/y^2$, $xdx/y^2$ and $dx/y$ equals $\rho dx/(\epsilon^2y^2)=\epsilon dx/y^2$, $\rho^2 dx/(\epsilon^2y^2)=\rho\epsilon dx/y^2$ and $\rho dx/(\epsilon y)=\epsilon^2 dx/y$. The action should be on the dual of $\rho dx/y$ and hence this eigenvalue becomes $\epsilon^{-2}=\rho^2\epsilon$. We see that the action of $\phi$ cyclically permutes three of the ramification points and fixes the fourth. 

The data to compute the contribution from the strata $X_{1,\GaSq}$ to equation $\eqref{eq-numeric}$, is found in the table below.

\bigskip
\vbox{
\centerline{\def\quad{\hskip 0.3em\relax}
\vbox{\offinterlineskip
\hrule
\halign{&\vrule#& \quad \hfil#\hfil \strut \quad  \cr
height2pt&\omit&&\omit&&\omit&&\omit&&\omit&&\omit& \cr
& $G$ && $f$ && $\Sigma_1(G)$ && $\phi(x,y)$ && $\xi_1(\phi),\xi_2(\phi),\xi_3(\phi)$ && $\mathbf p_{\nu}$ & \cr
height2pt&\omit&&\omit&&\omit&&\omit&&\omit&&\omit& \cr
\noalign{\hrule}
height2pt&\omit&&\omit&&\omit&&\omit&&\omit&&\omit& \cr
& $\mathbf C_3$ &&                            && $0$ && $(x, \rho y)$ && $  \rho,\rho,\rho $ && $\mathbf p_1^4 $ & \cr
& $\mathbf C_6$ && $x^4+ax^2+1$ && $-1$ && $(-x,\rho y)$ && $\rho,-\rho,-\rho$ && $\mathbf p_2^2 $ & \cr
& $\mathbf C_9$ && $x(x^3-1)$          && $1$ && $(\rho x, \epsilon y)$ && $\epsilon,\rho \epsilon, \rho^2 \epsilon$ && $\mathbf p_1 \mathbf p_3$ & \cr
& $\mathbf C_{12}$ && $x^4-1$     && $1$ && $(ix,\rho y)$ && $-\rho, i \rho, -i \rho $ && $\mathbf p_4$ & \cr
height2pt&\omit&&\omit&&\omit&&\omit&&\omit&&\omit& \cr
} \hrule}
}}
\noindent

\end{subsection}

\begin{subsection}{Numerical Euler characteristics for smooth genus $2$ curves joined with elliptic curves}
  Let us first consider smooth genus $2$ curves $C_{f_1,f_2}$ together with a marked root of $f_2$, which we place in infinity, see Section~\ref{sec-strata2}. Note that the hyperelliptic involution does not fix the marked point. There are two strata. The generic strata, with automorphism group $\mathbf C_3$, has a representative $f_1=x^2+ax+1$, $f_2=x$ for each $a\neq 0 \in \CC$. This gives an Euler characteristic equal to $-1$. The strata with automorphism group $\mathbf C_6$ consists of a point, given by $a=0$. In this case, the involution switches the two ramification points where $\alpha$ acts by $\rho$. In both cases, the subspace $H^1(C_{f_1,f_2},\CC)^{\rho}$ has a basis consisting of $xdx/y^2$ and the dual of $dx/y$. Computations as in the previous section gives the following table.

\bigskip
\vbox{
\centerline{\def\quad{\hskip 0.3em\relax}
\vbox{\offinterlineskip
\hrule
\halign{&\vrule#& \quad \hfil#\hfil \strut \quad  \cr
height2pt&\omit&&\omit&&\omit&&\omit&&\omit&&\omit& \cr
& $G$ && $f_1f_2^2$ && $\Sigma_1(G)$ && $\phi(x,y)$ && $\xi_1(\phi),\xi_2(\phi),\xi_3(\phi)$ && $\mathbf p_{\nu}$ & \cr
height2pt&\omit&&\omit&&\omit&&\omit&&\omit&&\omit& \cr
\noalign{\hrule}
height2pt&\omit&&\omit&&\omit&&\omit&&\omit&&\omit& \cr
& $\mathbf C_3$ && $(x^2+ax+1)x^2$  && $-1$ && $(x, \rho y)$ && $  \rho,\rho $ && $\mathbf p_1^2 $ & \cr
& $\mathbf C_6$ && $(x^2+1)x^2$ && $1$ && $(-x,\rho y)$ && $\rho,-\rho$ && $\mathbf p_2 $ & \cr
height2pt&\omit&&\omit&&\omit&&\omit&&\omit&&\omit& \cr
} \hrule}
}}
\noindent

The elliptic curves come with a marked ramification point at infinity and there is only one stratum consisting of the curve with equation $y^3=x^2+1$ and automorphism group $\ZZ/6\ZZ$ generated by the element $\phi:(x,y) \mapsto (-x,\rho y)$. The single 
eigenvalue of $\phi$ acting on $H^1_c(C,\CC)^{\rho}$ is $-\rho$. Furthermore $\phi^i$ permutes the ramification points if $i$ is odd and fixes them if $i$ is even. 

The possible automorphism groups of curves $C_{f_1,f_2,f_3} \in X_{2,\GaSq}$ are just products of the automorphism groups for $C_{f_1,f_2}$ and $C_{f_3}$. Moreover, we have that $H^1(C_{f_1,f_2,f_3},\CC)^{\rho}\cong H^1(C_{f_1,f_2},\CC)^{\rho} \oplus H^1(C_{f_3},\CC)^{\rho}$. So, piecing together the information above enables us to compute the contribution from $X_{2,\GaSq}$ to equation $\eqref{eq-numeric}$. 
\end{subsection}

\begin{subsection}{Numerical Euler characteristics for triples of elliptic curves} Triples of elliptic curves $C_{f_1,f_2,f_3^2}$ are described in Section~\ref{sec-strata3}. The ``backbone'', corresponding to $f_3^2$, will only have an automorphism group generated by $\alpha$, because two of its ramification points are fixed. Then there is an additional automorphism $\sigma$ by switching the two ``wings'' corresponding to $f_1$ and $f_2$. This gives rise to an automorphism group of the form $G=\mathbf C_3 \times (\mathbf C_6 \wr \mathfrak S_2)$, where $\wr$ denotes the wreath product. Say that $\alpha_i$ is an automorphism of $C_{f_i}$ for $i=1 \ldots3$ with eigenvalues $\tau_i$ acting on $H^1(C_{f_i},\CC)^{\rho}$, then these will also be the eigenvalues of the induced action on $H^1(C_{f_1,f_2,f_3^2},\CC)^{\rho} \cong H^1(C_{f_1},\CC)^{\rho} \oplus H^1(C_{f_2},\CC)^{\rho} \oplus H^1(C_{f_3^2},\CC)^{\rho} $. If the previous automorphism is composed with the involution in $\mathfrak S_2$, the eigenvalues will be $(\tau_1 \tau_2)^{1/2},-(\tau_1 \tau_2)^{1/2},\tau_3$. From this information one can compute the contribution from $X_{3,\GaSq}$ to equation $\eqref{eq-numeric}$.
\end{subsection}
\end{section}

\begin{section}{Our approach}  \label{sec-conj}  Here we will explain the approach that led us to the conjectures on
Picard modular forms in Section~\ref{sec-conj-lifts} and \ref{sec-conj-main}.

\begin{subsection}{Computer counts over finite fields}  \label{sec-computer} Using the results of Section~\ref{sec-counts} we computed 
$$\Tr\bigl(F_q ,e^{\mathrm{norm}}_{c,\mu}(\cX_{\GaSq} \otimes \bar k,\WW_{\lambda})\bigr)$$
for all prime powers $q\equiv_3 1$ such that $q \leq 67$, and all partitions $\lambda$ such that $n_1+n_2+2 \leq 40$. These traces always turned out to be in $\ZZ[\rho]$ as they should be, see Section~\ref{sec-traces}.

The conjectures of this section are based upon these computer counts (using the equality \eqref{eq-frobenii}). 
\end{subsection}

\begin{subsection}{Preview}
We are interested in calculating the trace of Hecke operators on the
$\mathfrak{S}_4$-isotypic components of the space 
$S_{j,k,l}(\Gamma_1[\sqrt{-3}])$ of cusp forms of given weight.
In the analogous case of the space $S_k$ of cusp forms of weight $k$ 
on ${\rm SL}(2,{\ZZ})$ one can use for even $k>0$ the formula
$$
{\rm Tr}(T(p),S_{k+2})={\rm Tr}(F_p,S[k+2])
$$
with $S[k+2]$ the Chow motive of dimension $2 \dim S_{k+2}$ 
associated by Scholl (\cite{Scholl}) to the space $S_{k+2}$
and $F_p$ denotes Frobenius. By Deligne's result the motive $S[k+2]$
can be found inside the cohomology of a local system $\mathbb V_k$ on the moduli
space ${\mathcal A}_1$ of elliptic curves
\begin{equation} \label{eq-A1}
e_c({\mathcal A}_1,\mathbb V_k)=-S[k+2]-1 
    \end{equation}
and we thus can use counts of points over finite fields to calculate the
trace of Frobenius on this cohomology and thus the traces of the Hecke
operators. 
Note that the $-1$ in the last formula comes from the Eisenstein
cohomology. By replacing $e_c$ by the inner cohomology $e_{!}$ 
we get rid of it. We remark that equation~\eqref{eq-A1} still holds for $k=0$ if we 
put $S[2]=-\LL-1$.

We want the analogue of this for our Picard modular case. Ideally,
in our case one would hope for the existence of a motive $S[j,k,l]$
of dimension $3\, \dim S_{j,k,l}(\Gamma[\sqrt{-3}])$ defined over $F$ 
such that
$$
{\rm Tr}(T(\nu),S_{j,k,l}(\Gamma_1[\sqrt{-3}]) =
{\rm Tr}(F_{\nu},S[j,k,l])
$$
with $T(\nu)$ the Hecke operator for any $\nu \in \ZZ[\rho]$ such that $\nu \equiv_3 1$ and $\nu \bar \nu=p$ a prime and $F_{\nu}$ is the Frobenius as in Section~\ref{sec-traces} (see also  Section~\ref{sec-notation}).
Moreover, $S[j,k,l]$ should appear $\mathfrak{S}_4$-equivariantly in the second inner cohomology group of the corresponding local system on our moduli space, see Proposition~\ref{prop-bggfilt}.

However, one must expect deviations from this due to the fact that there will be liftings from ${\rm U}(1)$ and ${\rm GL}(2)$. In Section~\ref{sec-conj-lifts} we make precise conjectures on all such lifts. To any of these lifts $f$ that is a Hecke eigenform we can (conjecturally) associate a reducible $3$-dimensional Galois representation $M_f$ defined over $F$ such that 
$$\Tr(T(\nu),f)=\Tr(F_{\nu},M_f).$$
But for most of these in we only see a contribution from a $1$-dimensional or $2$-dimensional part of $M_f$ in the \'etale cohomology of our local systems.
After removing these cusp forms we are left with a (conjectural) Hecke-invariant subspace of what we call \emph{genuine} Picard modular forms and that we denote by $S_{j,k,l}^{\rm gen}(\Gamma_1[\sqrt{-3}])$. So, to each Hecke eigenform in this space there should be a $3$-dimensional Galois representation appearing in the cohomology and its (normalized) Hodge degrees in Betti cohomology should be $(j+k-1,0)$, $(j+1,k-2)$ and $(0,j+k-1)$. 

For any $n_1 \equiv_3 n_2$ and $n_1 \equiv_2 n_3$, our goal is to have a formula analogous to equation~\eqref{eq-A1}, namely,
\begin{equation} \label{eq-main}
e_{c}^{\rm norm}({\mathcal X}_{\GaSq},{\WW}_{\lambda})=
\breve{S}[n(\lambda)]+e_{\rm extr}(\lambda)
    \end{equation}
equivariant for the action of $\mathfrak{S}_4$ and with $e_{\rm extr}(\lambda)$
coming from endoscopic groups such that (except for the case $n_1=n_2=n_3=0$) 
\begin{equation} \label{eq-main-tr}
{\rm Tr}(T(\nu),S_{n(\lambda)}^{\rm gen}(\GaSq))
={\rm Tr}(F_{\nu}, \breve{S}[n(\lambda)])\, .
    \end{equation}
for any $\nu \equiv_3 1 $ with norm $p$ a prime 
and with $n(\lambda)=(n_2,n_1+3,n_2+n_3-1)$.
Ideally equation~\eqref{eq-main} should be an equality of motives,
but we can also treat it as an equality of bookkeeping devices for
calculating traces as in equation~\eqref{eq-main-tr}. To make equation~\eqref{eq-main} hold also in the case $n_1=n_2=n_3=0$ we put $\breve{S}[0,3,2]=(\LL^2+\LL+1)\mathbf s_4$, see Proposition~\ref{prop-const}.

Analogously, for the numerical Euler characteristic we should have 
$$\dim_{\fS_4} S_{n(\lambda)}^{\rm gen}(\GaSq) = \frac{1}{3} \bigl(E_{c}(X_{\GaSq},{\WW}_{\lambda})-E_{\rm extr}(\lambda) \bigr).
$$

This is formulated in Main conjecture~\ref{conj-main} and Conjecture~\ref{conj-dim}. All of the conjectures of the following sections was found by analyzing the data described in Section~\ref{sec-computer} together with all of the knowledge acquired in the previous sections. 
\end{subsection}

\begin{subsection}{Notation for local systems and Hecke operators} \label{sec-not-locheck} In the following sections we will use a slightly different notation for our local systems $\lambda=n_1\gamma_1+n_2\gamma_2+n_3\gamma_3$ writing $\lambda=(a+i,i,-b+i)$ with $n_1=a$, $n_2=b$, $n_3=-b+i$. Assume from now on that $a \equiv_3 b$ and put $n(\lambda) =(b,a+3,i+2)$. So, if $i \equiv_2 a+b$ then we expect a Galois representation (of dimension $1$, $2$ or $3$) corresponding to each eigenform in $S_{n(\lambda)}(\GaSq)$ appearing (with positive coefficient) in $e_c({\mathcal X}_{\GaSq},{\WW}_{\lambda})$. If $i \not \equiv_2 a+b$ we expect the same contribution but with the $\fS_4$-action twisted by $\mathbf s_{1^4}$, compare Proposition~\ref{prop-det3}.

By $\nu$ we will always mean an element of $\ZZ[\rho]$ such that $\nu \equiv_3 1$ and $\nu\bar{\nu}=p$ with $p$ a prime. To such an element there is a corresponding Hecke operator acting on $S_{n(\lambda)}(\GaSq)$ that we denote by $T(\nu)$, see Section~\ref{sec-heckeoperator}.

Finally, $\mu$ will denote a partition of $4$.
\end{subsection}
\end{section}

\begin{section}{Conjectured Lifts} \label{sec-conj-lifts}
By analyzing the data described in Section~\ref{sec-computer} we see Galois representations in the cohomology of our local systems that seem to be associated to lifts of modular forms from ${\rm U}(1)$ or ${\rm GL}(2)$. The conjectures of this section are based upon these examples. 

Each such lifted eigenform then contributes a piece of dimension $1$, $2$ or $3$ to the cohomology. It is important for us to identify the $1$-dimensional  and $2$-dimensional pieces in order to be left with the `genuine' Picard modular forms with associated irreducible $3$-dimensional Galois representations.

We will use the notation from Section~\ref{sec-not-locheck}.

\begin{subsection}{Notation for elliptic modular forms} \label{sec-not-elliptic}
We write $\Gamma_0(N)$ and $\Gamma_1(N)$ for the usual subgroups of 
${\rm SL}(2,{\ZZ})$ and we write $S_k(\Gamma_0(N))$
for the space of cusp forms of weight $k$ on $\Gamma_0(N)$ and
$S_k^{\rm new}(\Gamma_0(N))$ for the subspace of new forms.
The dimensions of these spaces will be denoted by $s_k(\Gamma_0(N))$ and $s^{\rm new}_k(\Gamma_0(N))$. In our case we will have level $N=3$ or $N=9$. 

We note for level $3$ and even $k>2$ the dimension formula
$$
s_k^{\rm new}(\Gamma_0(3))=
\begin{cases} 
\lfloor {\frac{k}{6}} \rfloor +1 & \text{if $k\equiv_{12} \, \pm 2 $} \cr
\lfloor {\frac{k}{6}} \rfloor-1 & \text{if $k\equiv_{12} \, 0 $} \cr
\lfloor {\frac{k}{6}} \rfloor & \text{else.}\cr
\end{cases}
$$
For odd $k$ we split the space $S_k(\Gamma_1(3))$ as
$$
S_k(\Gamma_1(3))=S_k^{-}(\Gamma_1(3))\oplus S_k^{+}(\Gamma_1(3))
$$
into the $\pm$-eigenspace for the Fricke operator
$W_3$. For odd $k\geq 3$ the dimension of $S_k^{\pm}(\Gamma_1(3))$ is given by 
$$
s_k^{-}(\Gamma_1(3))=\lfloor \frac{k-3}{6} \rfloor, \quad   
s_k^{+}(\Gamma_1(3))=
\begin{cases} \lfloor\frac{k-3}{6} \rfloor +1 &   \text{if $k \equiv_6 1$}, \\
 \lfloor \frac{k-3}{6} \rfloor &  \text{else.} \\ 
\end{cases}
$$

Inside the space $S_k^{\rm new}(\Gamma_0(9))$ 
we consider the eigenforms $f$ for which the
twist $f_{\chi}$, with $\chi$ the non-trivial character modulo $3$,
is also an eigenform in $S_k^{\rm new}(\Gamma_0(9))$. 
These generate a subspace $\Sigma_k \subset S_k^{\rm new}(\Gamma_0(N))$.
It may happen
that $f=f_{\chi}$ and then $f$ will have
Hecke eigenvalues $a(p)=0$ for $p\equiv_3 2$. This happens
for $k\equiv_3 1$, and thus $k\equiv_6 4 $.
A Hecke operator $T(p)$ with $p\equiv_3 2 $ will have characteristic
polynomial on this space that is a polynomial in $x^2$ or $x$ times a
polynomial in $x^2$. 
We decompose $\Sigma_k$ for even $k$
$$
\Sigma_k= S_k^{-}(\Gamma_0(9))\oplus S^{+}_k(\Gamma_0(9))
$$
where $S^{\pm}_k(\Gamma_0(9))$ is the $\pm$-eigenspace for the twisting
$f \mapsto f_{\chi}$.
We found experimentally the dimension formulas for 
$\Sigma_k^{\pm}(\Gamma_0(9))$
$$
s_k^{-}(\Gamma_0(9))=\begin{cases} \lfloor \frac{k+4}{12}\rfloor -1 &
k \equiv_{12} 10 \\
\lfloor \frac{k+4}{12} \rfloor & \text{else} 
\end{cases}
$$
and
$$
s_k^{+}(\Gamma_0(9))=\begin{cases} \lfloor \frac{k+4}{12}\rfloor +1 &
k \equiv_{12} 4  \\
\lfloor \frac{k+4}{12}\rfloor & \text{else.} 
\end{cases}
$$

For odd $k$ and $\chi$ the quadratic character modulo $3$ we will
use the space $S_k^{\rm new}(\Gamma_0(9),\chi)$ with dimension
$$
s_k^{\chi}(\Gamma_0(9))= \lfloor \frac{k+1}{6}\rfloor \, .
$$
Finally, another space we will use occurs for  $k\equiv_6 1 $.
We consider the subspace $S^{\chi}_k(\Gamma_0(9))$ of
$S_k^{\rm new}(\Gamma_0(9),\chi)$
with $\chi$ the quadratic character modulo $3$,
generated by eigenforms such that both $f$ and its twist $f_{\chi}$
both belong to $S_k^{\rm new}(\Gamma_0(9),\chi)$ and are distinct.
\end{subsection}
\begin{subsection}{One-dimensional lifts} \label{sec-lifts1}
Here we present the conjectured lifts of modular forms from $\rm U(1)$. For each of these lifts we see a $1$-dimensional piece in the cohomology of the corresponding local system $\WW_{\lambda}$ with trace of the Frobenius $F_{\nu}$ equal to $\nu^{a+b+2}$, compare Definition~\ref{def-hol1}.

\noindent{\bf Case 1.} For $a\equiv_6 3 $
there is a theta series
$\zeta_{a+3} \in S_{0,a+3,1}(\GaSq)$
with $\mathfrak{S}_4$-representation $\mathbf s_{1^4}$.
This eigenform is constructed in \cite[Prop.\ 2]{Fi}.
For $a=3$ we find the form $\zeta$. The Hecke
eigenvalue of $T(\nu)$
is given by
$$
\nu^{a+2}+(p+1)\bar{\nu}^{a+1} \, ,
$$
see \cite[Prop.\ 9]{Fi}. 

In both of the following two cases the lift in $S_{b,a+3,l}(\GaSq)$ will have a Hecke eigenvalue of $T(\nu)$ given by
$$
\nu^{a+b+2}+  \nu^{b+1}\bar{\nu}^{a+1} + \bar{\nu}^{a+b+2}\, .
$$

\noindent
{\bf Case 2.} For $(a,b)\equiv_6 (5,2) $ we find an eigenform in $S_{b,a+3,2}(\GaSq)$ with $\mathfrak{S}_4$-representation $\mathbf s_{4}$. 

The first example is found in $S_{2,8,2}(\GaSq)$ and it is described in \cite[Example 16.7]{C-vdG}. 

\noindent
{\bf Case 3.} We conjecture that there is a lift
$$S^{-}_{b+2}(\Gamma_0(9)) \to S_{b,a+3,1-a}(\GaSq)$$ 
with representation $\mathbf s_4$. 
\end{subsection}
\begin{subsection}{Two-dimensional lifts} \label{sec-lifts2}
The Hecke eigenvalue of $T(p)$ for an elliptic eigenform $f$ will be denoted by $a_p(f)$. For each lift of an elliptic eigenform $f$, i.e. a lift from $\rm{GL}(2)$, described in this section we see a $2$-dimensional piece in the cohomology of the corresponding local system $\WW_{\lambda}$ with trace of Frobenius $F_p$ equal to $a_p(f)\nu^{b+1}$ in all cases but the first, where we just see $a_p(f)$, compare Definition~\ref{def-hol2}.

\noindent
{\bf Case 1.} 
We conjecture that there is a lift
$$
S_{a+b+3}^{-}(\Gamma_0(9)) \to S_{b,a+3,2}(\GaSq)
$$
with $\mathfrak{S}_4$-representation $\mathbf s_4$ and the lift of an eigenform $f$ will have Hecke eigenvalue of $T(\nu)$ given by
$$
a_p(f)+\nu^{b+1}\bar{\nu}^{a+1} \, .
$$
An example is given by the lift from 
$S_8^{-}(\Gamma_0(9))$ to $S_{1,7,2}(\GaSq)$; the generating lift is $\Psi_1$  
described in \cite[p.\ 44]{C-vdG}.
Another example is the lift from $S_{12}^{-}(\Gamma_0(9))$ to 
$S_{0,12,2}(\GaSq)={\CC} \zeta^2$, 
already considered by Finis in \cite[p.\ 178]{Fi}.

In the following five cases we will have a lift of an elliptic eigenform $f$ to $S_{b,a+3,l}(\GaSq)$ and it will have a Hecke eigenvalue of $T(\nu)$ given by
$$
a_p(f) \nu^{b+1}+\bar{\nu}^{a+b+2} \, .
$$

\noindent
{\bf Case 2a.} We conjecture that there is a lift
$$
S_{a+2}(\Gamma_0(1)) \to S_{b,a+3,b}(\GaSq)
$$
with $\mathfrak{S}_4$-representation $\mathbf s_{2,1^2}+\mathbf s_{1^4}$.

The first example is the lift of $\Delta \in S_{12}(\Gamma_0(1))$ to a form
in $S_{1,13,1}(\GaSq)$ given in the table on page $43$ of
\cite{C-vdG}. Lifts of this type were constructed by Kudla in \cite[Thm.\ 5.3]{Kudla1}.

\noindent
{\bf Case 2b.} We conjecture that there is a lift
$$
S_{a+2}^{-}(\Gamma_1(3)) \to S_{b,a+3,b}(\GaSq)
$$
with $\mathfrak{S}_4$-representation $\mathbf s_4+ \mathbf s_{3,1}$. 

The first example is the lift from $S_9^{-}(\Gamma_1(3))$ 
to $S_{1,10,1}(\GaSq)$ which appears in the table on p.\ 43 of \cite{C-vdG}. Lifts of this type were constructed by Kudla as in Case 2a.

\noindent
{\bf Case 3.} We conjecture that there is a lift
$$
S_{a+2}^{\rm new}(\Gamma_0(3)) \to S_{b,a+3,b}(\GaSq)
$$
with $\mathfrak{S}_4$-representation $\mathbf s_{2,1^2}$. 

The lift of $(\eta(3\tau)\eta(\tau))^6 \in S_6(\Gamma_0(3))$
to an element of $S_{1,7,1}(\GaSq)$ is an example.

\noindent{\bf Case 4.} We conjecture that there is lift
$$
S_{a+2}^{-}(\Gamma_0(9)) \to S_{b,a+3,b}(\GaSq)
$$
with $\mathfrak{S}_4$-representation $\mathbf s_{3,1}$

An example is the form $F_{9,2}$ of Finis \cite[p.\ 151]{Fi} found in $S_{0,6,0}(\GaSq)$. 

\noindent
{\bf Case 5.} We conjecture that there is a lift
$$
S_{a+2}^{\chi}(\Gamma_0(9)) \to S_{b,a+3,b}(\GaSq)
$$
with $\mathfrak{S}_4$-representation $\mathbf s_{2,2}$.

The first example is the lift from $S^{\chi}_5(\Gamma_0(9))$ 
to $S_{3,6,0}(\GaSq)$ described in \cite[p.\ 50]{C-vdG}.
\end{subsection}
\begin{subsection}{Genuine Picard modular forms}\label{sec-genuine}
We conjecture that the space $S_{n(\lambda)}(\GaSq)$ decomposes into a Hecke-invariant direct sum of a subspace generated by all the lifts described in Section~\ref{sec-lifts1} and \ref{sec-lifts2}
and a subspace $S_{n(\lambda)}^{\rm gen}(\GaSq)$ which we call the space of genuine Picard modular forms. 
\end{subsection}
\begin{subsection}{Three-dimensional lifts}\label{threedimlifts}
As above, the Hecke eigenvalue of $T(p)$ for an elliptic eigenform $f$ of weight $k$ will be denoted by $a_p(f)$. In the cohomology of our local systems we see examples of $3$-dimensional pieces of the form ${\rm Sym}^2(M_f)$, with $M_f$ the motive associated to the elliptic eigenform, and with Hecke eigenvalue for $T(\nu)$ equal to $a_p(f)^2-p^{k-1}$.

We list these (conjectural) examples of lifts, which should be eigenforms in $S_{n(\lambda)}^{\rm gen}(\GaSq)$, without formulating more general conjectures.

\begin{itemize}
    \item The eigenform $f \in S_5^{\chi}(\Gamma_0(9))$ lifts to an eigenform in $S_{3,6,2}(\GaSq)$ with  $\mathfrak{S}_4$-representation $\mathbf s_{2,1^2}+\mathbf s_{1^4}$.
\item The eigenform $f \in S_6(\Gamma_0(3))$ lifts to an eigenform in $S_{4,7,2}(\GaSq)$ with  $\mathfrak{S}_4$-representation $\mathbf s_{3,1}$. 
\item The eigenform $f \in S_7^{\chi}(\Gamma_0(9))$ lifts to an eigenform in $S_{5,8,2}(\GaSq)$ with $\mathfrak{S}_4$-representation $\mathbf s_{1^4}$. 
\item The eigenform $f \in S_8^{-}(\Gamma_0(9))$ lifts to an eigenform in $S_{6,9,2}(\GaSq)$ with  $\mathfrak{S}_4$-representation $\mathbf s_{4}$. 
\item The eigenform $f \in S_9^{-}(\Gamma_1(3))$ lifts to an eigenform in $S_{7,10,2}(\GaSq)$ with  $\mathfrak{S}_4$-representation $\mathbf s_{4}$. 
\end{itemize}
\end{subsection}
\end{section}

\begin{section}{Conjectures on the cohomology of local systems} 
\label{sec-conj-main}
Recall the notation from Section~\ref{sec-not-locheck}. 
We will consider the normalized motivic (in the sense of Definition~\ref{def-norm}) Euler characteristic with compact support
$$
e_c(\lambda):=e_c^{\rm norm}(\mathcal{X}_{\GaSq}, {\WW}_{\lambda})
$$
and the inner variant of this $e_{!}(\lambda)$. Define also $e_{c,\mu}(\lambda)$ as in Section~\ref{modloc}. Put also $E_c(\lambda):=E_c(X_{\GaSq}, {\WW}_{\lambda})$, see Section~\ref{sec-numeric}.

Put $\lambda'=(b-i,-i,-a-i)$, so $(\lambda')'=\lambda$, and note by Proposition~\ref{prop-dual} that we have a duality 
$$e_c(\lambda')=\overline{e_c(\lambda)}.$$

\begin{subsection}{The main conjecture} \label{sec-extra}
In this section we formulate the main conjectures of the article. First we present a series of definitions of different contributions to the cohomology. 

The contributions making up the \emph{extraneous} contribution $e_{\rm extr}(\lambda)$, see Definition~\ref{def-S}, will consist of Hecke characters $\LL^{i,j}$, see Section~\ref{sec-hecke}, and motives from elliptic modular forms. Namely, to all the spaces of elliptic modular forms $S_k(\Gamma_0(1))$, $S^{\mathrm{new}}_k(\Gamma_0(3))$, $S^{\pm}_k(\Gamma_1(3))$, $S^{\pm}_k(\Gamma_0(9))$ and $S^{\chi}_k(\Gamma_0(9))$ introduced in Section~\ref{sec-not-elliptic}
we have corresponding motives (of twice their dimension)  $S[\Gamma_0(1),k]$, $S^{\mathrm{new}}[\Gamma_0(3),k]$, $S^{\pm}_k[\Gamma_1(3),k]$, $S^{\pm}[\Gamma_0(9),k]$ and $S^{\chi}[\Gamma_0(9),k]$ with the property that for any prime $p$ the trace of the Hecke operator $T(p)$ on the space of modular forms equals the trace of Frobenius $F_{\nu}$ on the corresponding motive.  

\begin{definition}
We define the $\mathfrak{S}_4$-representations
$$
\alpha_{j}= \begin{cases} \mathbf s_{4} & j\equiv_6 0 \\
\mathbf s_{1^4} & j \equiv_6 3  \\ 
0 & \text{else} \\
 \end{cases}
\quad \text{ and }\quad
\beta_{j}= \begin{cases} \mathbf s_{3,1} & j\equiv_6 0\\
\mathbf s_{2,1^2}  & j \equiv_6 3\\
0 & \text{else.} \\  
 \end{cases}
$$
and we define $\delta_j$ to be $1$ if $j \equiv_6 0$ and $0$ else. 
\end{definition}

Proposition~\ref{prop-eis} gives a formula for the (normalized) Eisenstein cohomology $e_{\rm Eis}(\lambda)=e_c(\lambda)-e_{!}(\lambda)$ for all regular $\lambda$. Proposition~\ref{prop-eis} is only formulated in Betti cohomology, but Harder's result is actually motivic in the sense of Section~\ref{modloc}. We generalize the formula in the following definition. 

\begin{definition} \label{def-eisp}
We define $e_{\rm Eis}'(\lambda)$ 
as
\begin{multline*}
-(\alpha_i+\beta_i)\, \LL^{0,0}+ (\alpha_{i-b-1}+\beta_{i-b-1}) \, \LL^{b+1,0}+
(\alpha_{i+a+1}+\beta_{i+a+1})\, \LL^{0,a+1}  \\
-\begin{cases}
\alpha_{i-1}(\LL+1)\,  \LL^{b+1,0} & \mathrm{if} \; a=0 \; \mathrm{and} \; b\equiv_2 0\\ \alpha_{i+1}(\LL+1)\,  \LL^{0,a+1} &\mathrm{if } \; b=0 \; \mathrm{and} \; a\equiv_2 0 \,
\end{cases}
\end{multline*}

\end{definition}

All the following contributions should be found in $e_!(\lambda)$. Note that the Hodge degrees of all these contributions, for a regular local system, are either $(a+b+2,0)$, $(b+1,a+1)$ or $(0,a+b+2)$, as they need to be, see Section~\ref{sec-bgg}. Note also that the sign is always positive, which it should be for regular local systems since in that case only the second inner cohomology group can be non-zero, see Section~\ref{sec-bgg}.

\begin{definition}
We define the central endoscopic term $e_{\rm ce}(\lambda)$ as 
$$
\begin{aligned}
\alpha_i\left( s_{a+b+3}(\Gamma_0(1))+s_{a+b+3}^{-}(\Gamma_1(3))\right) \, \LL^{b+1,a+1} & \\
+\beta_i \left( s_{a+b+3}(\Gamma_0(1))+s_{a+b+3}^{\rm new}(\Gamma_0(3))+s_{a+b+3}^{+}(\Gamma_1(3))\right) \, \LL^{b+1,a+1} & \\
+\beta_{i+3}\, s_{a+b+3}^{+}(\Gamma_0(9))\, \LL^{b+1,a+1} \, + \,  
(\delta_i+\delta_{i+3}) \, s[2,2] \, s_{a+b+3}^{\chi}(\Gamma_0(9))\, \LL^{b+1,a+1}\,  & . \\
+ \text{$\alpha_i \, \LL^{b+1,a+1} \;$ if $(a,b)\equiv_6 (5,5)$.} \\ 
\end{aligned}
$$
\end{definition}

The following two contributions are connected to the lifted (holomorphic) forms in $S_{b,a+3,i+2}(\GaSqOne)$ described in the conjectures of Section~\ref{sec-conj-lifts}. 

\begin{definition} (holomorphic $1$-dimensional lifts) \label{def-hol1}
We define $e_{\rm 1\ell}(\lambda)$ as
$$
\begin{aligned}
\alpha_{i+a+4} s^{-}_{b+2}(\Gamma_0(9)) \, \LL^{a+b+2,0}+
\begin{cases}
\alpha_{i+3} \, \LL^{a+b+2,0} & \text{if $(a,b)\equiv_6 (5,2)$,} \\
\alpha_{i+4} \, \LL^{a+2,0} & \text{if $a\equiv_6 3$ and $b=0$.}\\
\end{cases} 
\end{aligned}
$$
\end{definition}

\begin{definition} (holomorphic $2$-dimensional lifts) \label{def-hol2}
We define $e_{\rm 2\ell}(\lambda)$ as
$$
\begin{aligned}
\alpha_{i+3} S^{-}[\Gamma_0(9),a+b+3] & \\
+ \alpha_{i-b-1} \left( S[\Gamma_0(1),a+2]+S^{-}[\Gamma_1(3),a+2] \right) \LL^{b+1,0} &\\
\beta_{i-b-1}\left( S[\Gamma_0(1),a+2]+S^{\rm new}[\Gamma_0(3),a+2]+
S^{-}[\Gamma_1(3),a+2]\right) \LL^{b+1,0} & \\
+\beta_{i-b+2} S^{-}[\Gamma_0(9),a+2]\,  \LL^{b+1,0} &\\
+(\delta_{i-b-1}+\delta_{i-b+2})s[2,2] \, S^{\chi}[\Gamma_0(9),a+2] \, \LL^{b+1,0} & .\\
\end{aligned}
$$
\end{definition}

Let $\bar{S}_{j,k,l}(\GaSq)$ be isomorphic as a vector space to $\bar{S}_{j,k,l}(\GaSq)$ but such that the Hecke operator $T(\nu)$ acts on $\bar{S}_{j,k,l}(\GaSq)$ as $T(\bar \nu)$ acts on $S_{j,k,l}(\GaSq)$.
The following two contributions (i.e. Definition~\ref{def-anthol1} and Definition~\ref{def-anthol2}) are connected to the lifted anti-holomorphic forms. In other words connected to the lifted (holomorphic) forms in $\bar{S}_{a,b+3,-i-1}(\GaSq)$ except for the contribution of the form $\alpha_{i+3} S^{-}[\Gamma_0(9),a+b+3]$. Compare this with Proposition~\ref{prop-dual}.

\begin{definition} \label{def-anthol1} (anti-holomorphic $1$-dimensional lifts)
We define $e_{\rm \overline{1\ell}}(\lambda)$ as
$$
\begin{aligned}
\alpha_{i-b+2} s^{-}_{a+2}(\Gamma_0(9)) \, \LL^{0,a+b+2} +
\begin{cases}
\alpha_{i+3} \, \LL^{0,a+b+2} & \text{if $(a,b)\equiv_6 (2,5)$,}\\
\alpha_{i-2} \, \LL^{0,b+2} & \text{if  $a=0$ and $b\equiv_6 3 $.} \\
\end{cases}
\end{aligned}
$$
\end{definition}

\begin{definition} \label{def-anthol2} (anti-holomorphic $2$-dimensional lifts)
We define $e_{\rm \overline{2\ell}}(\lambda)$ as
$$
\begin{aligned}
\alpha_{i+a+1} \left( S[\Gamma_0(1),b+2]+S^{-}[\Gamma_1(3),b+2]\right) \, \LL^{0,a+1} & \\
+\beta_{i+a+1} \left( s[\Gamma_0(1),b+2]+S^{\rm new}[\Gamma_0(3),b+2]
+S^{-}[\Gamma_1(3),b+2]\right) \, \LL^{0,a+1} & \\
+ \beta_{i+a+4} \, S^{-}[\Gamma_0(9),b+2] \, \LL^{0,a+1} & \\
+ (\delta_{i+a+1}+\delta_{i+a+4}) s[2,2] \, S^{\chi}[\Gamma_1(9),b+2] \, \LL^{0,a+1} & .\\
\end{aligned}
$$
\end{definition}

\begin{definition} \label{def-S} We put 
$$
e_{\rm extr}(\lambda)=
e'_{\rm Eis}(\lambda) +e_{\rm ce}(\lambda)+e_{1\ell}(\lambda)+e_{2\ell}(\lambda)+e_{\overline{1\ell}}(\lambda)+e_{\overline{2\ell}}(\lambda)
$$
and we define $e_{\rm extr,\mu}(\lambda)$ as in Section~\ref{modloc}.
 \end{definition}

\begin{mainconjecture} \label{conj-main}
We conjecture that for any $\lambda \neq (0,0,0)$, $\nu$ and $\mu$, if $i \equiv_2 a+b$ then, 
$$
{\rm Tr}(T({\nu}),S^{\rm gen}_{n(\lambda)}(\GaSq)^{\mu}) ={\rm Tr}(F_{\nu},e_{c,\mu}(\lambda)-e_{\rm extr,\mu{}{}}(\lambda))\,.
$$
\end{mainconjecture}

Assuming that the conjecture is true, this gives a possibility to compute the trace of the Hecke operators $T({\nu})$ by counts of points over ${\FF}_p$ as described in Section~\ref{sec-counts}.

\begin{remark} \label{rmk-dual} Note that $e_{\rm extr}(\lambda')=\overline{e_{\rm extr}(\lambda)}$. 

The main conjecture then implies that there is a Hecke-invariant isomorphism between $S_{j,k+3,l}^{\rm gen}(\Gamma[\sqrt{-3})^{\mu}$ and $\bar S_{k,j+3,1-l}^{\rm gen}(\Gamma[\sqrt{-3})^{\mu}$. 
Note that in general 
$$\dim S_{j,k+3,l}(\Gamma[\sqrt{-3}])\neq \dim S_{k,j+3,1-l}(\Gamma[\sqrt{-3}]).$$
According to Theorem \ref{dim-Ajkl} the difference of dimensions is equal
to 
$$
\begin{matrix}
&& l \equiv_3 0 & l\equiv_3 1 & l\equiv_3 2 \\
j\equiv_3 0 && k-1 & 1-j & 0 \\
j\equiv_3 1 && 1-j & k-1 & 0 \\
j\equiv_3 2 && 0 & 0 & k-j \\
\end{matrix}
$$
which is a consequence of the presence lifts.
\end{remark}

\begin{definition}
We define $E_{\rm extr}(\lambda)$ as $e_{\rm extr}(\lambda)$ but replacing 
\begin{itemize}
    \item[$\star$] $\LL^{i,j}$ by $1$ for any $i,j$
    \item[$\star$] $S[\Gamma_0(1),k]$, $S^{\mathrm{new}}[\Gamma_0(3),k]$, $S^{\pm}[\Gamma_1(3),k]$, $S^{\pm}[\Gamma_0(9),k]$,  $S^{\chi}[\Gamma_0(9),k]$ respectively by 
$2\dim S_k(\Gamma_0(1))$, $2\dim S^{\mathrm{new}}_k(\Gamma_0(3))$, $2\dim S^{\pm}_k(\Gamma_1(3))$, $2\dim S^{\pm}_k(\Gamma_0(9))$, $2\dim S^{\chi}_k(\Gamma_0(9))$ for any $k$,
\end{itemize}
in the formulas of Definition~\ref{def-eisp} until Definition~\ref{def-anthol2}. This gives an element in the Galois group of representations of $\fS_4$.
\end{definition}

\begin{conjecture} \label{conj-dim} We conjecture that for any $\lambda$, 
$$\dim_{\fS_4} S_{n(\lambda)}^{\rm gen}(\GaSq) = \frac{1}{3} \bigl(E_{c}(\lambda)-E_{\rm extr}(\lambda) \bigr).
$$
\end{conjecture}
\end{subsection}
\begin{subsection}{A congruence modulo $9$} 
Our experimental data lead us to conjecture a congruence for the eigenvalues
of Hecke operators.

\begin{conjecture} \label{conj-cong}
For any $j,k,l \geq 0$, $\nu$ and $\mu \vdash 4$ we conjecture that
$$
{\rm Tr}(T(\nu),S_{j,k+3,l}^{\rm gen}(\GaSq)^{\mu})
\equiv_9 3 \, \dim S_{j,k+3,l}^{\rm gen}(\GaSq)^{\mu}\, .
$$
\end{conjecture}

\begin{remark}
If one uses the evidence of the results in \cite{C-vdG} one might also conjecture
that for a prime $p\equiv 2 \, (\bmod \, 3)$ the trace on the space of genuine forms is
divisible by $9$.
\end{remark}
\end{subsection}

For the lifted forms described in Section~\ref{threedimlifts} 
this means that for a prime $p\equiv_3 1$ the congruence
$$a_p(f)^2-p^{k-1} \equiv_9 3$$
should hold and similarly a congruence 
$a_p(f)^2-p^{k-1} \equiv_9 0$ for primes $p\equiv_3 2$. 
\begin{subsection}{Evidence}
The conjectures of this section and the previous were based upon the computations described in Section~\ref{sec-computer}. Here we list a series of regularities in this data that lends credence to the conjectures. 

The following holds for all $\lambda$ such that $a+b+2 \leq 40$. 
\begin{itemize}
    \item[$\star$] The integer  $$E_{c}(\lambda)-E_{\rm extr}(\lambda)$$ is divisible by $3$.
    \item[$\star$] We find, using Theorem~\ref{dim-Ajkl}, that  
    $\dim S_{n(\lambda)}(\GaSqOne)$ equals 
   $$\frac{1}{3}(E_{c}(\lambda)-E_{\rm extr}(\lambda)),$$ 
   when replacing $\mathbf s_{\mu}$ with $\dim \mathbf s_{\mu}$,
   and adding the dimension of the lifts described in Section~\ref{sec-conj-lifts}. 
    \item[$\star$] If $$E_{c}(\lambda)-E_{\rm extr}(\lambda)=0$$ then 
    $$\Tr(F_q,e_{c}(\lambda)- e_{\rm extr}(\lambda))=0,$$
    for $q \leq 67$ and $q\equiv_3 1$.
    \item[$\star$] We have that (compare with Conjecture~\ref{conj-cong})
    $$ \Tr(F_q,e_{c}(\lambda)- e_{\rm extr}(\lambda)) \equiv_9 E_{c}(\lambda)-E_{\rm extr}(\lambda),
    $$
    for $q \leq 67$ and $q\equiv_3 1$.
    \item[$\star$] All traces computed (for $j=0$) in \cite{Fi} match with the ones computed using Conjecture~\ref{conj-main} for $p \leq 67$ and $p \equiv_3 1$.
    \item[$\star$] The ring of scalar valued modular forms, i.e. when $j=0$, is given in Proposition \ref{scalar-rings}. This gives a formula for $\dim S_{0,k+3,l}^{\rm gen}(\GaSq)^{\mu}$ for any $k,l,\mu$, which 
    matches the one given by Conjecture~\ref{conj-dim} (for $a+b+2 \leq 40$). For instance we have that  $$S_{0,6k+3,2}(\Gamma[\sqrt{-3}])=M_{0,6k-9,0}(\Gamma[\sqrt{-3}])\, \zeta^2$$ 
    and 
    $$\dim_{\fS_4} S_{0,6k+3,2}(\Gamma[\sqrt{-3}])={\rm Sym}^{2k-3}(\mathbf s_{2,1^2})$$ 
    for $k\geq 2$. 
    \item[$\star$] All traces computed in \cite{C-vdG} match with the ones computed using Conjecture~\ref{conj-main} for $p \leq 67$ and $p \equiv_3 1$. 
\end{itemize}
Note that the information
$$\Tr(F_{q^r},e_{c}(\lambda)- e_{\rm extr}(\lambda)),$$
for $r=1,\ldots,(E_{c}(\lambda)-E_{\rm extr}(\lambda))$ gives a way to compute the characteristic polynomial of $F_q$ acting on $e_{c}(\lambda)- e_{\rm extr}(\lambda)$, assuming that it is effective of dimension $E_{c}(\lambda)-E_{\rm extr}(\lambda)$. 

For all $\lambda$ such that $a+b+2 \leq 40$ and  $E_{c}(\lambda)-E_{\rm extr}(\lambda)=3$, see further in Section~\ref{sec-onedim}, the characteristic polynomial for $q=4$ and the partial information for $q=7$ has the expected structure (namely the one derived from the results of Section~\ref{sec-heckeoperator}).

See also Section~\ref{sec-harder} for evidence coming from congruences studied by Harder.
\end{subsection}
\begin{subsection}{Modules of vector-valued forms}
Define
$$
\mathcal{M}_j= \mathcal{M}_j^0 \oplus \mathcal{M}_j^1 \oplus \mathcal{M}_j^2
$$
with
$$
\mathcal{M}_j^{\ell}= \oplus_k M_{j,k,l}(\GaSq) \, .
$$
Then $\mathcal{M}_j$ is a module over $\mathcal{M}_0$; for $\mathcal{M}_0$ see Proposition~\ref{scalar-rings}. Guided by the heuristics of our conjectures
the structure of some modules $\mathcal{M}_j$ was determined in \cite{C-vdG}, e.g.\ for $j=1,2,3$. For example,
the module $\mathcal{M}_1^0$ is generated over $\mathcal{M}_0^0$ by 
three forms $\Phi_0,\Phi_1,\Phi_2 \in S_{1,7,0}(\Gamma_1[\sqrt{-3}])$ satisfying
a relation $\varphi_0\Phi_0+\varphi_1\Phi_1+\varphi_2\Phi_2=0$ with
$\varphi_0,\varphi_1,\varphi_2$ generators of the ring $\mathcal{M}_0^0$. 
For a table of Hecke eigenvalues of the $\Phi_i$ 
we refer to \cite[Table 7]{C-vdG}.
\end{subsection}
\begin{subsection}{Conjecture for the moduli space of genus $2$ curves}
In this section we will be brief and give a similar conjecture to the one above but in the case of genus $2$. 

Define the normalized compactly supported Euler characteristic $e^{\rm norm}_c$ analogously to Definition~\ref{def-norm}. Recall the notation from Section~\ref{sec-constant-g2} and define the representation,  
$$
\alpha_{k}= \begin{cases} \mathbf s_{2} \tilde{\mathbf s}_{2} & k\equiv_2 0\\
\mathbf s_{1,1} \tilde{\mathbf s}_{2} & k \equiv_2 1  \end{cases}
\quad \text{and}\quad
\beta_{k}= \begin{cases}\mathbf s_{1,1} \tilde{\mathbf s}_{1,1}  & k\equiv_2 0 \\
\mathbf s_{2} \tilde{\mathbf s}_{1,1}  & k \equiv_2 1\\
\end{cases}
$$

Let $W_3$ denote the Fricke operator and for any prime $p$ let $T(p)$ be the Hecke operator. For the proofs of the following properties of $W_3$, see \cite{Asai}.

If $k>0$ is even, and $f \in S_{k+2}(\Gamma_1(3))$ is an eigenform with $T(p)f=a_pf$ then $a_p=\bar a_p$ for all primes $p$. Moreover,  $a_3=\pm 3^{k/2}$ and $W_3(f)=-\mathrm{sgn}(a_3)f$.  The $\pm$-spaces of $W_3$ are clearly Hecke invariant and we denote the $\pm$-eigenspaces of $S_{k+2}(\Gamma_1(3))$ by $S^{\pm}_{k+2}(\Gamma_1(3))$. Define $S^{\pm}[\Gamma_1(3),k+2]$ analogously. 

If $k$ is odd, then $S_{k+2}(\Gamma_1(3))=S_{k+2}(\Gamma_0(3),\chi)$, where $\chi$ is the Dirichlet character of order $2$. 
If $f \in S_{k+2}(\Gamma_0(3),\chi)$ is an eigenform with $T(p)f=a_pf$, then $\bar a_p=\chi(p) a_p$ for all $p \nmid 3$ and $W_3(f)=c\bar f$ for some $c\in \mathbb C$ with $|c|=1$, and $W^2_3(f)=-f$. If $f \neq \bar f$ then $\pm i \bar c f+\bar f$ is an eigenvector for $W_3$ with eigenvalue $\mp i$. If $p \equiv_3 1$ then both these are also eigenvectors of $T(p)$ with eigenvalue $a_p$. For $k \equiv_3 2$ then there is an eigenvector $f \in S_{k+2}(\Gamma_0(3),\chi)$ such that $a_3=(-3)^{(k+1)/2}$, 
$W_3(f)=(-1)^{(k+1)/2} i f$ and 
$$a_p(f)=\mathrm{Tr}(F_p,\LL^{k+1,0}+\LL^{0,k+1}).$$
We denote the $\pm i$-eigenspaces of $S_{k+2}(\Gamma_1(3))$ by $S^{\pm}_{k+2}(\Gamma_1(3))$, and we define $S^{\pm}[\Gamma_1(3),k+2]$ analogously. 

  \begin{conjecture}
    For any $k>0,l \geq 0$ such that $k \equiv_6 l$,
\begin{multline*}
  e^{\rm norm}_c(\cX^{(2)}_{\GaSq},{\WW}_{k,l})=-\alpha_k-\beta_k  - S^+[\Gamma_1(3),k+2] \, \alpha_k 
  -S^-[\Gamma_1(3),k+2] \, \beta_k \\ + \delta_{k+1}(\LL^{k+1,0}\alpha_k+\LL^{0,k+1}\beta_k)+
  \delta_{k+7}(\LL^{0,k+1}\alpha_k+\LL^{k+1,0}\beta_k)
\end{multline*}
as elements of $K^{\fS_2\times \fS_2}_0({\rm Gal}_F)$, and where $\delta_i=1$ if $i \equiv_{12} 0$ and $0$ otherwise. 
\end{conjecture}

For $k=0$, put $S^+[\Gamma_1(3),2]=-\LL-1$ and $S^-[\Gamma_1(3),2]=0$ to make the formula correct, by Proposition~\ref{prop-genus2}. Looking at how $\fS_2 \times \fS_2$ acts on the two cusps of the Bailey-Borel compactification we find that the 
 Eisenstein contribution to $e^{\rm norm}_c(\cX^{(2)}_{\GaSq},{\WW}_{k,l})$, for $k \equiv_6 l$, is equal to $-\alpha_k-\beta_k$ 
 (compare the proof of Proposition~\ref{prop-eis}). 
\end{subsection}
\end{section}

\begin{section}{Examples} 
We now illustrate our conjectures and calculations by examples.

\begin{subsection}{Explicit examples}\label{sec-ex-individ} Let us study a series of local systems in some detail. 

\begin{subsubsection}{$\lambda=(10,4,4)$}
In the cohomology of this local system we expect to find contributions from the forms in $S_{0,9,0}(\GaSqOne)$. 

From the formula described in Section~\ref{sec-euler} we get the numerical Euler characteristic, 
$$E_c(X_{\GaSq},{\WW}_{\lambda})=\mathbf s_4 + 2 \,\mathbf s_{3, 1} +  3\,\mathbf s_{2, 1^2} + \mathbf s_{1^4}.$$

Following Section~\ref{sec-extra} the only non-zero contributions we 
have to $e_{\rm extr}(\lambda)$ are   
$$e'_{\mathrm{Eis}}(\lambda)= \LL^{1,0}(\mathbf s_{2,1^2} + \mathbf s_{1^4}),$$
$$e_{\overline{1\ell}}(\lambda)=\LL^{0,8} \mathbf s_4
$$
and 
$$e_{{2\ell}}(\lambda)=S^{-}[\Gamma_0(9),8] \, \LL^{1,0} \mathbf s_{3,1}+S^{\rm new}[\Gamma_0(3),8] \, \LL^{1,0} \mathbf s_{2,1^2}.
$$
In $e_{{2\ell}}(\lambda)$, the $\mathbf s_{3,1}$-term comes from the Kudla lift denoted $F_{9,1}=\varphi_0\varphi_1(\varphi_0-\varphi_1)$ and the $\mathbf s_{2,1^2}$-term comes from the Kudla lift of weight $9$ 
denoted $F_{9,2}$ by Finis (\cite[Tables p.\ 151 and 177]{Fi}).

Removing the extraneous contributions should leave us with contributions from the genuine Picard modular cusp forms. Starting with the numerical Euler characteristic we get 
$$E_c(X_{\GaSq},{\WW}_{\lambda})-E_{\rm extr}(X_{\GaSq},{\WW}_{\lambda})=0,$$
so there should be no genuine forms. And indeed we find that 
$${\rm Tr}(F_{\nu},e_c(\lambda))-{\rm Tr}(F_{\nu},e_{\rm extr}(\lambda))$$
which conjecturally equals
$${\rm Tr}(T(\nu),S_{0,9,0}^{\rm gen}(\GaSq))$$
is $0$ for all $p \leq 67$.

This also fits (adding the lifts) with the formula  
$$\dim_{\fS_4} S_{0,9,0}(\GaSq) = \mathbf s_{3, 1} + \mathbf s_{2, 1^2}$$
that we find by Proposition~\ref{scalar-rings}. 
\end{subsubsection}

\begin{subsubsection}{$\lambda=(16,4,4)$}
In the cohomology of this local system we expect to find contributions from the forms in $S_{0,15,0}(\GaSq)$. 

Again, from the formula in Section~\ref{sec-euler} we get, 
$$E_c(X_{\GaSq},{\WW}_{\lambda})=\mathbf s_4 + 5 \,\mathbf s_{3, 1} + 3\, \mathbf s_{2^2} + 7\,\mathbf s_{2, 1^2} + \mathbf s_{1^4}. $$

Similarly to the previous case, following Section~\ref{sec-extra}, we find that $e_{\rm extr}(\lambda)$ consists of 
$$e'_{\mathrm{Eis}}(\lambda)= \LL^{1,0}(\mathbf s_{2,1^2} + \mathbf s_{1^4}),$$
$$e_{\overline{1\ell}}(\lambda)=\LL^{0,14} \mathbf s_4
$$
and 
$$e_{{2\ell}}(\lambda)=S^{-}[\Gamma_0(9),14] \, \LL^{1,0} \mathbf s_{3,1}+S^{\rm new}[\Gamma_0(3),14] \, \LL^{1,0} \mathbf s_{2,1^2}.
$$

Removing the extraneous contributions should leave us with contributions from the genuine Picard modular cusp forms. Starting with the numerical Euler characteristic we get 
$$E_c(X_{\GaSq},{\WW}_{\lambda})-E_{\rm extr}(X_{\GaSq},{\WW}_{\lambda})=3 \,\mathbf s_{3, 1} + 3\,\mathbf s_{2^2}.$$

Dividing this expression by $3$ gives the conjectural result 
$$\dim_{\fS_4} S^{\rm gen}_{0,15,0}(\GaSq) = \mathbf s_{3, 1} + \mathbf s_{2^2}.$$
Together with the lifts we get 
$$\dim_{\fS_4} S_{0,15,0}(\GaSq) = 2\mathbf s_{3, 1} + \mathbf s_{2^2} + 3\,\mathbf s_{2, 1^2},$$
which fits with the formula we find by Proposition~\ref{scalar-rings}.

We can then compute  
$${\rm Tr}(F_{\nu},e_c(\lambda))-{\rm Tr}(F_{\nu},e_{\rm extr}(\lambda))$$
which conjecturally equals,  
$${\rm Tr}(T(\nu),S_{0,15,0}^{\rm gen}(\GaSq)).
$$

For the two $1$-dimensional isotypic components of the space of genuine forms we then (conjecturally) get Hecke eigenvalues, and a few of them are given in the following table. Note that the analogue of the Ramanujan conjecture for this situation holds; $N(\xi_{\nu}) \leq 3 N(\nu)^{a+b+2}$ for eigenvalues $\xi_{\nu}$. Similar observations can be made for the other tables appearing in this
section.

\begin{footnotesize}
\smallskip
\vbox{
\bigskip\centerline{\def\quad{\hskip 0.6em\relax}
\def\quod{\hskip 0.5em\relax }
\vbox{\offinterlineskip
\hrule
\halign{&\vrule#&\strut\quod\hfil#\quad\cr
height2pt&\omit&&\omit&&\omit&\cr
&$p$ && $S^{\rm gen}_{0,15,0}(\GaSq)^{(2^2)}$ && $S^{\rm gen}_{0,15,0}(\GaSq)^{(3,1)}$  & \cr
\noalign{\hrule}
& $7$   && $ -388107 \rho - 1608891$ && $-524625 - 205857 \rho$ & \cr
& $13$  && $-60967989 \rho - 9061701$ && $-36504663 + 20888505 \rho$ & \cr
& 19 && $-578216997 \rho - 665720736$ &&$ -398615136 - 1035916731 \rho$ & \cr
& 31 && $-690422256 \rho - 8829510909$ && $ 32032766937 + 14052080592 \rho$& \cr
& 37 && $  -111679368147 \rho - 59483009571$ && $ 30023590017 - 12661429743 \rho$& \cr
& 43 && $-98981609184 \rho + 131622854187$ && $298590045213 + 634311769248 \rho$& \cr
} \hrule}
}}
\end{footnotesize}
\end{subsubsection}

\begin{subsubsection}{$\lambda=(32,2,2)$} In the cohomology of this local system we expect to find contributions from the forms in $S_{0,33,1}(\GaSq)$. 

Again, from the formula in Section~\ref{sec-euler} we get, 
$$E_c(X_{\GaSq},{\WW}_{\lambda})=9 \, \mathbf s_4 + 27 \,\mathbf s_{3, 1} + 9\, \mathbf s_{2^2} + 19\,\mathbf s_{2, 1^2} + 2\mathbf s_{1^4}. $$

Following Section~\ref{sec-extra} we find that that $e_{\rm extr}(\lambda)$ only consists of 
$$e'_{\mathrm{Eis}}(\lambda)= \LL^{0,31}\mathbf s_{2,1^2} - \LL^{1,32} \mathbf s_{1^4}.$$

Removing this contribution from the Euler characteristic and dividing by $3$, as in the previous example, we get the following conjecture,  $$\dim_{\fS_4} S^{\rm gen}_{0,33,1}(\GaSq) = 3\mathbf s_{4}+9\,\mathbf s_{3, 1} + 3 \mathbf s_{2^2} + 6\,\mathbf s_{2, 1^2} + \mathbf s_{1^4}.$$
Since there are no lifts, this is the same as the dimensions of all cusp forms and it equals the formula $\mathbf s_{1^4}\mathrm{Sym}^9(\mathbf s_{2,1^2})$,  
found using Proposition~\ref{scalar-rings}.

Some (conjectural) Hecke eigenvalues for the $1$-dimensional isotypic component of the space of genuine forms corresponding to $\mathbf s_{1^4}$ are given in the following table. 

\begin{footnotesize}
\smallskip
\vbox{
\bigskip\centerline{\def\quad{\hskip 0.6em\relax}
\def\quod{\hskip 0.5em\relax }
\vbox{\offinterlineskip
\hrule
\halign{&\vrule#&\strut\quod\hfil#\quad\cr
height2pt&\omit&&\omit&\cr
&$p$ && $S^{\rm gen}_{0,33,1}(\GaSq)^{(1^4)}$ & \cr
\noalign{\hrule}
& $7$   && $-17187741337239 \rho - 27371045932368$ & \cr
& $13$  && $ 619757358250752891 \rho - 73897512261622296$ & \cr
& $19$  && $32397975717682438611 \rho + 161109729684241303755$ & \cr
& $31$  && $-450614269323285049766016 \rho + 463109207345192219515905$ & \cr
& $37$  && $ 4464950074069806168802623 \rho - 679365937587169490840376$ & \cr
& $43$  && $-62575475768038597846807512 \rho - 83275045472246397000970011$ & \cr
} \hrule}
}}
\end{footnotesize}

\end{subsubsection}

\begin{subsubsection}{$\lambda=(7,1,-2)$} In the cohomology of this local system we expect to find contributions from the forms in $S_{3,9,0}(\GaSqOne)$. Note that the modular forms occuring here are described in \cite[Prop. 15.2]{C-vdG}.

From the formula described in Section~\ref{sec-euler} we get 
$$E_c(X_{\GaSq},{\WW}_{\lambda})=\mathbf s_4 + 5 \,\mathbf s_{3, 1} + 6\, \mathbf s_{2^2} + 9\,\mathbf s_{2, 1^2} + 4\,\mathbf s_{1^4}.$$

From Proposition~\ref{prop-eis} and Section~\ref{sec-extra} we get 
$$e'_{\mathrm{Eis}}(\lambda)=e_{\mathrm{Eis}}(\lambda)= \LL^{4,0}(\mathbf s_{2,1^2} + \mathbf s_{1^4}).$$
The only other non-zero contributions we have to $e_{\rm extr}(\lambda)$ are 
$$e_{\overline{1\ell}}(\lambda)=\LL^{0,11} \mathbf s_4
$$
and 
$$e_{{2\ell}}(\lambda)=S^{-}[\Gamma_0(9),8] \, \LL^{4,0} \mathbf s_{3,1}+S^{\rm new}[\Gamma_0(3),8] \, \LL^{4,0} \mathbf s_{2,1^2}.
$$

Again, removing the extraneous contributions from the Euler characteristic and dividing by $3$, we get the following conjecture,  
$$\dim_{\fS_4} S^{\rm gen}_{3,9,0}(\GaSq) = \mathbf s_{3, 1} + 2\, \mathbf s_{2^2} + 2\,\mathbf s_{2, 1^2} + \mathbf s_{1^4}.$$
Together with the lifts we get the conjecture, 
$$\dim_{\fS_4} S_{3,9,0}(\GaSq) = 2\mathbf s_{3, 1} + 2\, \mathbf s_{2^2} + 3\,\mathbf s_{2, 1^2} + \mathbf s_{1^4}.$$
This conjectural expression fits with that Theorem~\ref{dim-Ajkl} tells us, namely that,  
$$\dim S_{3,9,0}(\GaSq)=20.$$

For the two $1$-dimensional isotypic components of the space of genuine forms we then (conjecturally) get Hecke eigenvalues, and a few of them are given in the following table. 

\begin{footnotesize}
\smallskip
\vbox{
\bigskip\centerline{\def\quad{\hskip 0.6em\relax}
\def\quod{\hskip 0.5em\relax }
\vbox{\offinterlineskip
\hrule
\halign{&\vrule#&\strut\quod\hfil#\quad\cr
height2pt&\omit&&\omit&&\omit&\cr
&$p$ && $S^{\rm gen}_{3,9,0}(\GaSq)^{(3,1)}$ && $S^{\rm gen}_{3,9,0}(\GaSq)^{(1^4)}$  & \cr
\noalign{\hrule}
& $7$   && $-2661-3735\rho$  &&   $  -39273-37755\rho $ & \cr
& $13$   && $697611-853785\rho $  &&   $ -616209-1939509\rho $ & \cr
& $19$   && $ -4019046-4493727\rho $  &&   $2924922+16469397\rho $ & \cr
& $31$   && $236296587+26549946\rho $  &&   $ -13532361-40067046\rho$ & \cr
& $37$   && $381974925-151949367\rho $  &&   $-294789795-270210663\rho $ & \cr
& $43$   && $685398387+28100862\rho $  &&   $ 1093524015+1099688094\rho $ & \cr
} \hrule}
}}
\end{footnotesize}
\end{subsubsection}

\begin{subsubsection}{$\lambda=(11,5,2)$} In the cohomology of this local system we expect to find contributions from the forms in $S_{3,9,1}(\GaSq)$. 

Again, from the formula in Section~\ref{sec-euler} we get, 
$$E_c(X_{\GaSq},{\WW}_{\lambda})=4 \, \mathbf s_4 + 7 \,\mathbf s_{3, 1} + 5\, \mathbf s_{2^2} + 6\,\mathbf s_{2, 1^2} + 3\,\mathbf s_{1^4}. $$

From Proposition~\ref{prop-eis} and Section~\ref{sec-extra} we get that the only non-zero contribution we have to $e_{\rm extr}(\lambda)$ are 
$$e'_{\mathrm{Eis}}(\lambda)=e_{\mathrm{Eis}}(\lambda)= \LL^{0,7}(\mathbf s_{4} + \mathbf s_{3,1}).$$
and
$$e_{\overline{2\ell}}(\lambda)=\LL^{0,7} S^{\chi}[\Gamma_0(9),5]  \mathbf s_{2,2}.$$

Removing this contribution from the Euler characteristic and dividing by $3$, as in the previous example, we get the following conjecture,  $$\dim_{\fS_4} S^{\rm gen}_{3,9,1}(\GaSq) = \mathbf s_{4}+2\,\mathbf s_{3, 1} +  \mathbf s_{2^2} + 2\,\mathbf s_{2, 1^2} + \mathbf s_{1^4}.$$
This is the same as the dimensions of all cusp forms since there are no lifts which fits with the result 
$$\dim S_{3,9,1}(\GaSq)=16,$$
following from Theorem~\ref{dim-Ajkl}.

Using the same method as above, we compute some (conjectural) Hecke eigenvalues for two of the $1$-dimensional spaces in the following table. Note that the eigenform in $S^{\rm gen}_{3,9,1}(\GaSq)^{(1^4)}$ equals $(E_0+E_2+E_2-E_3) \, \zeta$, with $E_i$ the Eisenstein series given in \cite[Lemma 15.1]{C-vdG}.

\begin{footnotesize}
\smallskip
\vbox{
\bigskip\centerline{\def\quad{\hskip 0.6em\relax}
\def\quod{\hskip 0.5em\relax }
\vbox{\offinterlineskip
\hrule
\halign{&\vrule#&\strut\quod\hfil#\quad\cr
height2pt&\omit&&\omit&&\omit&\cr
&$p$ && $S^{\rm gen}_{3,9,1}(\GaSq)^{(2,2)}$  && $S^{\rm gen}_{3,9,1}(\GaSq)^{(1^4)}$ & \cr
\noalign{\hrule}
& $7$   && $-39753 \rho - 15702 $  &&   $-3303\rho-20562  $ & \cr
& $13$   && $-2259729 \rho - 462012$  &&   $39537\rho-662244 $ & \cr
& $19$   && $-813897 \rho - 7616175$  &&   $-12094443\rho-15482085$  & \cr
& $31$   && $62423118 \rho + 189603705$  &&   $13979610\rho-2791545 $ & \cr
& $37$   && $154008855 \rho - 213937620$  &&   $-132007005\rho-420798660$ & \cr
& $43$   && $-1091048814 \rho - 311480763$  &&   $-1442196450\rho-484155105 $ & \cr
} \hrule}
}}
\end{footnotesize}
\end{subsubsection}

\begin{subsubsection}{$\lambda=(9,3,0)$} In the cohomology of this local system we expect to find contributions from the forms in $S_{3,9,2}(\GaSq)$. 

Once again, using the formula in Section~\ref{sec-euler} we get, 
$$E_c(X_{\GaSq},{\WW}_{\lambda})=2 \, \mathbf s_4 + 7 \,\mathbf s_{3, 1} + 3\, \mathbf s_{2^2} + 4\,\mathbf s_{2, 1^2}. $$

From Proposition~\ref{prop-eis} and Section~\ref{sec-extra} we get the non-zero contributions to $e_{\rm extr}(\lambda)$ as  
$$e'_{\mathrm{Eis}}(\lambda)= e_{\mathrm{Eis}}(\lambda)=-\LL^{0,0}(\mathbf s_{2,1^2} + \mathbf s_{1^4}),$$
$$e_{\rm ce}(\lambda)=\LL^{4,7} (\mathbf s_{3,1} + 2\, \mathbf s_{2,1^2}+\mathbf s_{1^4}),$$
and 
$$e_{2\ell}(\lambda)= S^{-}[\Gamma_0(9),12] \mathbf s_{4}.
$$

As in the previous examples, we use the numerical Euler characteristic to get the following conjecture,  
$$\dim_{\fS_4} S^{\rm gen}_{3,9,2}(\GaSq) = 2\,\mathbf s_{3, 1} +  \mathbf s_{2^2} + \mathbf s_{2, 1^2}.$$
Together with the $1$-dimensional space of lifts above we get
$$\dim S_{3,9,2}(\GaSq)=12,$$
which fits with Theorem~\ref{dim-Ajkl}.

Some (conjectural) Hecke eigenvalues for the $1$-dimensional spaces are found in the following table. 

\begin{footnotesize}
\smallskip
\vbox{
\bigskip\centerline{\def\quad{\hskip 0.6em\relax}
\def\quod{\hskip 0.5em\relax }
\vbox{\offinterlineskip
\hrule
\halign{&\vrule#&\strut\quod\hfil#\quad\cr
height2pt&\omit&&\omit&&\omit&\cr
&$p$ && $S^{\rm gen}_{3,9,2}(\GaSq)^{(2,2)}$ && $S^{\rm gen}_{3,9,2}(\GaSq)^{(2,1^2)}$ & \cr
\noalign{\hrule}
& $7$    &&   $-522 \rho - 771 $ && $-42405-73422\rho  $ & \cr
& $13$   &&   $64731 - 1053828 \rho $  && $2150805+1144836\rho $ & \cr
& $19$   &&   $9397530 \rho + 10858953$  && $1117083-2630970\rho $ & \cr
& $31$   &&   $-199487250 \rho - 223012887$  && $17764311-9145350\rho $ & \cr
& $37$  &&   $-283226796 \rho - 170478933$  && $144010695+424906308\rho $  & \cr
& $43$   &&   $456864210 \rho - 855993435$  && $-365663985-1035862434\rho $ & \cr
} \hrule}
}}
\end{footnotesize}
\end{subsubsection}
\begin{subsubsection}{$\lambda=(11,0,-5)$} In the cohomology of this local system we expect to find contributions from the forms in $S_{5,14,2}(\GaSq)$. 

Using the formula in Section~\ref{sec-euler} we get, 
$$E_c(X_{\GaSq},{\WW}_{\lambda})=9 \, \mathbf s_4 + 27 \,\mathbf s_{3, 1} + 18\, \mathbf s_{2^2} + 27\,\mathbf s_{2, 1^2} +9\,\mathbf s_{1^4}. $$

From Proposition~\ref{prop-eis} and Section~\ref{sec-extra} we get the non-zero contributions to $e_{\rm extr}(\lambda)$ as  
$$e'_{\mathrm{Eis}}(\lambda)= e_{\mathrm{Eis}}(\lambda)=(-\LL^{0,0}+\LL^{6,0}+\LL^{0,12})(\mathbf s_{4} + \mathbf s_{3,1}),$$
$$e_{\rm ce}(\lambda)=3\LL^{6,12} (\mathbf s_{4}+\mathbf s_{3,1} + \mathbf s_{2,2}).$$
$$e_{\overline{2\ell}}(\lambda)= \LL^{0,12} S^{\chi}[\Gamma_1(9),7] \mathbf s_{2^2}.
$$
and 
$$e_{2\ell}(\lambda)= \LL^{6,0} S^{-}[\Gamma_1(3),13] (\mathbf s_{4}+\mathbf s_{3,1})+\LL^{6,0} S^{\chi}[\Gamma_1(9),13] \mathbf s_{2^2}.
$$

As in the previous examples, we use the numerical Euler characteristic to get the following conjecture,  
$$\dim_{\fS_4} S^{\rm gen}_{5,14,2}(\GaSq) = \mathbf s_{4} + 7\,\mathbf s_{3, 1} +  3\, \mathbf s_{2^2} + 9 \, \mathbf s_{2, 1^2} + 3 \, \mathbf s_{1^4}.$$ 
Together with the $2$-dimensional space of lifts above we get
$$\dim S_{5,14,2}(\GaSq)=66,$$
which fits with Theorem~\ref{dim-Ajkl}.

Some (conjectural) Hecke eigenvalues for the $1$-dimensional space of genuine cusp forms are found in the following table. 

\begin{footnotesize}
\smallskip
\vbox{
\bigskip\centerline{\def\quad{\hskip 0.6em\relax}
\def\quod{\hskip 0.5em\relax }
\vbox{\offinterlineskip
\hrule
\halign{&\vrule#&\strut\quod\hfil#\quad\cr
height2pt&\omit&&\omit&\cr
&$p$ && $S^{\rm gen}_{5,14,2}(\GaSq)^{(4)}$ & \cr
\noalign{\hrule}
& $7$   && $-38516760 \rho-13589673$ & \cr
& $13$  && $-6017408280\rho -7487727117$ & \cr
& $19$  && $546522935760\rho+368972351247$ & \cr
& $31$  && $20336092789320\rho+55796255768703 $ & \cr
& $37$  && $ -147394045113480\rho + 55302806453187$ & \cr
& $43$  && $134094712536720\rho -23648747132697 $ & \cr
} \hrule}
}}
\end{footnotesize}

\end{subsubsection}
\end{subsection}
\begin{subsection}{One-dimensional spaces of genuine forms} \label{sec-onedim}
The cases for which $$\dim S_{j,k,l}^{\rm gen}(\Gamma_1[\sqrt{-3}])^{\mu}=1$$ are of special importance to us since in these cases counts of points over finite fields gives (using Conjecture~\ref{conj-main}) Hecke eigenvalues rather than just traces.

For $\mu=(4)$ we found 78 such cases using Conjecture~\ref{conj-dim}. We list all such $(j,k,l)$ below, but because of Remark~\ref{rmk-dual} we only list them up to duality.  

$$
\begin{matrix}
(0, 15, 1) & (0, 21, 1) & (0, 24, 1) & (0, 27, 0) & (0, 30, 1) & (0, 30, 2) & (0, 33, 0) \\ 
(0, 36, 0) & (0, 36, 2) & (0, 39, 2) & (0, 42, 0) & (0, 45, 2) & (1, 16, 0) & (1, 19, 0) \\
(1, 19, 1) & (1, 19, 2) & (1, 22, 1) & (1, 25, 1) & (1, 28, 2) & (2,
11, 0) & (2, 11, 1) \\ 
(2, 14, 0) & (2, 14, 1) & (2, 20, 2) & (2, 23, 2) & (3, 9, 1) & (3, 12, 0) & (3, 12, 1) \\
(3, 15, 0) & (3, 18, 2) & (4, 7, 0) & (4, 10, 0) & (4, 10, 2) & (4, 13, 2) & (5, 8, 0) \\ 
(5, 11, 2) & (5, 14, 2) & (6, 9, 0) & (6, 9, 2) & (7, 10, 2) &&
\end{matrix}
$$

For $\mu=(3,1), (2^2), (2,1^2), (1^4)$ we found $35, 44, 35, 76$ cases respectively. 
\end{subsection}
\end{section}

\begin{section}{Congruences of Harder type} \label{sec-harder}
According to Harder a prime appearing in the denominator
of a critical value of an $L$-function sometimes leads to a congruence
between modular forms, see \cite{Harder-1-2-3}. The shape of these
congruences was discussed by Harder after we found instances of
congruences, see \cite{Harder-shape} and see also Dummigan's discussion in \cite{Dummigan}.

In the case at hand we look at the standard $L$-function
associated to an algebraic Hecke character $\psi_m$ with the following
Euler factors. For a prime $p \equiv_3 1$ with $p=\nu_p \bar{\nu}_p$
and $\nu_p\equiv \bar{\nu}_p \equiv_3 1$ we have
$$
L_p(\psi_m,s)=1/(1-\nu_p^m p^{-s})(1-\bar{\nu}_p^m p^{-s})
$$
and for a prime $p \equiv_3 2$ we have
$$
L_p(\psi_m,s)=1/(1-(-p)^m p^{-s})\, ,
$$
while for $p=3$ we have $L_3(\psi_m,s)=1$ unless $m\equiv_6 0$ and then
$L_3(\psi_m,s)= 1/(1-(\sqrt{-3})^m 3^{-s})$. The completed $L$-function
$$
\Lambda(\psi_m,s)= \frac{\Gamma(s)}{(2\pi)^s} \prod L_p(\psi_m,s)
$$
extends to a holomorphic function of $s$ and satisfies a functional equation
relating $s$ with $m+1-s$. According to (an analogue of) a result of Hurwitz
we get rational quotients of critical values
$$
Q(m,n):= \frac{\Lambda(\psi_m,n-1)}{\Lambda(\psi_m,n)} \in {\QQ} \qquad
\text{for $n=m, m-1,\ldots, [\frac{m+1}{2}]$}\, .
$$

\begin{conjecture} {\rm (Harder's conjecture)} If a prime 
$\ell>m$ divides the denominator of $Q(m,n)$ there exists 
a Picard modular cusp form of weight $(b,a+3)=(m-n,2n-m+1)$,
which is an eigenform of the Hecke algebra, such that its Hecke eigenvalues
$\lambda_{\nu_p}$ for $p\equiv_3 1$ satisfy the congruence
$$
\lambda_{\nu_p}\equiv_{\ell} \bar{\nu}_p^{a+b+2}+(p^{a+1}+1){\nu}_p^{b+1} \, .
$$
\end{conjecture}

We found the following cases where the data available to us are in accordance with this conjecture. If the index of the local system ${\WW}_{\lambda}$
is $\lambda=(a+i,i,-b+i)$, we list the weight $(j,k,l)=(b,a+3,i+2)$
of the modular forms,
the representation of $\mathfrak{S}_4$, the index $(m,n)=(a+2b+3,a+b+3)$, the congruence prime $\ell$ 
and the value of $Q(m,n)$.

\begin{footnotesize}
\bigskip
\vbox{
\centerline{\def\quad{\hskip 0.3em\relax}
\vbox{\offinterlineskip
\hrule
\halign{&\vrule#& \quad \hfil#\hfil \strut \quad  \cr
height2pt &\omit&&\omit&&\omit&&\omit&&\omit& \cr
& $j,k,l$ && $\mu$ && $(m,n)$ && $\ell$ && $Q(m,n)$ & \cr
height2pt&\omit&&\omit&&\omit&&\omit&&\omit & \cr
\noalign{\hrule}
height2pt&\omit&&\omit&&\omit&&\omit&&\omit & \cr
& $2,11,2$ && $(1^4)$ && $(15,13)$ && $53$ && $2^4/3\cdot 53$ & \cr
& $1,13,1$ && $(1^4)$ && $(15,14)$ && $19$ && $53/2\cdot 5 \cdot 19$ & \cr
& $1,13,1$ && $(2,1^2)$ && $(15,13)$ && $19$ && $53/2 \cdot 5 \cdot 19$ & \cr
& $6,9,0$ && $(1^4)$ && $(21,15)$ && $271$ && $233/2\cdot 5 \cdot 271 $ & \cr
& $3,12,0$ && $(4)$ && $(18,15)$ && $29$ && $3 \cdot 5/2 \cdot 29$ & \cr
& $5,11,2$ && $(1^4)$ && $(21,16)$ && $17$ && $271/2\cdot 3\cdot 11 \cdot 17$ & \cr
& $2,20,2$ && $(4)$ && $(24,22)$ && $97$ && $11\cdot 457/2^2 \cdot 3\cdot 5^2 \cdot 97$ & \cr
& $1,22,1$ && $(4)$ && $(24,23)$ && $41$ && $23\cdot 97 /2 \cdot 3\cdot 5 \cdot 11 \cdot 41$ & \cr
& $0,27,0$ && $(1^4)$ && $(27,27)$ && $449$ && $3^2\cdot 179 \cdot 223/2^2\cdot 11 \cdot 17 \cdot 23 \cdot 449$ & \cr
& $0,33,0$ && $(1^4)$ && $(33,33)$ && $17093$ && $19\cdot 84802789/2^2\cdot 5 \cdot 11 \cdot 17 \cdot 23 \cdot 29 \cdot 17093$ & \cr
height2pt&\omit&&\omit&&\omit&&\omit&&\omit& \cr
} \hrule}
}}

\end{footnotesize}

\noindent
The Hecke eigenvalues in two of these cases are found in the following table.

\begin{footnotesize}
\smallskip
\vbox{
\bigskip\centerline{\def\quad{\hskip 0.6em\relax}
\def\quod{\hskip 0.5em\relax }
\vbox{\offinterlineskip
\hrule
\halign{&\vrule#&\strut\quod\hfil#\quad\cr
height2pt&\omit&&\omit&&\omit&\cr
&$p$ && $S_{2,11,2}^{\rm gen}(\GaSq)^{(1^4)}$  && $S_{6,9,0}^{\rm gen}(\GaSq)^{(1^4)}$ & \cr
\noalign{\hrule}
& $7$   && $113760\rho+180273$ &&  $-742581\rho-967245 $ & \cr
& $13$   && $6574680\rho+4136763 $ &&  $ -11444355\rho+37295661 $ & \cr
& $19$   && $-3105720\rho+22527309$ &&  $-1411116471\rho-1183781976$ & \cr
& $31$   && $1128613680\rho-206255175$ &&  $ 2162847960\rho+20439895125$ & \cr
& $37$   && $-1059546600\rho-631344705$ &&  $113910723225\rho+29288724825$ & \cr
& $43$   && $-3998935080\rho-6398875995$ &&  $-55912815000\rho-92116884255$ & \cr
} \hrule}
}}
\end{footnotesize}
\noindent

In fact, we searched for congruences in our heuristic data (for the cases where the space
$S^{\mathrm{gen}}_{j,k,l}(\GaSq)^{\mu}$ has dimension $1$) 
and then checked the value of the
corresponding quotient of critical $L$-values. In all cases except one
the congruence prime showed up in $Q(m,n)$. The one extra congruence
not explained by the above conjecture occurs for the local system ${\WW}_{\lambda}$
with $\lambda=(16,1,1)$ and $\mu=(3,1)$. We found a congruence
modulo $\ell=37$. But the corresponding $Q(18,18)=3\cdot 7 \cdot 19/2 \cdot 5 \cdot 11\cdot 17$
does not show $37$. Harder thinks that this congruence might be due to the second factor
$c(\phi,0)$ in \cite[page 590]{Ha1}. We list some eigenvalues for the case $\lambda=(16,1,1)$.

\begin{footnotesize}
\smallskip
\vbox{
\bigskip\centerline{\def\quad{\hskip 0.6em\relax}
\def\quod{\hskip 0.5em\relax }
\vbox{\offinterlineskip
\hrule
\halign{&\vrule#&\strut\quod\hfil#\quad\cr
height2pt&\omit&&\omit&\cr
&$p$ && $S_{0,18,0}(\GaSq)^{(3,1)}$& \cr
\noalign{\hrule}
& $7$   && $-37133403-19436265\rho$ & \cr
& $13$   && $-114953793-826184565\rho$ & \cr
& $19$   && $82348187646+48917648907\rho$ & \cr
& $31$   && $2339550247917-489600934794\rho$ & \cr
& $37$   && $6061060465185+ 27008238932829\rho$ & \cr
& $43$   && $-13426382809671-41363330321286\rho$ & \cr
} \hrule}
}}
\end{footnotesize}

\end{section}


\begin{thebibliography}{9999}

\bibitem{Achter-Pries} J.D. \ Achter, R.\ Pries: The integral monodromy of hyperelliptic and trielliptic curves.
Math. Ann. {\bf 338} (2007), no. 1, 18--206. 

\bibitem{Asai} T. \ Asai: On the Fourier coefficients of automorphic forms at various cusps and some applications to Rankin's convolution.
J. Math. Soc. Japan {\bf 28} (1976), no. 1, 48--61. 

\bibitem{Atiyah} M.F.\ Atiyah: Vector bundles over an elliptic curve.
Proc.\ London Math.\ Soc.\ {\bf 7} (1957), 414--452.

\bibitem{Ber} J. Bergstr\"om: 
Cohomology of moduli spaces of curves of genus three via point counts. 
J. Reine Angew. Math. {\bf 622} (2008), 155--187. 

\bibitem{BFvdG} J. Bergstr\"om, C. Faber, and G. van der Geer:
Siegel modular forms of degree three and the cohomology of local systems.
Selecta Math. {\bf 20} (2014), no. 1, 83--124. 
 
\bibitem{BvdG} J. Bergstr\"om, and G. van der Geer:
The Euler characteristic of local systems on the moduli of curves and abelian varieties of genus three. 
J. Topol. {\bf 1} (2008), no. 3, 651--662. 

\bibitem{website} J. Bergstr\"om, Fabien Cl\'ery, 
Carel Faber, and G. van der Geer:
Siegel Modular Forms of Degree Two and Three, 2017. \\ http://smf.compositio.nl.

\bibitem{CFG} F.\ Cl\'ery, C. Faber, and G. van der Geer:
Concomitants of ternary quartics and vector-valued Siegel and Teichm\"uller modular forms of genus three. Selecta Math.\ (2020) 

\bibitem{C-vdG} F.\ Cl\'ery, G.\ van der Geer: Generators for modules
of vector-valued Picard modular forms.
Nagoya Math. J.\ Volume {\bf 212}, (2013), 19--57.

\bibitem{Cox} D.\ Cox: Stickelberger and the eigenvalue theorem.
{\tt arXiv:2007.12573v1}.

\bibitem{Dalen}
K.\ Dalen: 
On a theorem of Stickelberger. Math.\ Scand.\ \textbf{3} (1955), 124--126.

\bibitem{Fermat}
H.\ Darmon, F.\ Diamond, and R.\ Taylor:
{\sl Fermat's last theorem.}
In: {Current developments in mathematics}, 1--154,
Int. Press, Cambridge, MA, 1994.

\bibitem{SGA} P. \ Deligne: Rapport sur la formule de trace. In: Cohomologie \'etale, SGA $4\tfrac12$, Lecture Notes in Math. \textbf{569}. Springer-Verlag, Berlin-New York, (1977).

\bibitem{Dummigan}
N.\ Dummigan: Eisenstein congruences for unitary groups, \\
{http://neil-dummigan.staff.shef.ac.uk/papers.html}

\bibitem{FvdG} C.\ Faber, G.\ van der Geer:
Sur la cohomologie des syst\`{e}mes locaux sur les espaces des modules des courbes de genre $2$ et des surfaces ab\'{e}liennes,
I, II. C.R.\ Acad.\ Sci.\ Paris, S\'er.\ I \textbf{338} (2004), 381--384, 467--470.

\bibitem{F-C} G.\ Faltings, C-L.\ Chai: Degeneration of abelian
varieties. Ergebnisse der Math.\ 22. 

\bibitem{Fe} J.M.\ Feustel: Ringe automorpher Formen auf der komplexen 
Einheitskugel und ihre Erzeugung durch Theta-konstanten. Preprint Ser.\ 
Akad.\ Wiss.\ DDR, P-Math-13, (1986).

\bibitem{Fi} T.\ Finis:
Some computational results on Hecke eigenvalues of modular forms on a 
unitary group.  Manuscripta Math.\  {\bf 96} (1998), 149--180.

\bibitem{vdG1} G.\ van der Geer: Rank one Eisenstein cohomology
  of local systems on the moduli space of abelian varieties.
Sci. China Math.\ {\bf 54} (2011), no. 8, 1621--1634. 

\bibitem{Gi} J.\ Giraud: Remarque sur une formule de Shimura-Taniyama. 
 Invent.\ Math.\ \textbf{5} (1968), 231--236.

\bibitem{Go} B.\ Gordon: Canonical models of Picard modular surfaces.
In: {\sl The zeta function of Picard modular surfaces}.
 Edited by Robert P. Langlands and Dinakar Ramakrishnan. Universit\'e de Montr\'eal, Centre de Recherches Math\'ematiques,
Montreal, QC, (1992), 1--30.

\bibitem{Ha1} G.\ Harder: Eisensteinkohomologie f\"ur Gruppen vom Typ $GU(2,1)$
Math.\ Annalen {\bf 278} (1987), 563--592.

\bibitem{Ha2} G.\ Harder: Eisensteinkohomologie und die Konstruktion gemischter Motive. Lecture Notes in Mathematics 1562. Springer Verlag.

\bibitem{Harder-1-2-3} G.\ Harder: A congruence between a Siegel 
and an elliptic modular form. In: The 1-2-3 of Modular Forms.
Springer Verlag, Berlin, (2008).

\bibitem{Harder-shape} G.\ Harder:
The shape of the congruences for $U(2,1)$.
Unpublished manuscript, 8 p., November 4, 2013.

\bibitem{H-M} J.\ Harris, I.\ Morrison: Moduli of Algebraic curves.
Graduate Texts in Mathematics, 187. Springer-Verlag, New York, (1998). 

\bibitem{Ho1} R.P.\ Holzapfel: Around Euler Partial Differential Equations.
VEB Deutscher Verlag der Wissenschaften, Berlin, (1986).

\bibitem{Ho2} R.P.\ Holzapfel:
The ball and some Hilbert problems. 
Lectures in Mathematics ETH Z\"urich. Birkh\"auser Verlag, Basel, (1995).

\bibitem{IR} K. \ Ireland, M. \ Rosen:
A classical introduction to modern number theory.
Graduate Texts in Mathematics, 84. Springer-Verlag, New York, (1990). 

\bibitem{Keranen} J.\ Keranen: Compact support cohomology of Picard modular
surfaces. UCLA Ph.D. thesis 2015. \\
https://escholarship.org/uc/item/6dd2h1zb 

\bibitem{L-R} The zeta functions of Picard modular surfaces. 
Edited by Robert P. Langlands and Dinakar Ramakrishnan. Universit\'e de Montr\'eal, Centre de Recherches Math\'ematiques, Montreal, QC, (1992).

\bibitem{Kudla1} S.\ Kudla: 
On Certain Arithmetic Automorphic Forms for S U (1, q ). 
Invent.\ Math.\ \textbf{ 52} (1979), 1--25.

\bibitem{Lan} K.-W. \ Lan: Arithmetic compactifications of PEL-type Shimura varieties.
London Mathematical Society Monographs Series, 36. Princeton University Press, Princeton, NJ, (2013). 

\bibitem{Larsen} M.\ Larsen: Arithmetic compactification of some Shimura surfaces.
In: {\sl The zeta functions of Picard modular surfaces}.
Edited by Robert P. Langlands and Dinakar Ramakrishnan. Universit\'e de Montr\'eal, Centre de Recherches Math\'ematiques,
Montreal, QC, (1992), 31--45.

\bibitem{L-S} J.-S. Li, J.\ Schwermer: On the Eisenstein cohomology of
arithmetic groups. Duke Math.\ J.\ {\bf 123} (2004), 141--169.

\bibitem{Mi} Y.\ Miyaoka: Stable Higgs bundles with trivial Chern classes. Several examples. Proc. Steklov Inst. Math. {\bf 264} (2009), no. 1, 123--130.

\bibitem{MSt-Y-Z} 
S.\  M\"uller-Stach, X.\ Ye, K.\ Zuo:
Mixed Hodge complexes and $L^2$-cohomology for local systems 
on ball quotients. 
Doc.\ Math.\  {\bf 17} (2012), 517--543. 

\bibitem{Mu1}  D.\ Mumford: Hirzebruch's Proportionality Theorem 
in the non-compact case. Invent.\ Math.\ {\bf 42} (1977), 239--272.

\bibitem{P1}  E.\ Picard: 
Sur une classe de groupes discontinus de substitutions lin\'eaires et
sur les fonctions de deux variables ind\'ependantes restant 
invariables par ces substitutions. 
Acta Math. \textbf{1} (1882), 297--320.

\bibitem{P2} E.\ Picard: 
Sur des fonctions de deux variables ind\'ependantes analogues aux
fonctions modulaires. Acta Math.\ \textbf{2} (1883), 114--135.

\bibitem{P3} E.\ Picard: 
Sur les formes quadratiques ternaires ind\'efinies \`a ind\'etermin\'ees conjugu\'ees et sur les fonctions hyperfuchsiennes correspondantes. 
Acta Math.\ \textbf{5} (1884), 121--182.

\bibitem{Rag} M.S.\ Raghunathan:
On the first cohomology of discrete subgroups of semisimple Lie groups. 
Amer.\ J.\ Math.\ \textbf{87} (1965), 103--139. 

\bibitem{Sap} L.\ Saper: $L$-modules and a conjecture of Rapoport and Goresky-MacPherson. Automorphic Forms, I, Ast\'erisque {\bf 298} (2005), 319--334.

\bibitem{Scholl} A.J.\ Scholl: 
Motives for modular forms. Invent.\ Math.
\textbf{100}  (1990),  no. 2, 419--430.

\bibitem{Shimura} G.\ Shimura: On purely transcendental fields of
automorphic functions of several variables. {\sl Osaka Math.\ Journal \bf 1}
(1964), pp.\ 1--14.

\bibitem{Shimura-arithmetic} G.\ Shimura: 
The arithmetic of automorphic forms with respect to a unitary group.
{\sl Ann. of Math. {\bf(2) 107}} (1978) pp.\ 569--605. 

\bibitem{Shintani} T.\ Shintani: On automorphic forms on unitary groups of order $3$. Unpublished manuscript. 

\bibitem{Zink} T.~Zink: \"Uber die Anzahl der Spitzen einiger 
arithmetischer Untergruppen unit\"arer Gruppen. 
{\sl Math.\ Nachr.\ \bf 89} (1979), 315--320.
 
\end{thebibliography}
\end{document}